\documentclass[11pt,english]{amsart}
\usepackage[T1]{fontenc}
\usepackage[latin9]{inputenc}
\usepackage{babel}
\usepackage{amsbsy}
\usepackage{amstext}
\usepackage{amsthm}
\usepackage{amssymb}
\usepackage{stmaryrd}
\usepackage{graphicx}
\usepackage[letterpaper]{geometry}
\usepackage[pdfusetitle,
 bookmarks=true,bookmarksnumbered=false,bookmarksopen=false,
 breaklinks=false,pdfborder={0 0 1},backref=false,colorlinks=false]
 {hyperref}
\hypersetup{
 colorlinks=true,linkcolor=teal,citecolor=purple,linktocpage=true}

\makeatletter
\numberwithin{equation}{section}
\numberwithin{figure}{section}
\theoremstyle{plain}
\newtheorem{thm}{\protect\theoremname}[section]
\theoremstyle{remark}
\newtheorem{rem}[thm]{\protect\remarkname}
\theoremstyle{plain}
\newtheorem{question}[thm]{\protect\questionname}
\theoremstyle{definition}
\newtheorem{convention}[thm]{Convention}
\theoremstyle{definition}
\newtheorem{defn}[thm]{\protect\definitionname}
\theoremstyle{plain}
\newtheorem{lem}[thm]{\protect\lemmaname}
\theoremstyle{plain}
\newtheorem{prop}[thm]{\protect\propositionname}
\theoremstyle{plain}
\newtheorem{cor}[thm]{\protect\corollaryname}
\theoremstyle{definition}
\newtheorem{example}[thm]{\protect\examplename}
\theoremstyle{remark}
\newtheorem{claim}[thm]{\protect\claimname}
\theoremstyle{definition}
\newtheorem{condition}[thm]{\protect\conditionname}

\@ifundefined{date}{}{\date{}}
\usepackage{amsmath, tikz-cd, amssymb}
\DeclareRobustCommand*\cal{\@fontswitch\relax\mathcal}

\makeatother

\providecommand{\claimname}{Claim}
\providecommand{\conditionname}{Condition}
\providecommand{\corollaryname}{Corollary}
\providecommand{\definitionname}{Definition}
\providecommand{\examplename}{Example}
\providecommand{\lemmaname}{Lemma}
\providecommand{\propositionname}{Proposition}
\providecommand{\questionname}{Question}
\providecommand{\remarkname}{Remark}
\providecommand{\theoremname}{Theorem}

\begin{document}
\title{Rational surgery exact triangles in Heegaard Floer homology}
\author{Gheehyun Nahm}
\thanks{The author was partially supported by the ILJU Academy and Culture
Foundation, the Simons collaboration \emph{New structures in low-dimensional
topology}, and a Princeton Centennial Fellowship.}
\address{Department of Mathematics, Princeton University, Princeton, New Jersey
08544, USA}
\email{gn4470@math.princeton.edu}
\begin{abstract}
We construct a new family of surgery exact triangles in Heegaard Floer
theory over the field with two elements. This family generalizes both
Ozsv\'{a}th and Szab\'{o}'s $n$- and $1/n$-surgery exact triangles
for positive integers $n$ and the author's recent $2$-surgery exact
triangle to all positive rational slopes.

The construction reduces to a combinatorial problem that involves
triangle and quadrilateral counting maps in a genus $1$ Heegaard
diagram. The main contribution of this paper is solving this combinatorial
problem, which is particularly tricky for slopes $r\neq n,1/n$; one
key idea is to use an involution that is closely related to the ${\rm Spin}^{c}$
conjugation symmetry.
\end{abstract}

\maketitle
\tableofcontents{}

\section{Introduction}

Heegaard Floer homology is a powerful invariant of closed, oriented
$3$-manifolds defined by Ozsv\'{a}th and Szab\'{o} \cite{MR2113019}.
An important property of Heegaard Floer homology is that it admits
surgery exact triangles: in \cite[Section 9]{MR2113020}, Ozsv\'{a}th
and Szab\'{o} constructed surgery exact triangles for slopes $r=n$
or $1/n$ for positive integers $n$, that involve the Heegaard Floer
homology groups of the $0$-surgery of a knot, the $r$-surgery of
a knot, and the underlying $3$\nobreakdash-manifold; compare \cite{floer1988instanton,MR2299739}.

There is a closely related invariant for knots, called knot Floer
homology, defined by Ozsv\'{a}th and Szab\'{o} \cite{MR2065507}
and independently by Rasmussen \cite{MR2704683}. Ozsv\'{a}th, Stipsicz,
and Szab\'{o} \cite{1407.1795} defined a variant of knot Floer homology
called $t$-modified knot Floer homology and used it to define an
infinite family of knot concordance invariants. In \cite{MR3694597},
they used the special case $t=1$ and called it unoriented knot Floer
homology, to give a lower bound for the smooth $4$-dimensional crosscap
number.

Recently, the author \cite[Theorem~1.3]{nahm2025unorientedskeinexacttriangle}
constructed a new surgery exact triangle for slope~$2$. This surgery
exact triangle involves the Heegaard Floer homology groups of the
$0$-surgery of a knot, the $2$-surgery of a knot, and the unoriented
knot Floer homology of the knot in the underlying $3$-manifold; this
is a Heegaard Floer analogue of Bhat's recent $2$-surgery exact triangle
\cite{2311.04242} in instanton Floer homology.

In this paper, we construct a new family of surgery exact triangles
in Heegaard Floer theory over the field $\mathbb{F}=\mathbb{Z}/2\mathbb{Z}$
with two elements (Theorem~\ref{thm:rational-surgery-kgeneral}).
For each triple $(p,q,k)$ where $p$ and $q$ are coprime positive
integers and $k=0,\cdots,p-1$, we construct a surgery exact triangle
for slope $p/q$. This family contains Ozsv\'{a}th and Szab\'{o}'s
$n$- and $1/n$-surgery exact triangles and the author's $2$-surgery
exact triangle: they correspond to $(p,q,k)=(n,1,0)$, $(1,n,0)$,
and $(2,1,1)$.

These surgery exact triangles for $k=0$ only involve Heegaard Floer
homology groups, and they directly generalize Ozsv\'{a}th and Szab\'{o}'s
$n$- and $1/n$-surgery exact triangles to all positive rational
slopes $p/q$. Note that even for $k=0$, constructing these surgery
exact triangles is particularly tricky for general $(p,q)\neq(n,1),(1,n)$:
see Subsubsection~\ref{subsec:intro-k-0} and Remark~\ref{rem:pq53}
for further discussion.

For general $(p,q,k)$, two out of the three homology groups in the
surgery exact triangle are Heegaard Floer homology groups, but the
third group is a variant $\boldsymbol{HFK}_{p,q,k}^{-}(Y,K)$ of knot
Floer homology that we define in Definition~\ref{def:hfkpqk}. This
variant of knot Floer homology should be thought of as a version of
$2k/p$\nobreakdash-modified knot Floer homology: see Subsection~\ref{subsec:The-modified-knot}
for further discussion.

For clarity, we first state the main theorem for $k=0$ separately
in Theorem~\ref{thm:rational-surgery-k=00003D0}. For additional
clarity, Theorem~\ref{thm:rational-surgery-k=00003D0} is divided
into two cases, $p\ge q$ and $p<q$. In Subsection~\ref{subsec:530},
we restate Theorem~\ref{thm:rational-surgery-k=00003D0} for the
special case where $p\ge q$, $Y=S^{3}$, and the framing $\lambda$
is the Seifert framing (Corollary~\ref{cor:seifert}), and we describe
this surgery exact triangle for the unknot and the right handed trefoil
for $(p,q)=(5,3)$.
\begin{thm}[Main theorem, special case $k=0$]
\label{thm:rational-surgery-k=00003D0}Let $K$ be a knot in a closed,
oriented $3$\nobreakdash-manifold~$Y$ with framing $\lambda$,
and let $\mu$ be the meridian. Then, for all rational $p/q>0$ where
$p$ and $q$ are coprime positive integers, there is an $\mathbb{F}\llbracket U\rrbracket$-linear
exact triangle that involves the Heegaard Floer homology groups of
the $\lambda$-surgery $Y_{\lambda}(K)$, the $p\mu+q\lambda$-surgery
$Y_{p\mu+q\lambda}(K)$, and the underlying manifold $Y$. More precisely:
\begin{enumerate}
\item If $p\ge q$, then the exact triangle involves (A) $q$ copies of
$\boldsymbol{HF}^{-}(Y_{\lambda}(K))$, (B) $\boldsymbol{HF}^{-}(Y_{p\mu+q\lambda}(K))$,
and (C) the Heegaard Floer homology group with twisted coefficients
$\underline{\boldsymbol{HF}^{-}}(Y;\mathbb{F}[\mathbb{Z}/p\mathbb{Z}])$
where the twist is given by the homology class $[K]\in H_{1}(Y)$
of the knot, i.e.\ there is an $\mathbb{F}\llbracket U\rrbracket$-linear
exact triangle 
\[\begin{tikzcd}[ampersand replacement=\&]
	{\bigoplus_{j=0}^{q-1} \boldsymbol{HF}^{-}(Y_{\lambda}(K))} \&\& {\boldsymbol{HF}^-(Y_{p\mu+q\lambda}(K))} \\
	\& {\underline{\boldsymbol{HF}^{-}}(Y;\mathbb{F}[\mathbb{Z}/p\mathbb{Z}])}
	\arrow[from=1-1, to=1-3]
	\arrow[from=1-3, to=2-2]
	\arrow[from=2-2, to=1-1]
\end{tikzcd}\]
\item \label{enu:p<q}If $p<q$, then the exact triangle involves (A) a
direct sum of the Heegaard Floer homology groups with twisted coefficients
$\underline{\boldsymbol{HF}^{-}}(Y_{\lambda}(K);\mathbb{F}[\mathbb{Z}/s_{i}\mathbb{Z}])$
for $i=0,\cdots,p-1$ where the twist is given by the homology class
$[K_{\lambda}]\in H_{1}(Y_{\lambda}(K))$ of the dual knot and 
\begin{equation}
s_{i}=\begin{cases}
\left\lceil q/p\right\rceil  & i=0,\cdots,(q\bmod p)-1\\
\text{\ensuremath{\left\lfloor q/p\right\rfloor }} & i=(q\bmod p),\cdots,p-1
\end{cases}\label{eq:si}
\end{equation}
where $q\bmod p\in\{0,\cdots,p-1\}$ is the residue of $q$ modulo
$p$, (B)~$\boldsymbol{HF}^{-}(Y_{p\mu+q\lambda}(K))$, and (C) the
Heegaard Floer homology group with twisted coefficients $\underline{\boldsymbol{HF}^{-}}(Y;\mathbb{F}[\mathbb{Z}/p\mathbb{Z}])$
where the twist is given by the homology class $[K]\in H_{1}(Y)$
of the knot, i.e.\ there is an $\mathbb{F}\llbracket U\rrbracket$-linear
exact triangle 
\[\begin{tikzcd}[ampersand replacement=\&]
	{\bigoplus_{i=0}^{p-1}\underline{\boldsymbol{HF}^{-}}(Y_{\lambda}(K);\mathbb{F}[\mathbb{Z}/s_{i}\mathbb{Z}])} \&\& {\boldsymbol{HF}^-(Y_{p\mu+q\lambda}(K))} \\
	\& {\underline{\boldsymbol{HF}^{-}}(Y;\mathbb{F}[\mathbb{Z}/p\mathbb{Z}])}
	\arrow[from=1-1, to=1-3]
	\arrow[from=1-3, to=2-2]
	\arrow[from=2-2, to=1-1]
\end{tikzcd}\]
\end{enumerate}
\end{thm}

\begin{rem}
If $K\subset Y$ is nullhomologous, then $\underline{\boldsymbol{HF}^{-}}(Y;\mathbb{F}[\mathbb{Z}/p\mathbb{Z}])$
is simply $p$ copies of $\boldsymbol{HF}^{-}(Y)$.
\end{rem}

The following theorem is our main theorem, which is a generalization
of Theorem~\ref{thm:rational-surgery-k=00003D0}. In Subsection~\ref{subsec:5312},
we describe this surgery exact triangle for the unknot and the right
handed trefoil with the Seifert framing for $(p,q,k)=(5,3,1)$ and
$(5,3,2)$.
\begin{thm}[Main theorem, general case]
\label{thm:rational-surgery-kgeneral}Let $K$ be a knot in a closed,
oriented $3$\nobreakdash-manifold~$Y$ with framing $\lambda$,
and let $\mu$ be the meridian. Then, for all rational $p/q>0$ where
$p$ and $q$ are coprime positive integers and $k=0,\cdots,p-1$,
there is an $\mathbb{F}\llbracket U\rrbracket$-linear exact triangle
\[\begin{tikzcd}[ampersand replacement=\&]
	{\bigoplus_{i = 0} ^ {\min(p,q) - 1}\underline{\boldsymbol{HF}^{-}}(Y_{\lambda}(K);\mathbb{F}[\mathbb{Z}/s_{i}\mathbb{Z}])} \&\& {\boldsymbol{HF}^-(Y_{p\mu+q\lambda}(K))} \\
	\& {\boldsymbol{HFK}_{p,q,k}^{-}(Y,K)}
	\arrow[from=1-1, to=1-3]
	\arrow[from=1-3, to=2-2]
	\arrow[from=2-2, to=1-1]
\end{tikzcd}\]Here, $s_{0},\cdots,s_{p-1}$ are given by Equation~(\ref{eq:si})
and $\underline{\boldsymbol{HF}^{-}}(Y_{\lambda}(K);\mathbb{F}[\mathbb{Z}/s_{i}\mathbb{Z}])$
is the Heegaard Floer homology group with twisted coefficients where
the twist is given by the homology class $[K_{\lambda}]\in H_{1}(Y_{\lambda}(K))$
of the dual knot, and $\boldsymbol{HFK}_{p,q,k}^{-}(Y,K)$ is the
\emph{modified knot Floer homology group} that we define in Definition~\ref{def:hfkpqk}.
\end{thm}

\begin{rem}[Theorems~\ref{thm:rational-surgery-k=00003D0}~and~\ref{thm:rational-surgery-kgeneral}
for the hat, plus, and infinity versions]
\label{rem:hpi}The chain complexes that give rise to the homology
groups in our exact triangles (Theorems~\ref{thm:rational-surgery-k=00003D0}~and~\ref{thm:rational-surgery-kgeneral})
are $\mathbb{F}\llbracket U\rrbracket$-modules. Hence, we define
the hat, plus, and infinity versions in the standard way, i.e.\ by
the following algebraic modification: if $\boldsymbol{C}^{-}$ is
a chain complex of $\mathbb{F}\llbracket U\rrbracket$-modules, define
\[
\widehat{C}:=\boldsymbol{C}^{-}\otimes_{\mathbb{F}\llbracket U\rrbracket}\mathbb{F}\llbracket U\rrbracket/U,\ C^{+}:=\boldsymbol{C}^{-}\otimes_{\mathbb{F}\llbracket U\rrbracket}(U^{-1}\mathbb{F}\llbracket U\rrbracket/\mathbb{F}\llbracket U\rrbracket),\ C^{\infty}:=\boldsymbol{C}^{-}\otimes_{\mathbb{F}\llbracket U\rrbracket}U^{-1}\mathbb{F}\llbracket U\rrbracket.
\]
Our main theorems, Theorems~\ref{thm:rational-surgery-k=00003D0}~and~\ref{thm:rational-surgery-kgeneral},
hold in the hat, plus, and infinity versions as well (see Appendix~\ref{subsec:Proof-of-Theorem}).
\end{rem}

Since the author's $2$-surgery exact triangle \cite[Theorem~1.3]{nahm2025unorientedskeinexacttriangle}
(Theorem~\ref{thm:rational-surgery-kgeneral} for $(p,q,k)=(2,1,1)$)
is a Heegaard Floer analogue of Bhat's $2$-surgery exact triangle
\cite{2311.04242} in instanton Floer homology, we ask the following
question.
\begin{question}
Are there analogous surgery exact triangles to Theorem~\ref{thm:rational-surgery-kgeneral}
in instanton Floer homology? What is the instanton Floer theoretic
analogue of $\boldsymbol{HFK}_{p,q,k}^{-}(Y,K)$?
\end{question}

\subsection{\label{subsec:Idea-of-the}Idea of the proof}

To construct our rational surgery exact triangles (Theorems~\ref{thm:rational-surgery-k=00003D0}~and~\ref{thm:rational-surgery-kgeneral}),
we use the strategy of \cite{MR2141852} (compare \cite{MR2350128,2308.15658,nahm2025unorientedskeinexacttriangle}),
of using the triangle detection lemma \cite[Lemma~4.2]{MR2141852}
(Lemma~\ref{lem:triangle-det}) and reducing to a local computation
in a genus $1$ Heegaard diagram, and hence to a purely combinatorial
problem (Theorem~\ref{thm:gen-local-comp}): we recall this in Subsection~\ref{subsec:inter-local}
and Appendix~\ref{sec:Reduction-of-Theorem}. The main contribution
of this paper is solving this purely combinatorial problem, which
is particularly tricky for general $p/q\neq n,1/n$; we discuss this
below and in Remark~\ref{rem:pq53}.

\begin{figure}[h]
\begin{centering}
\includegraphics{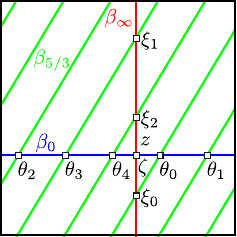}
\par\end{centering}
\caption{\label{fig:diagramt2-infto}A genus $1$ Heegaard diagram $(\mathbb{T}^{2},\beta_{0},\beta_{r},\beta_{\infty},z)$
for $r=5/3$. The intersection points $\{\theta_{0},\theta_{1},\theta_{2},\theta_{3},\theta_{4}\}=\beta_{0}\cap\beta_{r}$,
$\{\xi_{0},\xi_{1},\xi_{2}\}=\beta_{r}\cap\beta_{\infty}$, and $\{\zeta\}=\beta_{\infty}\cap\beta_{0}$
are labelled.}
\end{figure}

\subsubsection{\label{subsec:intro-k-0}Case $k=0$}

Let $p,q$ be coprime positive integers, let $r=p/q$, and let us
focus on the $k=0$ case for now (see Subsubsection~\ref{subsec:General}
for general $k$). When we construct the rational surgery exact triangle
for slope $r$, the relevant genus $1$ Heegaard diagram $(\mathbb{T}^{2},\beta_{0},\beta_{r},\beta_{\infty},z)$
consists of three attaching curves $\beta_{0},\beta_{r},\beta_{\infty}$,
with slope $0,r,\infty$, respectively: see Figure~\ref{fig:diagramt2-infto}
for $r=p/q=5/3$. We equip $\beta_{0}$ and $\beta_{\infty}$ with
local systems as follows (see Remark~\ref{rem:necessity-local-system}
for why we need local systems): (A) equip $\beta_{0}$ with a rank
$q$ local system $(E_{0},\phi_{0})$ where $E_{0}:=\bigoplus_{j=0}^{q-1}x_{j}\mathbb{F}\llbracket U\rrbracket$
has basis $x_{0},\cdots,x_{q-1}$ and $\phi_{0}$ is the monodromy
whose definition we postpone to Section~\ref{sec:The-main-local},
and (B) equip $\beta_{\infty}$ with a rank $p$ local system $(E_{\infty},\phi_{\infty})$
where $E_{\infty}:=\bigoplus_{i=0}^{p-1}y_{i}\mathbb{F}\llbracket U\rrbracket$
and $\phi_{\infty}$ is the monodromy whose definition we also postpone
to Section~\ref{sec:The-main-local}.

Now, the combinatorial problem is to find cycles 
\[
\psi_{0r}\in\boldsymbol{CF}^{-}(\beta_{0}^{(E_{0},\phi_{0})},\beta_{r}),\ \psi_{r\infty}\in\boldsymbol{CF}^{-}(\beta_{r},\beta_{\infty}^{(E_{\infty},\phi_{\infty})}),\ \psi_{\infty0}\in\boldsymbol{CF}^{-}(\beta_{\infty}^{(E_{\infty},\phi_{\infty})},\beta_{0}^{(E_{0},\phi_{0})})
\]
such that the triangle counting maps vanish, i.e.
\begin{equation}
\mu_{2}(\psi_{0r}\otimes\psi_{r\infty})=0,\ \mu_{2}(\psi_{r\infty}\otimes\psi_{\infty0})=0,\ \mu_{2}(\psi_{\infty0}\otimes\psi_{0r})=0,\label{eq:mu2-intro}
\end{equation}
and the quadrilateral counting maps are the identity\footnote{\label{fn:Note-that-Equation()}Note that Equation~(\ref{eq:mu3-imprecise})
is imprecise: in Heegaard Floer homology, we usually avoid defining
$\boldsymbol{CF}^{-}(\beta,\beta)$, let alone the unit ${\rm Id}\in\boldsymbol{CF}^{-}(\beta,\beta)$.
Instead, we consider standard translates of the attaching curves (see
Subsection~\ref{subsec:Standard-translates}): for instance, instead
of showing $\mu_{3}(\psi_{0r}\otimes\psi_{r\infty}\otimes\psi_{\infty0})\equiv{\rm Id}\bmod U$,
we let $\beta_{0}'$ be a standard translate of $\beta_{0}$ (see
Figure~\ref{fig:diagramt2}), define the corresponding cycle $\psi_{\infty0}'\in\boldsymbol{CF}^{-}(\beta_{\infty}^{(E_{\infty},\phi_{\infty})},{\beta_{0}'}^{(E_{0},\phi_{0})})$
(this is the image of $\psi_{\infty0}$ under the nearest point map
from Subsection~\ref{subsec:Standard-translates}), let $\Theta_{0}^{+}\in\beta_{0}\cap\beta_{0}'$
be the top grading intersection point, and show that
\[
\mu_{3}(\psi_{0r}\otimes\psi_{r\infty}\otimes\psi_{\infty0}')\equiv{\rm Id}_{E_{0}}\Theta_{0}^{+}\mod U.
\]
} modulo $U$, i.e.
\begin{equation}
\mu_{3}(\psi_{0r}\otimes\psi_{r\infty}\otimes\psi_{\infty0})\equiv{\rm Id},\ \mu_{3}(\psi_{r\infty}\otimes\psi_{\infty0}\otimes\psi_{0r})\equiv{\rm Id},\ \mu_{3}(\psi_{\infty0}\otimes\psi_{0r}\otimes\psi_{r\infty})\equiv{\rm Id}\mod U.\label{eq:mu3-imprecise}
\end{equation}

The chain complexes 
\begin{equation}
\boldsymbol{CF}^{-}(\beta_{0}^{(E_{0},\phi_{0})},\beta_{r}),\ \boldsymbol{CF}^{-}(\beta_{r},\beta_{\infty}^{(E_{\infty},\phi_{\infty})}),\ \boldsymbol{CF}^{-}(\beta_{\infty}^{(E_{\infty},\phi_{\infty})},\beta_{0}^{(E_{0},\phi_{0})})\label{eq:chain-cpxes-intro}
\end{equation}
are all free, rank $pq$ modules over $\mathbb{F}\llbracket U\rrbracket$,
with trivial differential $\partial=0$. Moreover, the basis of $E_{0},E_{\infty}$
and the intersection points $\beta_{0}\cap\beta_{r}$, $\beta_{r}\cap\beta_{\infty}$,
and $\beta_{\infty}\cap\beta_{0}$ give rise to bases of these chain
complexes.\footnote{Let us denote the intersection points $\beta_{0}\cap\beta_{r}$ as
$\theta_{0},\cdots,\theta_{p-1}$, $\beta_{r}\cap\beta_{\infty}$
as $\xi_{0},\cdots,\xi_{q-1}$, and $\beta_{\infty}\cap\beta_{0}$
as $\zeta$. Then, the chain complex $\boldsymbol{CF}^{-}(\beta_{0}^{(E_{0},\phi_{0})},\beta_{r})$
is freely generated over the $\mathbb{F}\llbracket U\rrbracket$-module
${\rm Hom}_{\mathbb{F}\llbracket U\rrbracket}(E_{0},\mathbb{F}\llbracket U\rrbracket)$
by the intersection points $\theta_{i}$, and hence has basis $x_{j}^{\ast}\theta_{i}$
for $i=0,\cdots,p-1$ and $j=0,\cdots,q-1$. Similarly, $\boldsymbol{CF}^{-}(\beta_{r},\beta_{\infty}^{(E_{\infty},\phi_{\infty})})$
has basis $\{y_{i}\xi_{j}\}_{i,j}$, and $\boldsymbol{CF}^{-}(\beta_{\infty}^{(E_{\infty},\phi_{\infty})},\beta_{0}^{(E_{0},\phi_{0})})$
has basis $\{x_{j}y_{i}^{\ast}\zeta\}_{i,j}$.}

Writing the cycles $\psi_{0r},\psi_{r\infty},\psi_{\infty0}$ as $\mathbb{F}\llbracket U\rrbracket$-linear
combinations of the above basis elements, Equation~(\ref{eq:mu2-intro})
becomes a system of $3pq$ quadratic equations in $3pq$ variables
over $\mathbb{F}\llbracket U\rrbracket$, and Equation~(\ref{eq:mu3-imprecise})
becomes an additional system of cubic equations. Our main result is
that there always exists a solution that satisfies both these systems
of equations. Showing this is particularly tricky for general $p/q\neq n,1/n$:
indeed, the cycles we find are \emph{not} simply the sum of some of
the basis elements; the coefficients of the $\mathbb{F}\llbracket U\rrbracket$-linear
combination are in general complicated. See Remark~\ref{rem:pq53}
for a concrete example.

There are two main ideas: first, we restrict our attention to particular
subspaces of the chain complexes (\ref{eq:chain-cpxes-intro}); these
are the subspaces spanned by the \emph{standard basis elements} (Definition~\ref{def:standard-basis}).
Restricting to these subspaces makes the triangle counting maps more
symmetric, and reduces Equation~(\ref{eq:mu2-intro}) to a more tractable
system of linear equations. More precisely, we define an $\mathbb{F}\llbracket U\rrbracket$-linear
map $F:X\to Y$ between two free, rank $q$ modules $X,Y$ over $\mathbb{F}\llbracket U\rrbracket$,
and also consider the adjoint $F^{\ast}:Y^{\ast}\to X^{\ast}$ of
$F$. Then, an element $(x,y)\in{\rm ker}F\oplus{\rm ker}F^{\ast}$
gives rise to a solution $\psi_{0r},\psi_{r\infty},\psi_{\infty0}$
of Equation~(\ref{eq:mu2-intro}); see Section~\ref{sec:The-minus-version}.

This first main idea is motivated by that if we restrict to these
particular subspaces, then the polygon counting maps for $\beta_{0}^{(E_{0},\phi_{0})},\beta_{r},\beta_{\infty}^{(E_{\infty},\phi_{\infty})}$
in $\mathbb{T}^{2}$ have clean interpretations in terms of the polygon
counting maps for curves in a cover of $\mathbb{T}^{2}$, without
local systems; see Section~\ref{sec:Local-systems-and}. Also motivated
by this interpretation, we first find a solution modulo $U$ in Section~\ref{sec:The-hat-version},
i.e.\ we find $\psi_{0r},\psi_{r\infty},\psi_{\infty0}$ modulo $U$,
which we denote as $\widehat{\psi}_{0r},\widehat{\psi}_{r\infty},\widehat{\psi}_{\infty0}$,
that satisfy Equation~(\ref{eq:mu2-intro}) modulo $U$ and Equation~(\ref{eq:mu3-imprecise}).
It turns out that if both $x\in{\rm ker}F$ and $y\in{\rm ker}F^{\ast}$
are nonzero, then the corresponding $\psi_{0r},\psi_{r\infty},\psi_{\infty0}$
reduce to $\widehat{\psi}_{0r},\widehat{\psi}_{r\infty},\widehat{\psi}_{\infty0}$
modulo $U$. Hence, we are left to show ${\rm ker}F\neq0$.

The second main idea is to use a symmetry of the genus $1$ Heegaard
diagram (Figure~\ref{fig:diagramt2-infto}) to show that ${\rm ker}F\neq0$.
This symmetry is closely related to the ${\rm Spin}^{c}$ conjugation
symmetry of Heegaard Floer homology, and it induces an involution
$\iota$ of $X$ and $Y$ such that $F$ is $\iota$-equivariant.
Computing the ranks of the $\iota$-fixed subspaces, we find ${\rm rank}X^{\iota}={\rm rank}Y^{\iota}+1$,
which implies ${\rm rank}({\rm ker}F)\ge1$; see Section~\ref{sec:The-minus-version}.

\begin{rem}[Necessity of local systems]
\label{rem:necessity-local-system}To apply the triangle detection
lemma, we would like to find a Heegaard quadruple diagram $(\Sigma,\boldsymbol{\alpha},\boldsymbol{\beta}_{0},\boldsymbol{\beta}_{r},\boldsymbol{\beta}_{\infty},z)$
such that the homology groups in the exact triangles are $\boldsymbol{HF}^{-}(\boldsymbol{\alpha},\boldsymbol{\beta}_{i})$
for $i\in\{0,r,\infty\}$. (See Subsection~\ref{subsec:inter-local}
for further discussion on how this relates to the above genus $1$
diagram.) However, since we consider twisted coefficients in Theorem~\ref{thm:rational-surgery-k=00003D0},
instead of the above, we should equip $\boldsymbol{\beta}_{0}$ and
$\boldsymbol{\beta}_{\infty}$ with local systems and require
\begin{gather*}
\bigoplus_{i=0}^{\min(p,q)-1}\underline{\boldsymbol{HF}^{-}}(Y_{\lambda}(K);\mathbb{F}[\mathbb{Z}/s_{i}\mathbb{Z}])\cong\boldsymbol{HF}^{-}(\boldsymbol{\alpha},\boldsymbol{\beta}_{0}^{(E_{0},\phi_{0})}),\\
\boldsymbol{HF}^{-}(Y_{p\mu+q\lambda}(K))\cong\boldsymbol{HF}^{-}(\boldsymbol{\alpha},\boldsymbol{\beta}_{r}),\\
\underline{\boldsymbol{HF}^{-}}(Y;\mathbb{F}[\mathbb{Z}/p\mathbb{Z}])\cong\boldsymbol{HF}^{-}(\boldsymbol{\alpha},\boldsymbol{\beta}_{\infty}^{(E_{\infty},\phi_{\infty})}).
\end{gather*}
\end{rem}

\subsubsection{\label{subsec:General}General $k$}

For general $k$, we would like to similarly equip $\boldsymbol{\beta}_{\infty}$
with some local system $(E_{\infty}^{k},\phi_{\infty}^{k})$ such
that
\[
\boldsymbol{HFK}_{p,q,k}^{-}(Y,K)\cong\boldsymbol{HF}^{-}(\boldsymbol{\alpha},\boldsymbol{\beta}_{\infty}^{(E_{\infty}^{k},\phi_{\infty}^{k})}).
\]
For this, we need to allow the monodromy $\phi_{\infty}^{k}$ to a
priori involve negative powers of $U$, as in \cite{2308.15658,nahm2025unorientedskeinexacttriangle}.

We would like to solve the combinatorial problem of finding cycles
\[
\psi_{0r}\in\boldsymbol{CF}^{-}(\beta_{0}^{(E_{0},\phi_{0})},\beta_{r}),\ \psi_{r\infty}\in\boldsymbol{CF}^{-}(\beta_{r},\beta_{\infty}^{(E_{\infty}^{k},\phi_{\infty}^{k})}),\ \psi_{\infty0}\in\boldsymbol{CF}^{-}(\beta_{\infty}^{(E_{\infty}^{k},\phi_{\infty}^{k})},\beta_{0}^{(E_{0},\phi_{0})})
\]
that satisfy Equations~(\ref{eq:mu2-intro})~and~(\ref{eq:mu3-imprecise}).
The main idea here is that it is possible to define the local system
$(E_{\infty}^{k},\phi_{\infty}^{k})$, or in other words $\boldsymbol{HFK}_{p,q,k}^{-}(Y,K)$,
such that this combinatorial problem is similar to that for $k=0$.
Concretely, if we restrict to the subspaces spanned by the standard
basis elements as in Subsubsection~\ref{subsec:intro-k-0}, then
the triangle counting maps $\mu_{2}$ do not depend on $k$ (Proposition~\ref{prop:general-k-same}).

\subsection{\label{subsec:The-modified-knot}The modified knot Floer homology
groups \texorpdfstring{$\boldsymbol{HFK}_{p,q,k}^{-}(Y,K)$}{HFKp,q,k(Y,K)}}

To define and motivate these modified knot Floer homology groups $\boldsymbol{HFK}_{p,q,k}^{-}(Y,K)$,
we first recall the definition of the knot Floer chain complex ${\cal CFK}^{-}(Y,K)$,
the Heegaard Floer chain complex with twisted coefficients, and the
$t$-modified knot Floer chain complex $t\boldsymbol{CFK}^{-}(Y,K)$.
\begin{convention}
Throughout, $\mathbb{F}\llbracket W,Z\rrbracket$ is viewed as an
$\mathbb{F}\llbracket U\rrbracket$-algebra via $U=WZ$.
\end{convention}

\begin{defn}[Knot Floer chain complex \cite{MR2065507,MR2704683}]
\label{def:knotfloer}Let $K$ be a knot in a closed, oriented $3$-manifold
$Y$. Let $(\Sigma,\boldsymbol{\alpha},\boldsymbol{\beta},w,z)$ be
a doubly pointed Heegaard diagram that represents this knot $K\subset Y$.
Then the \emph{knot Floer chain complex ${\cal CFK}^{-}(Y,K)$} is
freely generated over $\mathbb{F}\llbracket W,Z\rrbracket$ by intersection
points ${\bf x}\in\mathbb{T}_{\alpha}\cap\mathbb{T}_{\beta}$:
\[
{\cal CFK}^{-}(Y,K):=\bigoplus_{{\bf x}\in\mathbb{T}_{\boldsymbol{\alpha}}\cap\mathbb{T}_{\boldsymbol{\beta}}}\mathbb{F}\llbracket W,Z\rrbracket{\bf x},
\]
and the differential is given by counting the basepoint $w$ with
weight $W$ and $z$ with weight $Z$:
\[
\partial{\bf x}:=\sum_{{\bf y}\in\mathbb{T}_{\boldsymbol{\alpha}}\cap\mathbb{T}_{\boldsymbol{\beta}}}\sum_{\phi\in D({\bf x},{\bf y}),\ \mu(\phi)=1}\#\mathcal{M}(\phi)W^{n_{w}(\phi)}Z^{n_{z}(\phi)}{\bf y},
\]
where $D({\bf x},{\bf y})$ is the set of two-chains from ${\bf x}$
to ${\bf y}$ (Definition~\ref{def:Given-intersection-points}).
\end{defn}

\begin{defn}[{Heegaard Floer chain complex with $\mathbb{F}[\mathbb{Z}/N\mathbb{Z}]$-twisted
coefficients \cite[Section~8]{MR2113020}}]
\label{def:twisted-coeff}Let $Y$ be a closed, oriented $3$-manifold,
let $N$ be a positive integer, and let $\gamma\in H_{1}(Y)$. Let
$K\subset Y$ be any knot with homology class $\gamma$. Let $E_{N}:=\bigoplus_{i=0}^{N-1}e_{i}\mathbb{F}\llbracket U\rrbracket$
be the free, rank $N$ module over $\mathbb{F}\llbracket U\rrbracket$
with basis $e_{0},\cdots,e_{N-1}$. Define $\phi_{N}:E_{N}\to E_{N}$
as the $\mathbb{F}\llbracket U\rrbracket$-linear map such that $\phi_{N}(e_{i})=e_{i+1}$,
where the indices of $e$ are interpreted modulo $N$. View $E_{N}$
as an $\mathbb{F}\llbracket W,Z\rrbracket$-module by letting $W$
act as $\phi_{N}$ and $Z$ act as $U\phi_{N}^{-1}$. Then, the \emph{Heegaard
Floer chain complex with $\mathbb{F}[\mathbb{Z}/N\mathbb{Z}]$-twisted
coefficients, twisted by $\gamma$} is defined as 
\[
\underline{\boldsymbol{CF}^{-}}(Y;\mathbb{F}[\mathbb{Z}/N\mathbb{Z}]):={\cal CFK}^{-}(Y,K)\otimes_{\mathbb{F}\llbracket W,Z\rrbracket}E_{N},
\]
and denote its homology as $\underline{\boldsymbol{HF}^{-}}(Y;\mathbb{F}[\mathbb{Z}/N\mathbb{Z}])$.
\end{defn}

\begin{defn}[$t$-modified knot Floer chain complex \cite{1407.1795}]
\label{def:t-knotfloer}Let $p$ be a positive integer and $k=0,\cdots,p-1$.
Let $D_{p,k}:=\mathbb{F}\llbracket U^{1/p}\rrbracket$, and view it
as an $\mathbb{F}\llbracket W,Z\rrbracket$-module by letting $W$
act as $U^{k/p}$ and $Z$ act as $U^{1-(k/p)}$. Let $K$ be a knot
in a closed, oriented $3$-manifold $Y$, let $t=2k/p$, and further
assume that $p$ and $k$ are coprime. Then, the\emph{ $t$-modified
knot Floer chain complex} is 
\[
t\boldsymbol{CFK}^{-}(Y,K):={\cal CFK}^{-}(Y,K)\otimes_{\mathbb{F}\llbracket W,Z\rrbracket}D_{p,k},
\]
and denote its homology as $t\boldsymbol{HFK}^{-}(Y,K)$.
\end{defn}

\begin{rem}
The author's $2$-surgery exact triangle \cite[Theorem~1.3]{nahm2025unorientedskeinexacttriangle}
involves the unoriented knot Floer homology group, which is defined
as $t\boldsymbol{HFK}^{-}(Y,K)$ for $t=1$.
\end{rem}

Now, we define our modified knot Floer chain complex $\boldsymbol{CFK}_{p,q,k}^{-}(Y,K)$
and compare it with Heegaard Floer homology with twisted coefficients
and $2k/p$-modified knot Floer homology.
\begin{defn}[Modified knot Floer chain complex $\boldsymbol{CFK}_{p,q,k}^{-}(Y,K)$]
\label{def:hfkpqk}Let $K$ be a knot in a closed, oriented $3$-manifold
$Y$, let $p,q$ be coprime positive integers, and let $k=0,\cdots,p-1$.
Let $E_{p,q,k}:=\bigoplus_{i=0}^{p-1}e_{i}\mathbb{F}\llbracket U\rrbracket$
be the free, rank $p$ module over $\mathbb{F}\llbracket U\rrbracket$
with basis $e_{0},\cdots,e_{p-1}$. Define $\phi_{p,q,k}:E_{p,q,k}\to E_{p,q,k}$
as the $\mathbb{F}\llbracket U\rrbracket$-linear map such that 
\[
\phi_{p,q,k}(e_{i})=\begin{cases}
Ue_{i+q} & {\rm if}\ i+q\in\{0,\cdots,k-1\}\mod p\\
e_{i+q} & {\rm if}\ i+q\in\{k,\cdots,p-1\}\mod p
\end{cases},
\]
where the indices of $e$ are interpreted modulo $p$. View $E_{p,q,k}$
as an $\mathbb{F}\llbracket W,Z\rrbracket$-module by letting $W$
act as $\phi_{p,q,k}$ and $Z$ act as $U\phi_{p,q,k}^{-1}$. Let
\[
\boldsymbol{CFK}_{p,q,k}^{-}(Y,K):={\cal CFK}^{-}(Y,K)\otimes_{\mathbb{F}\llbracket W,Z\rrbracket}E_{p,q,k},
\]
and denote its homology as $\boldsymbol{HFK}_{p,q,k}^{-}(Y,K)$.
\end{defn}

In the following sequence of remarks, we compare the modified knot
Floer chain complex $\boldsymbol{CFK}_{p,q,k}^{-}(Y,K)$ with the
Heegaard Floer chain complex with twisted coefficients and the $2k/p$-modified
knot Floer chain complex in various cases.
\begin{rem}[Case $k=0$]
\label{rem:hfkpqk}If $k=0$, then 
\[
\boldsymbol{CFK}_{p,q,0}^{-}(Y,K)\simeq\underline{\boldsymbol{CF}^{-}}(Y;\mathbb{F}[\mathbb{Z}/p\mathbb{Z}])
\]
where we twist by $[K]\in H_{1}(Y)$.
\end{rem}

\begin{rem}[Case $q=k$]
\label{rem:hfkpqk-1}If $q=k$, then we have
\begin{equation}
\boldsymbol{CFK}_{p,q,q}^{-}(Y,K)\simeq\frac{2q}{p}\boldsymbol{CFK}^{-}(Y,K).\label{eq:t-modified-qk}
\end{equation}
Indeed, view $D_{p,k}$ as a free, rank $p$ $\mathbb{F}\llbracket U\rrbracket$-module
with basis $d_{i}:=U^{i/p}$, $i=0,\cdots,p-1$. Then,
\[
Wd_{i}=\begin{cases}
Ud_{i+k} & {\rm if}\ i+k\in\{0,\cdots,k-1\}\mod p\\
d_{i+k} & {\rm if}\ i+k\in\{k,\cdots,p-1\}\mod p
\end{cases},
\]
where the indices of $d$ are interpreted modulo $p$. Hence, the
$\mathbb{F}\llbracket U\rrbracket$-linear map 
\[
D_{p,q}\to E_{p,q,q}:d_{i}\mapsto e_{i}
\]
is an $\mathbb{F}\llbracket W,Z\rrbracket$-module isomorphism, and
so Equation~(\ref{eq:t-modified-qk}) follows.
\end{rem}

\begin{rem}[Case $(p,k)=1$]
\label{rem:hfkpqk-2}More generally, let us consider the case where
$q$ and $k$ are not necessarily equal, but $p$ and $k$ are still
coprime. In this case, it is no longer true that $\boldsymbol{HFK}_{p,q,k}^{-}(Y,K)$
is isomorphic to $\frac{2k}{p}\boldsymbol{HFK}^{-}(Y,K)$, but we
can say something similar. As a concrete example, let us consider
the case $(p,q,k)=(5,1,2)$. Consider the $\mathbb{F}\llbracket W,Z\rrbracket$-subalgebra
\[
R_{5,1,2}:=\mathbb{F}\llbracket U^{2/5},U^{3/5}\rrbracket\subset D_{5,2}=\mathbb{F}\llbracket U^{1/5}\rrbracket.
\]
View $\mathbb{F}\llbracket U^{2/5},U^{3/5}\rrbracket$ as a free $\mathbb{F}\llbracket U\rrbracket$-module
with basis $1,U^{6/5},U^{2/5},U^{3/5},U^{4/5}$; then the $\mathbb{F}\llbracket U\rrbracket$-linear
map 
\[
R_{5,1,2}\to E_{5,1,2}:1\mapsto e_{1},U^{6/5}\mapsto e_{4},U^{2/5}\mapsto e_{2},U^{3/5}\mapsto e_{0},U^{4/5}\mapsto e_{3}
\]
is an $\mathbb{F}\llbracket W,Z\rrbracket$-module isomorphism. Hence,
$\boldsymbol{HFK}_{5,1,2}^{-}(Y,K)$ is the $4/5$-modified knot Floer
homology group defined over $\mathbb{F}\llbracket U^{2/5},U^{3/5}\rrbracket$
instead of $\mathbb{F}\llbracket U^{1/5}\rrbracket$.

In general, if $p$ and $k$ are coprime, then we show in Lemma~\ref{lem:interpret-local-system}
that there exists some $\mathbb{F}\llbracket W,Z\rrbracket$-submodule
$R_{p,q,k}\subset D_{p,k}=\mathbb{F}\llbracket U^{1/p}\rrbracket$
that is isomorphic to $E_{p,q,k}$ as an $\mathbb{F}\llbracket W,Z\rrbracket$-module.
Hence, $\boldsymbol{HFK}_{p,q,k}^{-}(Y,K)$ is the $2k/p$-modified
knot Floer homology defined over the module $R_{p,q,k}$ instead of
$\mathbb{F}\llbracket U^{1/p}\rrbracket$.
\end{rem}

\begin{rem}[Interpretation of $\boldsymbol{HFK}_{p,q,k}^{-}(Y,K)$]
\label{rem:hfkpqk-3}In view of the above remarks, we should think
of the modified knot Floer homology group $\boldsymbol{HFK}_{p,q,k}^{-}(Y,K)$
for general $(p,q,k)$ as a version of $2k/p$-modified knot Floer
homology with twisted coefficients.
\end{rem}

\subsection{Organization}

To prove Theorem~\ref{thm:rational-surgery-kgeneral}, we need to
consider Heegaard Floer homology with local systems whose monodromy
may involve negative powers of $U$, as in \cite{2308.15658,nahm2025unorientedskeinexacttriangle}.
In Section~\ref{sec:Heegaard-Floer-homology}, we review Heegaard
Floer homology in this setting. In particular, in Subsection~\ref{subsec:Weak-admissibility-and},
we review the additional subtleties that these local systems introduce;
we address these in Appendix~\ref{sec:Weak-admissibility-and}.

By standard arguments that we recall in Appendix~\ref{sec:Reduction-of-Theorem},
Theorems~\ref{thm:rational-surgery-k=00003D0}~and~\ref{thm:rational-surgery-kgeneral}
reduce to combinatorial local computations that involve polygon counting
maps; the goal of Section~\ref{sec:The-main-local} is to state the
main local computation (Theorem~\ref{thm:gen-local-comp}) precisely.
Roughly speaking, this says that we can find certain cycles such that
the triangle counting maps ($\mu_{2}$) vanish, and the quadrilateral
counting maps ($\mu_{3}$) are the identity modulo $U$ (recall Footnote~\ref{fn:Note-that-Equation()}).
In Subsection~\ref{subsec:The-quadrilateral-counting}, we find the
cycles modulo $U$ and show that the quadrilateral counting maps are
the identity modulo $U$. That we can find lifts of these cycles such
that the triangle counting maps vanish is more involved; we show this
in Section~\ref{sec:The-minus-version}.

Recall from Subsection~\ref{subsec:Idea-of-the} that there are three
main ideas, two for $k=0$ plus an additional idea for general $k$:
first, we define\emph{ standard basis elements }and restrict to the
subspaces spanned by them. Second, for $k=0$, we consider an involution
to show that we can make the triangle counting maps vanish. Third,
for general $k$, we define the local systems so that the triangle
counting maps are the same as the case $k=0$.

In Section~\ref{sec:Local-systems-and}, we explain the first main
idea. We define standard basis elements, and show that polygon counting
maps for attaching curves with local systems can be interpreted as
polygon counting maps in a cover, if we restrict to the subspaces
spanned by the standard basis elements.

In Section~\ref{sec:The-hat-version}, we motivate the definition
of the local systems (and hence explain some parts of the third main
idea), find the cycles modulo $U$, and prove that the quadrilateral
counting maps are the identity modulo $U$. For completeness, we also
prove that the triangle counting maps vanish modulo $U$ and hence
prove the main local computation for the hat version; the proof of
the main local computation for the minus version does not depend on
this.

In Section~\ref{sec:The-minus-version}, we complete the proof of
the main local computation. We explain the second and third main ideas
in Subsections~\ref{subsec:triangle-k=00003D0}~and~\ref{subsec:Triangle-maps-for},
respectively, and show for $k=0$ and general $k$, respectively,
that there exist cycles that reduce to the cycles from Section~\ref{sec:The-hat-version}
modulo $U$ and such that the triangle counting maps vanish. These
involve some technical lemmas which we prove in Appendix~\ref{sec:Checks-for-Subsubsection}.
These proofs are purely algebraic; in Subsection~\ref{subsec:Picture-proofs},
we give a ``picture interpretation'' of the involution of the second
main idea and a ``picture proof'' of the third main idea, which
we find more illuminating.

In Section~\ref{sec:Examples}, we describe Theorems~\ref{thm:rational-surgery-k=00003D0}~and~\ref{thm:rational-surgery-kgeneral}
for $(p,q)=(5,3)$ and $k=0,1,2$, for the unknot $O$ and the right
handed trefoil $T$ with the Seifert framing in $S^{3}$.

\subsection{\label{subsec:Conventions}Conventions}

Define $a\bmod n\in\{0,\cdots,n-1\}$ for $n\in\mathbb{Z}_{>0}$ and
$a\in\mathbb{Z}$ or $a\in\mathbb{Z}/n\mathbb{Z}$ as follows: if
$a\in\mathbb{Z}$, then $a\bmod n$ denotes the remainder of $a$
modulo $n$; if $a\in\mathbb{Z}/n\mathbb{Z}$, then $a\bmod n$ denotes
the representative of $a$ in $\{0,\cdots,n-1\}$. If $a$ is coprime
to $n$, then $a^{-1}\bmod n$ is the integer $r\in\{0,\cdots,n-1\}$
such that $ra\equiv1\bmod n$.

For Heegaard diagrams drawn on the plane, the almost complex structures
rotate in a counterclockwise direction. For instance, the triangles
in Figures~\ref{fig:zigzags53}~and~\ref{fig:tildet2-1-1} have
holomorphic representatives for the order red, blue, green.

We use the notation $]x,y[$ to denote the open interval between $x$
and $y$ to avoid confusion between $]x,y[$ and the point $(x,y)\in\mathbb{R}^{2}$.

\subsection{Acknowledgements}

Theorem~\ref{thm:rational-surgery-k=00003D0} was motivated by the
question of whether there exist rational surgery exact triangles in
Heegaard Floer homology, and Theorem~\ref{thm:rational-surgery-kgeneral}
was further motivated by Bhat's $2$-surgery exact triangle in instanton
Floer homology \cite{2311.04242} which led to \cite[Theorem~1.3]{nahm2025unorientedskeinexacttriangle}.
We thank Peter Ozsv\'{a}th for introducing this question to the author,
encouraging the author to work on it, explaining a lot of the arguments
in this paper, and for helpful discussions. We also thank Ian Zemke
for his continuous support, teaching the author a lot of previous
works, especially \cite{2308.15658}, and for helpful discussions.
We thank Deeparaj Bhat for sharing his work on the 2-surgery exact
triangle back in March 2023. We thank Jae Hee Lee, Robert Lipshitz,
Jacob Rasmussen, and Fan Ye for helpful discussions.

\section{\label{sec:Heegaard-Floer-homology}Heegaard Floer homology with
local systems}

In this paper, we consider attaching curves equipped with local systems
that a priori involve negative powers of $U$, which were first considered
recently by Zemke \cite{2308.15658} (compare \cite{nahm2025unorientedskeinexacttriangle}).
In this section, we review Heegaard Floer homology in this new setting.
\begin{defn}[{Heegaard diagram \cite[Definition 3.1]{MR2443092}}]
Given a closed, oriented, genus $g$ surface $\Sigma$, an \emph{attaching
curve} $\boldsymbol{\alpha}=\alpha^{1}\cup\cdots\cup\alpha^{g}$ is
the union of $g$ many pairwise disjoint, simple closed curves $\alpha^{1},\cdots,\alpha^{g}$
on $\Sigma$ whose images span a $g$-dimensional subspace of $H_{1}(\Sigma)$.
A\emph{ Heegaard diagram} is a collection $(\Sigma,\boldsymbol{\alpha}_{0},\cdots,\boldsymbol{\alpha}_{m},z)$
of such a surface $\Sigma$, attaching curves $\boldsymbol{\alpha}_{0},\cdots,\boldsymbol{\alpha}_{m}$,
and a point $z$ on $\Sigma$ away from the attaching curves, which
we call the \emph{basepoint}. We assume that attaching curves intersect
transversely, and that there are no triple intersections. 
\end{defn}

\begin{defn}[Two-chains]
Given a Heegaard diagram $(\Sigma,\boldsymbol{\alpha}_{0},\boldsymbol{\alpha}_{1},\cdots,\boldsymbol{\alpha}_{m},z)$,
an\emph{ elementary two-chain }is a connected component of $\Sigma\setminus\left(\boldsymbol{\alpha}_{0}\cup\cdots\cup\boldsymbol{\alpha}_{m}\right)$,
and a \emph{two-chain} is a formal sum of elementary two-chains.

Denote the \emph{$\boldsymbol{\alpha}_{i}$-boundary of ${\cal D}$}
as $\partial_{\boldsymbol{\alpha}_{i}}{\cal D}$ (which is a one-chain).
A \emph{cornerless two-chain} is a two-chain ${\cal D}$ such that
$\partial_{\boldsymbol{\alpha}_{i}}{\cal D}$ is a cycle for all $i$.

If ${\cal D}$ is a two-chain and $x$ is a point away from the attaching
curves, then the \emph{local multiplicity $n_{x}({\cal D})$ of ${\cal D}$
at $x$} is the coefficient of the elementary two-chain containing
$x$ in ${\cal D}$. By abuse of notation, if $x=\{x_{1},\cdots,x_{n}\}$
is a set of such points, then let $n_{x}({\cal D}):=n_{x_{1}}({\cal D})+\cdots+n_{x_{n}}({\cal D})$.

A two-chain is \emph{nonnegative} if all its local multiplicities
are nonnegative.
\end{defn}

\begin{defn}[Set of two-chains $D({\bf x}_{1},\cdots,{\bf x}_{k+1})$ with vertices
${\bf x}_{1},\cdots,{\bf x}_{k+1}$]
\label{def:Given-intersection-points}Let $k\ge1$ be an integer,
and let ${\bf x}_{j}\in\mathbb{T}_{\boldsymbol{\alpha}_{i_{j-1}}}\cap\mathbb{T}_{\boldsymbol{\alpha}_{i_{j}}}$
be intersection points for $j=1,\cdots,k+1$ (interpret the indices
of $i$ and ${\bf x}$ modulo $k+1$). Define the \emph{set of two-chains
$D({\bf x}_{1},\cdots,{\bf x}_{k+1})$ with vertices ${\bf x}_{1},\cdots,{\bf x}_{k+1}$}
as the set of two-chains ${\cal D}$ such that 
\[
\partial(\partial_{\boldsymbol{\alpha}_{i_{j}}}{\cal D})={\bf x}_{j+1}-{\bf x}_{j},\ \partial_{\boldsymbol{\alpha}_{i}}{\cal D}=0
\]
for $j=1,\cdots,k+1$ and $i\neq i_{1},\cdots,i_{k+1}$.
\end{defn}

\subsection{Local systems}

We equip attaching curves with local systems of the following form.
\begin{defn}[Local systems]
Given an attaching curve $\boldsymbol{\alpha}$ of a Heegaard diagram,
a \emph{local system on $\boldsymbol{\alpha}$} is a triple $(E,\phi,A)$
where (1) $E$ is a finite, free $\mathbb{F}\llbracket U\rrbracket$-module,
(2) $\phi$ is an element of ${\rm Hom}_{\mathbb{F}\llbracket U\rrbracket}(E,E)$
which we call the \emph{monodromy}, and (3) $A$ is an oriented arc
on the Heegaard surface such that it intersects $\boldsymbol{\alpha}$
transversely, at a single point.\footnote{Of course, it is sufficient to consider an oriented point $p$ on
$\boldsymbol{\alpha}$ instead of an oriented arc. We chose this definition
to simplify the exposition for standard translates (Subsection~\ref{subsec:Standard-translates}).} Write $\boldsymbol{\alpha}^{(E,\phi,A)}$, or simply $\boldsymbol{\alpha}^{(E,\phi)}$
or $\boldsymbol{\alpha}^{E}$, to signify that $\boldsymbol{\alpha}$
is equipped with the local system $(E,\phi,A)$.

If $E=\mathbb{F}\llbracket U\rrbracket$ and $\phi={\rm Id}_{\mathbb{F}\llbracket U\rrbracket}$,
then we say that $(E,\phi,A)$ is a\emph{ trivial local system}. In
this case, we may omit specifying the oriented arc $A$ and also simply
write $\boldsymbol{\alpha}$ instead of $\boldsymbol{\alpha}^{(E,\phi,A)}$.
\end{defn}

Roughly speaking, whenever the $\boldsymbol{\alpha}$-boundary $\partial_{\boldsymbol{\alpha}}{\cal D}$
of a two-chain ${\cal D}$ intersects $A$, we get a contribution
of $\phi$ or $\phi^{-1}$ depending on the sign of the intersection.
More precisely, we have Definitions~\ref{def:hf-chain-cpx}~and~\ref{def:Define-(higher)-composition}.

The cautious reader might have noticed that $\phi^{-1}$ is a priori
not defined; however, as we will explain in Subsection~\ref{subsec:Weak-admissibility-and}
and Appendix~\ref{sec:Weak-admissibility-and}, Equations~(\ref{eq:differential})~and~(\ref{eq:higher-multiplication})
are well-defined (after a minor modification) for all the Heegaard
diagrams and local systems in this paper.
\begin{defn}[Heegaard Floer chain complex]
\label{def:hf-chain-cpx}Given a Heegaard diagram $(\Sigma,\boldsymbol{\alpha},\boldsymbol{\beta},z)$
and local systems $(E_{\boldsymbol{\alpha}},\phi_{\boldsymbol{\alpha}},A_{\boldsymbol{\alpha}}),(E_{\boldsymbol{\beta}},\phi_{\boldsymbol{\beta}},A_{\boldsymbol{\beta}})$
on $\boldsymbol{\alpha},\boldsymbol{\beta}$, respectively, define
the $\mathbb{F}\llbracket U\rrbracket$-module $\boldsymbol{CF}^{-}(\boldsymbol{\alpha}^{E_{\boldsymbol{\alpha}}},\boldsymbol{\beta}^{E_{\boldsymbol{\beta}}})$
as a direct sum of copies of ${\rm Hom}_{\mathbb{F}\llbracket U\rrbracket}(E_{\boldsymbol{\alpha}},E_{\boldsymbol{\beta}})$:
\[
\boldsymbol{CF}^{-}(\boldsymbol{\alpha}^{E_{\boldsymbol{\alpha}}},\boldsymbol{\beta}^{E_{\boldsymbol{\beta}}}):=\bigoplus_{{\bf x}\in\mathbb{T}_{\boldsymbol{\alpha}}\cap\mathbb{T}_{\boldsymbol{\beta}}}{\rm Hom}_{\mathbb{F}\llbracket U\rrbracket}(E_{\boldsymbol{\alpha}},E_{\boldsymbol{\beta}}){\bf x},
\]
and define the differential $\partial$ as the $\mathbb{F}\llbracket U\rrbracket$-linear
map such that 
\begin{equation}
\partial(e{\bf x})=\sum_{{\cal D}\in D({\bf x},{\bf y}),\ \mu({\cal D})=1}\#{\cal M}({\cal D})U^{n_{z}({\cal D})}\rho({\cal D})(e){\bf y},\label{eq:differential}
\end{equation}
where $e\in{\rm Hom}_{\mathbb{F}\llbracket U\rrbracket}(E_{\boldsymbol{\alpha}},E_{\boldsymbol{\beta}})$
and $\rho({\cal D}):{\rm Hom}_{\mathbb{F}\llbracket U\rrbracket}(E_{\boldsymbol{\alpha}},E_{\boldsymbol{\beta}})\to{\rm Hom}_{\mathbb{F}\llbracket U\rrbracket}(E_{\boldsymbol{\alpha}},E_{\boldsymbol{\beta}})$
is the \emph{monodromy of ${\cal D}$}: 
\[
\rho({\cal D})=\phi_{\boldsymbol{\beta}}^{\#(A_{\boldsymbol{\beta}}\cap\partial_{\boldsymbol{\beta}}{\cal D})}\circ e\circ\phi_{\boldsymbol{\alpha}}^{\#(A_{\boldsymbol{\alpha}}\cap\partial_{\boldsymbol{\alpha}}{\cal D})},
\]
where the intersection numbers $\#(A_{\boldsymbol{\alpha}}\cap\partial_{\boldsymbol{\alpha}}{\cal D})$
and $\#(A_{\boldsymbol{\beta}}\cap\partial_{\boldsymbol{\beta}}{\cal D})$
are counted with sign.
\end{defn}

\begin{defn}[Polygon counting maps $\mu_{n}$]
\label{def:Define-(higher)-composition}Define (higher) composition
maps $\mu_{d}$ similarly: if $\boldsymbol{\alpha}_{0}^{(E_{0},\phi_{0},A_{0})},\cdots,\boldsymbol{\alpha}_{d}^{(E_{d},\phi_{d},A_{d})}$
are attaching curves with local systems, then the composition map
is given by 
\begin{multline}
\mu_{d}:\boldsymbol{CF}^{-}(\boldsymbol{\alpha}_{0}^{E_{0}},\boldsymbol{\alpha}_{1}^{E_{1}})\otimes\cdots\otimes\boldsymbol{CF}^{-}(\boldsymbol{\alpha}_{d-1}^{E_{d-1}},\boldsymbol{\alpha}_{d}^{E_{d}})\to\boldsymbol{CF}^{-}(\boldsymbol{\alpha}_{0}^{E_{0}},\boldsymbol{\alpha}_{d}^{E_{d}}):\\
e_{1}{\bf x}_{1}\otimes\cdots\otimes e_{d}{\bf x}_{d}\mapsto\sum_{{\cal D}\in D({\bf x}_{1},\cdots,{\bf x}_{d},{\bf y}),\ \mu({\cal D})=2-d}\#{\cal M}({\cal D})U^{n_{z}({\cal D})}\rho({\cal D})(e_{1}\otimes\cdots\otimes e_{d}){\bf y},\label{eq:higher-multiplication}
\end{multline}
where $\rho({\cal D})$ is the \emph{monodromy of ${\cal D}$} defined
as follows: 
\begin{multline*}
\rho({\cal D}):{\rm Hom}_{\mathbb{F}\llbracket U\rrbracket}(E_{0},E_{1})\otimes\cdots\otimes{\rm Hom}_{\mathbb{F}\llbracket U\rrbracket}(E_{d-1},E_{d})\to{\rm Hom}_{\mathbb{F}\llbracket U\rrbracket}(E_{0},E_{d}):\\
e_{1}\otimes\cdots\otimes e_{d}\mapsto\phi_{d}^{\#(A_{d}\cap\partial_{\boldsymbol{\alpha}_{d}}({\cal D}))}\circ e_{d}\circ\phi_{d-1}^{\#(A_{d-1}\cap\partial_{\boldsymbol{\alpha}_{d-1}}({\cal D}))}\circ e_{d-1}\circ\cdots\circ\phi_{1}^{\#(A_{1}\cap\partial_{\boldsymbol{\alpha}_{1}}({\cal D}))}\circ e_{1}\circ\phi_{0}^{\#(A_{0}\cap\partial_{\boldsymbol{\alpha}_{0}}({\cal D}))}.
\end{multline*}
\end{defn}

\subsection{\label{subsec:Weak-admissibility-and}Positivity and weak admissibility}

Given a Heegaard diagram $(\Sigma,\boldsymbol{\alpha}_{0},\cdots,\boldsymbol{\alpha}_{m},z)$
together with local systems $(E_{i},\phi_{i},A_{i})$ on $\boldsymbol{\alpha}_{i}$,
there are two technical conditions that need to be satisfied in order
for the attaching curves $\boldsymbol{\alpha}_{i}^{(E_{i},\phi_{i},A_{i})}$
with local systems to form an $A_{\infty}$-category, and for changing
the almost complex structure to induce an $A_{\infty}$-functor. In
this subsection, we consider a simple case of this in Lemma~\ref{lem:local-system-simple-lemma};
see Appendix~\ref{sec:Weak-admissibility-and} for a complete treatment.

First, note that we did not require $\phi_{i}\in{\rm Hom}_{\mathbb{F}\llbracket U\rrbracket}(E_{i},E_{i})$
to be invertible; in fact, for some of the $\phi_{i}$'s that we consider
in this paper, their inverses will a priori only be elements of ${\rm Hom}_{\mathbb{F}\llbracket U\rrbracket}(U^{-1}E_{i},U^{-1}E_{i})$.
Hence, roughly speaking, we have to check that negative powers of
$U$ do not appear in the $\mu_{d}$'s and the maps induced by changing
the almost complex structure, for all the Heegaard diagrams and local
systems that we consider in this paper. We call this \emph{positivity}.

Second, the sum (\ref{eq:higher-multiplication}) may a priori be
infinite, and so we need to check that it is well-defined, and also
that the map induced by changing the almost complex structure is well-defined.
We call this \emph{weak admissibility}, following the standard convention.
It will turn out that for the local systems that we consider in this
paper, the usual definition of weak admissibility (Definition~\ref{def:weakly-admissible})
works.
\begin{defn}[Weak admissibility]
\label{def:weakly-admissible}A Heegaard diagram $(\Sigma,\boldsymbol{\alpha}_{0},\cdots,\boldsymbol{\alpha}_{m},z)$
is \emph{weakly admissible} if every cornerless two-chain ${\cal D}$
such that $n_{z}({\cal D})=0$ has both positive and negative local
multiplicities.
\end{defn}

\begin{figure}[h]
\begin{centering}
\includegraphics[scale=1.5]{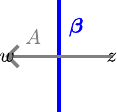}
\par\end{centering}
\caption{\label{fig:local-system-simple-lemma}A local diagram for the Heegaard
diagram near $A$ for Lemma~\ref{lem:local-system-simple-lemma}.
If a path $\gamma$ in $\boldsymbol{\beta}$ goes from top to bottom
in this diagram, then $\#(A\cap\gamma)=1$.}
\end{figure}

If there are only two attaching curves and one of them is equipped
with the trivial local system, then Lemma~\ref{lem:local-system-simple-lemma}
gives a simple condition that guarantees positivity and weak admissibility.
\begin{lem}[Chain complex is well-defined, simple case]
\label{lem:local-system-simple-lemma}Let $(\Sigma,\boldsymbol{\alpha},\boldsymbol{\beta},z)$
be a weakly admissible Heegaard diagram and let $(E,\phi,A)$ be a
local system on $\boldsymbol{\beta}$. If (1) $U\phi^{-1}\in{\rm Hom}_{\mathbb{F}\llbracket U\rrbracket}(E,E)$,
(2) the oriented arc $A$ is disjoint from $\boldsymbol{\alpha}$,
and (3) the initial point of $A$ is $z$ (see Figure~\ref{fig:local-system-simple-lemma}),
then $(\boldsymbol{CF}^{-}(\boldsymbol{\alpha},\boldsymbol{\beta}^{E}),\partial)$
is a well-defined chain complex. Moreover, if we let $w$ be the terminal
point of $A$, then 
\[
\boldsymbol{CF}^{-}(\boldsymbol{\alpha},\boldsymbol{\beta}^{E})\cong{\cal CFK}^{-}(\boldsymbol{\alpha},\boldsymbol{\beta})\otimes_{\mathbb{F}\llbracket W,Z\rrbracket}E,
\]
where ${\cal CFK}^{-}(\boldsymbol{\alpha},\boldsymbol{\beta})$ is
the knot Floer chain complex (Definition~\ref{def:knotfloer}) defined
by the doubly pointed Heegaard diagram $(\Sigma,\boldsymbol{\alpha},\boldsymbol{\beta},w,z)$,
and $E$ is viewed as an $\mathbb{F}\llbracket W,Z\rrbracket$-module
by letting $W$ act as $\phi$ and $Z$ act as $U\phi^{-1}$.
\end{lem}

\begin{proof}
First, the weak admissibility condition ensures that the differential
on ${\cal CFK}^{-}(\boldsymbol{\alpha},\boldsymbol{\beta})$ is well-defined.
Both the underlying modules of $\boldsymbol{CF}^{-}(\boldsymbol{\alpha},\boldsymbol{\beta}^{E})$
and ${\cal CFK}^{-}(\boldsymbol{\alpha},\boldsymbol{\beta})\otimes_{\mathbb{F}\llbracket W,Z\rrbracket}E$
are freely generated over $E$ by intersection points $\mathbb{T}_{\boldsymbol{\alpha}}\cap\mathbb{T}_{\boldsymbol{\beta}}$;
this gives a module isomorphism, say $\Psi$, between them. Hence,
for any ${\bf x},{\bf y}\in\mathbb{T}_{\boldsymbol{\alpha}}\cap\mathbb{T}_{\boldsymbol{\beta}}$
and two-chain ${\cal D}\in D({\bf x},{\bf y})$, by considering the
ends of $A\cap{\cal D}$, we obtain 
\[
\#(A\cap\partial_{\boldsymbol{\beta}}{\cal D})=n_{w}({\cal D})-n_{z}({\cal D}),
\]
and so 
\[
U^{n_{z}({\cal D})}\phi^{\#(A\cap\partial_{\boldsymbol{\beta}}{\cal D})}=(U\phi^{-1})^{n_{z}({\cal D})}\phi^{n_{w}({\cal D})}.
\]
This shows that the module isomorphism $\Psi$ is a chain map and
that in particular, $\partial$ on $\boldsymbol{CF}^{-}(\boldsymbol{\alpha},\boldsymbol{\beta}^{E})$
is a well-defined differential.
\end{proof}

\subsection{\label{subsec:Standard-translates}Standard translates}

In Heegaard Floer homology, instead of defining $\boldsymbol{HF}^{-}(\boldsymbol{\alpha},\boldsymbol{\alpha})$
and the identity in $\boldsymbol{HF}^{-}(\boldsymbol{\alpha},\boldsymbol{\alpha})$,
we consider a \emph{standard translate} $\boldsymbol{\alpha}'$ of
$\boldsymbol{\alpha}$ and define $\Theta_{\boldsymbol{\alpha}}^{+}\in\boldsymbol{CF}^{-}(\boldsymbol{\alpha},\boldsymbol{\alpha}')$
that serves as the identity. In this subsection, we review this and
set up notations for attaching curves equipped with local systems.

\begin{figure}[h]
\begin{centering}
\includegraphics[scale=1.5]{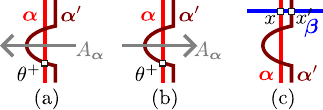}
\par\end{centering}
\caption{\label{fig:standard-translate}(a), (b): Two cases for a standard
translate $\boldsymbol{\alpha}'$ of $\boldsymbol{\alpha}$, near
the oriented arc $A_{\boldsymbol{\alpha}}$. The basepoint is not
drawn. (c): $x'$ is the nearest point to $x$.}
\end{figure}

Let $\boldsymbol{\alpha}=\alpha^{1}\cup\cdots\cup\alpha^{g}$ be an
attaching curve in some Heegaard diagram. Another attaching curve
$\boldsymbol{\alpha}'={\alpha'}^{1}\cup\cdots\cup{\alpha'}^{g}$ is
a \emph{standard translate} of $\boldsymbol{\alpha}$ if each ${\alpha'}^{i}$
is given by pushing off $\alpha^{i}$ slightly and then modifying
in a small region so that $|\alpha^{i}\cap{\alpha'}^{j}|=2\delta_{ij}$.
Assume that these small regions are disjoint from all the attaching
curves that are not $\boldsymbol{\alpha}$ or standard translates
of $\boldsymbol{\alpha}$. Define $\Theta_{\boldsymbol{\alpha}}^{+}\in\mathbb{T}_{\boldsymbol{\alpha}}\cap\mathbb{T}_{\boldsymbol{\alpha}'}$
as the intersection point in the top homological grading of $\boldsymbol{CF}^{-}(\boldsymbol{\alpha},\boldsymbol{\alpha}')$,
i.e.\ it is given by the intersection points $\theta^{+}$ in Figure~\ref{fig:standard-translate}~(a)~and~(b).

If $\boldsymbol{\alpha}$ is equipped with a local system $(E_{\boldsymbol{\alpha}},\phi_{\boldsymbol{\alpha}},A_{\boldsymbol{\alpha}})$,
then further assume that the Heegaard diagram looks like Figure~\ref{fig:standard-translate}~(a)~or~(b)
near the oriented arc $A_{\boldsymbol{\alpha}}$ (in fact, in this
paper, it is possible to only work with the case Figure~\ref{fig:standard-translate}~(a)).
In particular, $A_{\boldsymbol{\alpha}}$ is disjoint from all the
other attaching curves $\boldsymbol{\beta}\neq\boldsymbol{\alpha},\boldsymbol{\alpha}'$.

Equip $\boldsymbol{\alpha}'$ with the same local system $(E_{\boldsymbol{\alpha}},\phi_{\boldsymbol{\alpha}},A_{\boldsymbol{\alpha}})$.
Then,
\[
{\rm Id}_{E_{\boldsymbol{\alpha}}}\Theta_{\boldsymbol{\alpha}}^{+}\in\boldsymbol{CF}^{-}(\boldsymbol{\alpha}^{E_{\boldsymbol{\alpha}}},{\boldsymbol{\alpha}'}^{E_{\boldsymbol{\alpha}}})
\]
is a cycle, and for any other attaching curve $\boldsymbol{\beta}$
with local system $(E_{\boldsymbol{\beta}},\phi_{\boldsymbol{\beta}},A_{\boldsymbol{\beta}})$,
\[
\mu_{2}(-\otimes{\rm Id}_{E_{\boldsymbol{\alpha}}}\Theta_{\boldsymbol{\alpha}}^{+}):\boldsymbol{CF}^{-}(\boldsymbol{\beta}^{E_{\boldsymbol{\beta}}},{\boldsymbol{\alpha}}^{E_{\boldsymbol{\alpha}}})\to\boldsymbol{CF}^{-}(\boldsymbol{\beta}^{E_{\boldsymbol{\beta}}},{\boldsymbol{\alpha}'}^{E_{\boldsymbol{\alpha}}})
\]
is a quasi-isomorphism by the same argument as \cite[Proposition~9.8]{MR2113019}.

A closely related module isomorphism is the \emph{nearest point map}
\[
\Phi:\boldsymbol{CF}^{-}(\boldsymbol{\beta}^{E_{\boldsymbol{\beta}}},{\boldsymbol{\alpha}}^{E_{\boldsymbol{\alpha}}})\to\boldsymbol{CF}^{-}(\boldsymbol{\beta}^{E_{\boldsymbol{\beta}}},{\boldsymbol{\alpha}'}^{E_{\boldsymbol{\alpha}}}),
\]
defined as follows: for each ${\bf x}=\{x^{1},\cdots,x^{g}\}\in\mathbb{T}_{\boldsymbol{\beta}}\cap\mathbb{T}_{\boldsymbol{\alpha}}$,
define\emph{ its nearest point }${\bf x}'=\{{x^{1}}',\cdots,{x^{g}}'\}\in\mathbb{T}_{\boldsymbol{\beta}}\cap\mathbb{T}_{\boldsymbol{\alpha}'}$
by letting each ${x^{i}}'\in\boldsymbol{\beta}\cap\boldsymbol{\alpha}'$
be the nearest point to $x^{i}\in\boldsymbol{\beta}\cap\boldsymbol{\alpha}$
(see Figure~\ref{fig:standard-translate}~(c)). Similarly, if $\boldsymbol{\beta}'$
is a standard translate of $\boldsymbol{\beta}$, define 
\[
\Phi:\boldsymbol{CF}^{-}({\boldsymbol{\beta}}^{E_{\boldsymbol{\beta}}},{\boldsymbol{\alpha}}^{E_{\boldsymbol{\alpha}}})\to\boldsymbol{CF}^{-}({\boldsymbol{\beta}'}^{E_{\boldsymbol{\beta}}},{\boldsymbol{\alpha}}^{E_{\boldsymbol{\alpha}}}),\ \Phi:\boldsymbol{CF}^{-}({\boldsymbol{\beta}}^{E_{\boldsymbol{\beta}}},{\boldsymbol{\alpha}}^{E_{\boldsymbol{\alpha}}})\to\boldsymbol{CF}^{-}({\boldsymbol{\beta}'}^{E_{\boldsymbol{\beta}}},{\boldsymbol{\alpha}'}^{E_{\boldsymbol{\alpha}}})
\]
similarly. Given an element $\psi\in\boldsymbol{CF}^{-}(\boldsymbol{\beta}^{E_{\boldsymbol{\beta}}},{\boldsymbol{\alpha}}^{E_{\boldsymbol{\alpha}}})$,
we denote $\Phi(\psi)$ as $\psi'$ or even $\psi$ by abuse of notation.
Note that $\Phi$ is in general not a chain map.

\section{\label{sec:The-main-local}The main local computation}

By a standard argument involving the triangle detection lemma \cite[Lemma 4.2]{MR2141852}
and a neck-stretching argument, constructing surgery exact triangles
(Theorem~\ref{thm:rational-surgery-kgeneral}) reduces to a combinatorial
local computation in a genus $1$ Heegaard diagram. The main theorem
of this paper is this combinatorial local computation, and the goal
of this section is to state this precisely (Theorem~\ref{thm:gen-local-comp}).
We recall how Theorem~\ref{thm:rational-surgery-kgeneral} follows
from Theorem~\ref{thm:gen-local-comp} in Subsection~\ref{subsec:inter-local}
and Appendix~\ref{sec:Reduction-of-Theorem}.

Note that the indices of the variables should often be interpreted
modulo $p$ or $q$ depending on the context, and some variables have
indices in the index set $I$ (Definition~\ref{def:index-set}).
As a rule of thumb, $i\in\mathbb{Z}/p\mathbb{Z}$ and $j\in\mathbb{Z}/q\mathbb{Z}$.
\begin{defn}
\label{def:index-set}The \emph{index set} $I$ is the set of $(i,\ell)$
for $i\in\mathbb{Z}/p\mathbb{Z}$ such that $c_{i}\neq0$, and $\ell\in\mathbb{Z}/c_{i}\mathbb{Z}$.
\end{defn}

\subsection{\label{subsec:The-setup}The genus $1$ Heegaard diagram}

\begin{figure}[h]
\begin{centering}
\raisebox{-0.5\height}{\includegraphics{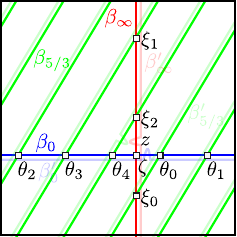}}\qquad{}\raisebox{-0.5\height}{\includegraphics[scale=2.5]{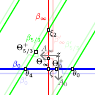}}
\par\end{centering}
\caption{\label{fig:diagramt2}Curves and intersection points on $\mathbb{T}^{2}$
for $(p,q)=(5,3)$. Right: zoomed in to show the intersection points
$\Theta_{0}^{+},\Theta_{r}^{+},\Theta_{\infty}^{+}$ and the oriented
arcs $A_{\infty},A_{0}$. If a path $\gamma$ in $\beta_{\infty}$
goes from top to bottom, then $\#(A_{\infty}\cap\gamma)=1$, and if
a path $\gamma$ in $\beta_{0}$ goes from left to right, then $\#(A_{0}\cap\gamma)=1$.}
\end{figure}

Let $p,q\in\mathbb{Z}_{>0}$ be coprime, and let $r=p/q$. Let us
first describe the genus $1$ Heegaard diagram that we work with (see
Figure~\ref{fig:diagramt2}); here, the Heegaard surface is $\mathbb{T}^{2}=\mathbb{R}^{2}/\mathbb{Z}^{2}$.
Note that this genus $1$ Heegaard diagram is weakly admissible.

\subsubsection*{The attaching curves $\beta_{0},\beta_{r},\beta_{\infty}$}

The circle $\beta_{0}$ is the projection of the $x$-axis of $\mathbb{R}^{2}$
to $\mathbb{T}^{2}$, and $\beta_{\infty}$ is the projection of the
$y$-axis of $\mathbb{R}^{2}$ to $\mathbb{T}^{2}$. Consider the
line in $\mathbb{R}^{2}$ with slope $r$ that intersects the $x$-axis
at the point $(1/(2p),0)$. The circle $\beta_{r}$ is the projection
of this line to $\mathbb{T}^{2}$.

\subsubsection*{Standard translates}

Let $\beta_{0}',\beta_{r}',\beta_{\infty}'$ be the standard translate
of $\beta_{0},\beta_{r},\beta_{\infty}$, respectively, as in Figure~\ref{fig:diagramt2},
and let $\Theta_{0}^{+},\Theta_{r}^{+},\Theta_{\infty}^{+}$ be the
intersection points in $\beta_{0}\cap\beta_{0}'$, $\beta_{r}\cap\beta_{r}'$,
$\beta_{\infty}\cap\beta_{\infty}'$, respectively, in the top homological
grading. 

Note that we are free to swap $\beta_{0}$ and $\beta_{0}'$, $\beta_{r}$
and $\beta_{r}'$, and $\beta_{\infty}$ and $\beta_{\infty}'$: the
position of $\Theta_{0}^{+},\Theta_{r}^{+},\Theta_{\infty}^{+}$ in
the Heegaard diagram will change, but the statement and the proof
of the main local computation, Theorem~\ref{thm:gen-local-comp},
are unaffected.

\subsubsection*{The intersection points $\theta_{i},\xi_{j},\zeta$}

Let $\overline{\pi}:\mathbb{R}^{2}\to\mathbb{T}^{2}=\mathbb{R}^{2}/\mathbb{Z}^{2}$
be the quotient map. In $\mathbb{R}^{2}$, 
\[
\overline{\pi}^{-1}(\beta_{r})\cap([0,1]\times\{0\})
\]
consists of $p$ points. Denote their images in $\mathbb{T}^{2}$
as $\theta_{0},\cdots,\theta_{p-1}$, from left to right. Similarly,
\[
\overline{\pi}^{-1}(\beta_{r})\cap(\{0\}\times[-1,0])
\]
consists of $q$ points. Denote their images in $\mathbb{T}^{2}$
as $\xi_{0},\cdots,\xi_{q-1}$, from top to bottom. Finally, denote
the unique intersection point $\beta_{\infty}\cap\beta_{0}$ as $\zeta$.

The indices of $\theta$ are interpreted modulo $p$, and the indices
of $\xi$ are interpreted modulo $q$.

\subsubsection*{The basepoint $z$}

The basepoint $z$ is at $(\varepsilon,\varepsilon)$ for some $\varepsilon>0$
sufficiently small such that no attaching curves except $\beta_{\infty}$
and $\beta_{\infty}'$ intersect $[-\varepsilon,\varepsilon]\times\{\varepsilon\}$.

\subsubsection*{The oriented arcs $A_{\infty}$ and $A_{0}$}

The oriented arc $A_{\infty}$ is $[-\varepsilon,\varepsilon]\times\{\varepsilon\}$
and is oriented from right to left. The oriented arc $A_{0}$ is $\{\varepsilon\}\times[-\varepsilon,\varepsilon]$
and is oriented from top to bottom.

\subsection{\label{subsec:Local-systems}Local systems}

Let us fix coprime positive integers $p,q$. In this paper, we will
consider particular kinds of local systems on our attaching curves.
They are determined by two sequences: $u=(u_{0},\cdots,u_{p-1})$
where $u_{i}\in\{1,U\}$, and $c=(c_{0},\cdots,c_{p-1})$ where $c_{i}\in\mathbb{Z}_{\ge0}$.
The indices of $u$ and $c$ are interpreted modulo $p$.

Given such $u$ and $c$, we consider the following local systems.
\begin{enumerate}
\item Equip $\beta_{0}$ with $(E_{0},\phi_{0},A_{0})$, where 
\[
E_{0}:=\bigoplus_{i=0}^{p-1}\bigoplus_{\ell=0}^{c_{i}-1}x_{i,\ell}\mathbb{F}\llbracket U\rrbracket,\ \phi_{0}=\sum_{i=0}^{p-1}\sum_{\ell=0}^{c_{i}-1}x_{i,\ell+1}x_{i,\ell}^{\ast}:E_{0}\to E_{0},
\]
i.e.\ $E_{0}$ is a free, rank $\sum_{i=0}^{p-1}c_{i}$ $\mathbb{F}\llbracket U\rrbracket$-module
and $\phi_{0}:x_{i,\ell}\mapsto x_{i,\ell+1}$. Here, the indices
of $x$ lie in the index set $I$ from Definition~\ref{def:index-set}.
\item Equip $\beta_{r}$ with the trivial local system.
\item Equip $\beta_{\infty}$ with $(E_{\infty}^{u},\phi_{\infty}^{u},A_{\infty})$,
where $E_{\infty}^{u}$ is a free, rank $p$ $\mathbb{F}\llbracket U\rrbracket$-module
$E_{\infty}^{u}:=\bigoplus_{i=0}^{p-1}y_{i}\mathbb{F}\llbracket U\rrbracket$
and the monodromy is
\[
\phi_{\infty}^{u}=\sum_{i=0}^{p-1}u_{i}y_{i+1}y_{i}^{\ast}:E_{\infty}^{u}\to E_{\infty}^{u}.
\]
Here, the indices of $y$ are interpreted modulo $p$.
\end{enumerate}

\subsection{The statement}

Let $p,q\in\mathbb{Z}_{>0}$ be coprime, $r=p/q$, and let $k=0,\cdots,p-1$.
\begin{defn}
\label{def:si}Let $s_{0},\cdots,s_{p-1}$ be the unique weakly decreasing
sequence with values $\left\lceil q/p\right\rceil ,\left\lfloor q/p\right\rfloor $,
such that $\sum_{i=0}^{p-1}s_{i}=q$, i.e.\ if we let $q=up+v$ for
$u,v\in\mathbb{Z}$, $v\in\{0,\cdots,p-1\}$, then 
\[
s_{0}=\cdots=s_{v-1}=u+1,\ s_{v}=\cdots=s_{p-1}=u.
\]
\end{defn}

We consider the local systems on the attaching curves from Subsection~\ref{subsec:Local-systems},
where $u$ and $c$ are as follows. See Section~\ref{sec:The-hat-version}
for the motivation behind this definition. Recall the definitions
of $a\bmod n,a^{-1}\bmod n\in\{0,\cdots,n-1\}$ from Subsection~\ref{subsec:Conventions}.
\begin{defn}[Local systems for Theorem~\ref{thm:gen-local-comp}]
\label{def:actual-local-system}Define sequences $u=(u_{0},\cdots,u_{p-1})$
and $c=(c_{0},\cdots,c_{p-1})$ as follows:
\begin{equation}
u_{(-iq^{-1}-1)\bmod p}=\begin{cases}
U & {\rm if}\ i=0,\cdots,k-1\\
1 & {\rm otherwise}
\end{cases},\ c_{i}=s_{(i-k)\bmod p}.\label{eq:uc}
\end{equation}
Let $(E_{0}^{k},\phi_{0}^{k},A_{0})$ and $(E_{\infty}^{k},\phi_{\infty}^{k},A_{\infty})$
be the corresponding local systems on $\beta_{0}$ and $\beta_{\infty}$,
respectively, defined as in Subsection~\ref{subsec:Local-systems}.
We suppress $p$ and $q$ from the notation for simplicity.
\end{defn}

The following is our main theorem, which we prove in Sections~\ref{sec:Local-systems-and},~\ref{sec:The-hat-version},~and~\ref{sec:The-minus-version}.
See Sections~\ref{sec:Local-systems-and}~and~\ref{sec:The-hat-version}
for the motivation behind the definitions of the cycles $\psi_{0r}^{k},\psi_{r\infty}^{k},\psi_{\infty0}^{k}$
and Equation~(\ref{eq:uvw-modu}), respectively. Note that the indices
of $s$, $y$, and $\theta$ are interpreted modulo $p$, the indices
of $v$ and $\xi$ are interpreted modulo $q$, and the indices of
$u$, $w$, and $x$ lie in the index set $I$ (Definition~\ref{def:index-set}).
\begin{thm}[Main local computation]
\label{thm:gen-local-comp}Let $p$ and $q$ be coprime positive
integers. There exist elements $u_{i,\ell},v_{j},w_{i,\ell}$ in $\mathbb{F}\llbracket U\rrbracket$
for $i=0,\cdots,p-1$, $\ell=0,\cdots,s_{i}-1$, $j=0,\cdots,q-1$,
such that 
\begin{equation}
u_{i,\ell},w_{i,\ell}\equiv\begin{cases}
1 & {\rm if}\ \ell=0\\
0 & {\rm otherwise}
\end{cases}\mod U,\ v_{j}\equiv\begin{cases}
1 & {\rm if}\ j\equiv-1,-2,\cdots,-\min(p,q)\mod q\\
0 & {\rm otherwise}
\end{cases}\mod U,\label{eq:uvw-modu}
\end{equation}
and if we define
\begin{gather*}
\psi_{0r}^{k}=\sum_{i=0}^{p-k-1}\sum_{\ell=0}^{s_{i}-1}u_{i,\ell}x_{i+k,\ell}^{\ast}\theta_{i+k}+\sum_{p-k}^{p-1}\sum_{\ell=0}^{s_{i}-1}u_{i,\ell}x_{i+k,\ell+1}^{\ast}\theta_{i+k}\in\boldsymbol{CF}^{-}(\beta_{0}^{E_{0}^{k}},\beta_{r}),\\
\psi_{r\infty}^{k}=\sum_{j=0}^{q-1}v_{j}y_{-((j+k)\bmod q)q^{-1}\bmod p}\xi_{(j+k)\bmod q}\in\boldsymbol{CF}^{-}(\beta_{r},\beta_{\infty}^{E_{\infty}^{k}}),\\
\psi_{\infty0}^{k}=\sum_{i=0}^{p-1}\sum_{\ell=0}^{s_{i}-1}w_{i,\ell}x_{i+k,\ell}y_{-(i+k)q^{-1}\bmod p}^{\ast}\zeta\in\boldsymbol{CF}^{-}(\beta_{\infty}^{E_{\infty}^{k}},\beta_{0}^{E_{0}^{k}}),
\end{gather*}
then the following hold:
\begin{enumerate}
\item \label{enu:bigon}The elements $\psi_{0r}^{k}$, $\psi_{r\infty}^{k}$,
and $\psi_{\infty0}^{k}$ are cycles.
\item \label{enu:triangle}The triangle counting maps are zero: 
\[
\mu_{2}(\psi_{\infty0}^{k}\otimes\psi_{0r}^{k})=0,\ \mu_{2}(\psi_{0r}^{k}\otimes\psi_{r\infty}^{k})=0,\ \mu_{2}(\psi_{r\infty}^{k}\otimes\psi_{\infty0}^{k})=0.
\]
\item \label{enu:quadrilateral}The quadrilateral counting maps are the
``identity'' modulo $U$:
\begin{gather*}
\mu_{3}(\psi_{0r}^{k}\otimes\psi_{r\infty}^{k}\otimes{\psi_{\infty0}^{k}}')\equiv{\rm Id}_{E_{0}^{k}}\Theta_{0}^{+}\mod U,\\
\mu_{3}(\psi_{r\infty}^{k}\otimes\psi_{\infty0}^{k}\otimes{\psi_{0r}^{k}}')\equiv\Theta_{r}^{+}\mod U,\\
\mu_{3}(\psi_{\infty0}^{k}\otimes\psi_{0r}^{k}\otimes{\psi_{r\infty}^{k}}')\equiv{\rm Id}_{E_{\infty}^{k}}\Theta_{\infty}^{+}\mod U.
\end{gather*}
Here, $\psi_{\infty0}^{k}{}'\in\boldsymbol{CF}^{-}(\beta_{\infty}^{E_{\infty}^{k}},{\beta_{0}'}^{E_{0}^{k}})$,
$\psi_{0r}^{k}{}'\in\boldsymbol{CF}^{-}(\beta_{0}^{E_{0}^{k}},\beta_{r}')$,
resp.\ $\psi_{r\infty}^{k}{}'\in\boldsymbol{CF}^{-}(\beta_{r},{\beta_{\infty}'}^{E_{\infty}^{k}})$
are the images of $\psi_{\infty0}^{k}$, $\psi_{0r}^{k}$, resp.\ $\psi_{r\infty}^{k}$
under the nearest point map (Subsection~\ref{subsec:Standard-translates}).
\end{enumerate}
Moreover, we can let either
\begin{equation}
u_{i,\ell}=\begin{cases}
1 & {\rm if}\ \ell=0\\
0 & {\rm otherwise}
\end{cases}\ \forall i,\ell,\ {\rm or}\ w_{i,\ell}=\begin{cases}
1 & {\rm if}\ \ell=0\\
0 & {\rm otherwise}
\end{cases}\ \forall i,\ell.\label{eq:options}
\end{equation}
\end{thm}

\begin{proof}
(\ref{enu:bigon}) holds for all $u_{i,\ell},v_{j},w_{i,\ell}$: in
fact $\mu_{1}\equiv0$, since there are no two-chains that can contribute
to $\mu_{1}$. Also, we can check directly that (\ref{enu:quadrilateral})
holds whenever Equation~(\ref{eq:uvw-modu}) is satisfied; we show
this in Subsection~\ref{subsec:The-quadrilateral-counting}. The
hardest part is showing that there exist $u_{i,\ell},v_{j},w_{i,\ell}$
such that (\ref{enu:triangle}) and Equation~(\ref{eq:uvw-modu})
are satisfied; we show this in Section~\ref{sec:The-minus-version}.
We also give a less algebra-heavy proof for the hat version in Section~\ref{sec:The-hat-version}.
\end{proof}
\begin{rem}[Case $p=1$ or $q=1$]
Of course, for any $u,v,w\in\mathbb{F}\llbracket U\rrbracket$ such
that $u,v,w\equiv1$~mod~$U$, if Theorem~\ref{thm:gen-local-comp}~(\ref{enu:bigon}),
(\ref{enu:triangle}), (\ref{enu:quadrilateral}), and Equation~(\ref{eq:uvw-modu})
hold for $u_{i,\ell}$, $v_{j}$, $w_{i,\ell}$, then they also hold
for $uu_{i,\ell}$, $vv_{j}$, $ww_{i,\ell}$. 

For the special case $(p,q,k)=(p,1,k)$, there are only three variables
$u_{0,0},v_{0},$ and $w_{0,0}$, and Equation~(\ref{eq:uvw-modu})
says that they are all $1$ mod $U$. Hence, by the above observation,
we can let
\[
u_{0,0}=v_{0}=w_{0,0}=1.
\]
Also, \cite[Section 9]{MR2113020} shows that for the special case
$(p,q,k)=(1,q,0)$, we can let 
\[
u_{0,\ell}=w_{0,\ell}=\begin{cases}
1 & {\rm if}\ \ell=0\\
0 & {\rm otherwise}
\end{cases},\ v_{j}=\begin{cases}
1 & {\rm if}\ j=q-1\\
0 & {\rm otherwise}
\end{cases}.
\]
\end{rem}

\begin{rem}[Case $(p,q)=(5,3)$ is not as simple]
\label{rem:pq53}Unlike the case $p=1$ or $q=1$, the elements $u_{i,\ell},v_{j},w_{i,\ell}$
have to be complicated in general. For instance, for $(p,q)=(5,3)$,
some of them must not be $0$ or $1$ (compare Remark~\ref{rem:minus-hard}):
we show in Examples~\ref{exa:61example}~and~\ref{exa:53uv} that
we must have
\begin{equation}
u_{0,0}w_{0,0}=tc,\ u_{1,0}w_{1,0}=tb,\ u_{2,0}w_{2,0}=tc\label{eq:36}
\end{equation}
for some unit $t\in\mathbb{F}\llbracket U\rrbracket$, and that $(u_{0,0},u_{1,0},u_{2,0})=(1,1,1)$,
$(v_{0},v_{1},v_{2})=(b,a,a)$, $(w_{0,0},w_{1,0},w_{2,0})=(c,b,c)$
works, where 
\begin{equation}
a:=\sum_{m\in\mathbb{Z}}U^{\frac{15m^{2}+27m+12}{2}},\ b:=\sum_{m\in\mathbb{Z}}(U^{\frac{15m^{2}+7m}{2}}+U^{\frac{15m^{2}+13m+2}{2}}),\ c:=\sum_{m\in\mathbb{Z}}U^{\frac{15m^{2}+25m+10}{2}}.\label{eq:abc}
\end{equation}
More concretely, for instance for $k=0$, the following work: 
\[
\psi_{0r}^{0}=x_{0,0}^{\ast}\theta_{0}+x_{1,0}^{\ast}\theta_{1}+x_{2,0}^{\ast}\theta_{2},\ \psi_{r\infty}^{0}=by_{0}\xi_{0}+ay_{3}\xi_{1}+ay_{1}\xi_{2},\ \psi_{\infty0}^{0}=cx_{0,0}y_{0}^{\ast}\zeta_{0}+bx_{1,0}y_{3}^{\ast}\zeta_{1}+cx_{2,0}y_{1}^{\ast}\zeta_{2}.
\]
\end{rem}

\begin{rem}[Simplification using Nakayama's lemma]
\label{rem:nakayama}If we do not care about identifying the three
maps in the exact triangles (Theorems~\ref{thm:rational-surgery-k=00003D0}~and~\ref{thm:rational-surgery-kgeneral}),
then by Corollary~\ref{cor:triangle-detection-nakayama}, the following
weaker version of Theorem~\ref{thm:gen-local-comp} is sufficient:
instead of Theorem~\ref{thm:gen-local-comp}~(\ref{enu:triangle}),
we only have that one of the $\mu_{2}$'s is zero and the other two
are zero modulo $U$. Interestingly, this does not significantly simplify
our solution to the combinatorial problem: for instance for $k=0$,
it reduces Subsubsections~\ref{subsec:triangle0rinf-rinf0}~and~\ref{subsec:triangleinf0r}
to only one of them, but they are essentially the same.
\end{rem}

\subsection{\label{subsec:inter-local}Recovering Theorems~\ref{thm:rational-surgery-k=00003D0}~and~\ref{thm:rational-surgery-kgeneral}}

As we mentioned earlier, Theorems~\ref{thm:rational-surgery-k=00003D0}~and~\ref{thm:rational-surgery-kgeneral}
follow from Theorem~\ref{thm:gen-local-comp} by standard arguments.
In this subsection, we explain a part of this, focusing on how the
local systems on $\beta_{0},\beta_{r},\beta_{\infty}$ correspond
to the three homology groups of the surgery exact triangles. We postpone
recalling the rest of the standard argument to Appendix~\ref{sec:Reduction-of-Theorem}.
Compare~\cite[Section~9]{MR2113020},~\cite[Section~12]{2011.00113}.

For this, we introduce the following definition.
\begin{defn}[Genus stabilization]
\label{def:genus-stabilization}A genus $g$ Heegaard diagram $(\Sigma,\boldsymbol{\alpha}_{0},\cdots,\boldsymbol{\alpha}_{m},z)$
is a \emph{genus stabilization of a genus $1$ Heegaard diagram $(\mathbb{T}^{2},\alpha_{0},\cdots,\alpha_{m},z)$}
if there exists a genus $(g-1)$ Heegaard diagram $(\overline{\Sigma},\overline{\boldsymbol{\alpha}}_{0},\cdots,\overline{\boldsymbol{\alpha}}_{m},\overline{z})$
such that (A) $\overline{\boldsymbol{\alpha}}_{0},\cdots,\overline{\boldsymbol{\alpha}}_{m}$
are pairwise standard translates, (B) $\Sigma=\mathbb{T}^{2}\#\overline{\Sigma}$
is obtained by connected summing $\mathbb{T}^{2}$ and $\overline{\Sigma}$
along points near $z$ and $\overline{z}$, respectively, and (C)
$\boldsymbol{\alpha}_{i}=\alpha_{i}\cup\overline{\boldsymbol{\alpha}}_{i}$
for all $i$.

Denote the intersection point of $\mathbb{T}_{\overline{\boldsymbol{\alpha}}_{i}}\cap\mathbb{T}_{\overline{\boldsymbol{\alpha}}_{j}}$
in the top homological grading as $\Theta_{ij}^{+}$. If $x\in\alpha_{i}\cap\alpha_{j}$,
then define its \emph{stabilization} $S(x)\in\mathbb{T}_{\boldsymbol{\alpha}_{i}}\cap\mathbb{T}_{\boldsymbol{\alpha}_{j}}$
as the intersection point given by $x$ and $\Theta_{ij}^{+}$.

If the attaching curve $\alpha_{i}$ is equipped with a local system
$(E_{i},\phi_{i},A_{i})$, then equip $\boldsymbol{\alpha}_{i}$ with
the same local system $(E_{i},\phi_{i},A_{i})$. Let 
\[
S:\boldsymbol{CF}^{-}(\alpha_{i}^{E_{i}},\alpha_{j}^{E_{j}})\to\boldsymbol{CF}^{-}(\boldsymbol{\alpha}_{i}^{E_{i}},\boldsymbol{\alpha}_{j}^{E_{j}})
\]
be the $\mathbb{F}\llbracket U\rrbracket$-linear map induced by the
above stabilization map. Note that $S$ is in general not a chain
map.
\end{defn}

Let $K$ be a knot in a closed, oriented $3$-manifold $Y$ with framing
$\lambda$, and let $\mu$ be the meridian. Then, as in \cite[Lemma~9.2]{MR2113020},
we can find a weakly admissible, genus $g$ Heegaard diagram $(\Sigma,\boldsymbol{\alpha},\boldsymbol{\beta}_{0},\boldsymbol{\beta}_{r},\boldsymbol{\beta}_{\infty},z)$
such that the following conditions are satisfied:
\begin{enumerate}
\item $(\Sigma,\boldsymbol{\beta}_{0},\boldsymbol{\beta}_{r},\boldsymbol{\beta}_{\infty},z)$
is a genus stabilization of the Heegaard diagram $(\mathbb{T}^{2},\beta_{0},\beta_{r},\beta_{\infty},z)$.
\item Consider the oriented arc $A_{0}\subset\Sigma$. Recall that the initial
point of $A_{0}$ is $z$; let $w_{0}\in\Sigma$ be the terminal point
of $A_{0}$. Then, $\boldsymbol{\alpha}$ is disjoint from $A_{0}$,
and the doubly pointed Heegaard diagram $(\Sigma,\boldsymbol{\alpha},\boldsymbol{\beta}_{0},w_{0},z)$
represents $(Y_{\lambda},K_{\lambda})$ where $K_{\lambda}\subset Y_{\lambda}$
is the dual knot of $K$.
\item The Heegaard diagram $(\Sigma,\boldsymbol{\alpha},\boldsymbol{\beta}_{r},z)$
represents $Y_{p\mu+q\lambda}(K)$.
\item \label{enu:infty}Consider the oriented arc $A_{\infty}\subset\Sigma$.
Recall that the initial point of $A_{\infty}$ is $z$; let $w\in\Sigma$
be the terminal point of $A_{\infty}$. Then, $\boldsymbol{\alpha}$
is disjoint from $A_{\infty}$, and the doubly pointed Heegaard diagram
$(\Sigma,\boldsymbol{\alpha},\boldsymbol{\beta}_{\infty},w,z)$ represents
$(Y,K)$.
\end{enumerate}
Let $p$ and $q$ be coprime positive integers, and let $k=0,\cdots,p-1$.
As in Definition~\ref{def:genus-stabilization}, equip $\boldsymbol{\beta}_{0}$
with the local system $(E_{0}^{k},\phi_{0}^{k},A_{0})$ and equip
$\boldsymbol{\beta}_{\infty}$ with the local system $(E_{\infty}^{k},\phi_{\infty}^{k},A_{\infty})$.
Then, by Lemma~\ref{lem:local-system-simple-lemma}, the above conditions
imply 
\begin{gather*}
\boldsymbol{CF}^{-}(\boldsymbol{\alpha},\boldsymbol{\beta}_{0}^{E_{0}^{k}})\simeq\bigoplus_{i=0}^{\min(p,q)-1}\underline{\boldsymbol{CF}^{-}}(Y_{\lambda}(K);\mathbb{F}[\mathbb{Z}/s_{i}\mathbb{Z}]),\\
\boldsymbol{CF}^{-}(\boldsymbol{\alpha},\boldsymbol{\beta}_{r})\simeq\boldsymbol{CF}^{-}(Y_{p\mu+q\lambda}(K)),\\
\boldsymbol{CF}^{-}(\boldsymbol{\alpha},\boldsymbol{\beta}_{\infty}^{E_{\infty}^{k}})\simeq{\cal CFK}^{-}(Y,K)\otimes_{\mathbb{F}\llbracket W,Z\rrbracket}E_{\infty}^{k},
\end{gather*}
where the twist of $\underline{\boldsymbol{CF}^{-}}(Y_{\lambda}(K);\mathbb{F}[\mathbb{Z}/s_{i}\mathbb{Z}])$
is given by the homology class of the dual knot of $K$, and we view
$E_{\infty}^{k}$ as an $\mathbb{F}\llbracket W,Z\rrbracket$-module
by letting $W$ act as $\phi_{\infty}^{k}$ and $Z$ act as $U(\phi_{\infty}^{k})^{-1}$.
\begin{rem}[Recover $\boldsymbol{HFK}_{p,q,k}^{-}(Y,K)$]
\label{rem:recover}Indeed, ${\cal CFK}^{-}(Y,K)\otimes_{\mathbb{F}\llbracket W,Z\rrbracket}E_{\infty}^{k}\simeq\boldsymbol{CFK}_{p,q,k}^{-}(Y,K)$
since $E_{\infty}^{k}$ is isomorphic to $E_{p,q,k}$ from Definition~\ref{def:hfkpqk}
as $\mathbb{F}\llbracket W,Z\rrbracket$-modules: let $\Phi:E_{\infty}^{k}\to E_{p,q,k}$
be the $\mathbb{F}\llbracket U\rrbracket$-module isomorphism given
by $\Phi(y_{i})=e_{qi+k-1}$. Then, we can check that
\[
\Phi\phi_{\infty}^{k}\Phi^{-1}=\sum_{i=0}^{p-1}u_{q^{-1}(i-k+1)}e_{i+q}e_{i}^{\ast},\ \mathrm{and}\ u_{q^{-1}(i-k+1)}=\begin{cases}
U & {\rm if}\ i+q\in\{0,\cdots,k-1\}\\
1 & {\rm otherwise}
\end{cases}.
\]
\end{rem}

We will show a slightly stronger statement than Theorems~\ref{thm:rational-surgery-k=00003D0}~and~\ref{thm:rational-surgery-kgeneral}:
we additionally identify the maps in the exact triangles. Indeed,
we show in Appendix~\ref{sec:Reduction-of-Theorem} that the stabilizations
of $\psi_{0r}^{k},\psi_{r\infty}^{k},\psi_{\infty0}^{k}$ are cycles
that the following triangle is exact: 
\begin{equation}\label{eq:exact-triangle34}
\begin{tikzcd}[ampersand replacement=\&]
	{\boldsymbol{HF}^{-}(\boldsymbol{\alpha},\boldsymbol{\beta}_{0}^{E_0 ^k} )} \&\& {\boldsymbol{HF}^{-}(\boldsymbol{\alpha},\boldsymbol{\beta}_{r})} \\
	\& {\boldsymbol{HF}^{-}(\boldsymbol{\alpha},\boldsymbol{\beta}_{\infty}^{E_\infty ^k})}
	\arrow["{\mu_{2}(-\otimes S(\psi_{0r}^k))}", from=1-1, to=1-3]
	\arrow["{\mu_{2}(-\otimes S(\psi_{r\infty }^k))}"{description}, from=1-3, to=2-2]
	\arrow["{\mu_{2}(-\otimes S(\psi_{\infty 0}^k))}"{description}, from=2-2, to=1-1]
\end{tikzcd}
\end{equation}

Finally, we record the following lemma which was used in Remark~\ref{rem:hfkpqk-2}.
\begin{lem}[$\boldsymbol{HFK}_{p,q,k}^{-}$ as the $2k/p$-modified knot Floer
homology over a submodule $R\subset\mathbb{F}\llbracket U^{1/p}\rrbracket$,
when $(p,k)=1$]
\label{lem:interpret-local-system}Let $p$ and $q$ be coprime positive
integers, and let $k=0,\cdots,p-1$ be such that $p$ and $k$ are
coprime. View $\mathbb{F}\llbracket U^{1/p}\rrbracket$ as an $\mathbb{F}\llbracket W,Z\rrbracket$-module
by letting $W$ act as $U^{k/p}$ and $Z$ act as $U^{1-(k/p)}$.
Then, there exists an $\mathbb{F}\llbracket W,Z\rrbracket$-submodule
$R\subset\mathbb{F}\llbracket U^{1/p}\rrbracket$ that is isomorphic
to $E_{\infty}^{k}$ as an $\mathbb{F}\llbracket W,Z\rrbracket$-module.
\end{lem}

\begin{proof}
Let us define a sequence $\{m_{i}\}_{i\in\mathbb{Z}/p\mathbb{Z}}$.
First, we recursively define $\{m_{i}\}_{i\in\mathbb{Z}}$ and show
that $m_{p+i}=m_{i}$. Let $m_{0}=0$, and let 
\[
m_{i+1}:=\begin{cases}
m_{i}+\frac{k}{p} & {\rm if}\ u_{i}=1\\
m_{i}+\frac{k}{p}-1 & {\rm if}\ u_{i}=U
\end{cases}.
\]
Then since exactly $k$ of $u_{0},\cdots,u_{p-1}$ are $U$, we have
$m_{p+i}=m_{i}$ for all $i$. Also note that since $p$ and $k$
are coprime, $m_{0},\cdots,m_{p-1}$ are pairwise distinct modulo
$1$.

Let $m_{n}$ be the minimum of $m_{0},\cdots,m_{p-1}$, and define
$m_{i}'=m_{i}-m_{n}$. Let $R\subset\mathbb{F}\llbracket U^{1/p}\rrbracket$
be the free, rank $p$ $\mathbb{F}\llbracket U\rrbracket$-submodule
with basis $U^{m_{0}'},\cdots,U^{m_{p-1}'}$: indeed, they are linearly
independent since the $m_{i}'$'s modulo $1$ are pairwise distinct.
Let $f:E_{\infty}^{k}\to R$ be the $\mathbb{F}\llbracket U\rrbracket$-linear
isomorphism given by $y_{i}\mapsto U^{m_{i}'}$. Then it is routine
to check that $f$ is $\mathbb{F}\llbracket W,Z\rrbracket$-linear.
\end{proof}

\section{\label{sec:Local-systems-and}Local systems and covering spaces}

Our main theorem, Theorem~\ref{thm:gen-local-comp}, is a statement
about the polygon counting maps $\mu_{1}$, $\mu_{2}$, and $\mu_{3}$
of attaching curves $\beta_{0}^{E_{0}},\beta_{r},\beta_{\infty}^{E_{\infty}^{u}}$
and their standard translates equipped with local systems in $\mathbb{T}^{2}$.
Note that the cycles $\psi_{0r}^{k},\psi_{r\infty}^{k},\psi_{\infty0}^{k}$
only involve some of the basis elements $x_{a,b}^{\ast}\theta_{i},y_{c}\xi_{j},x_{a,b}y_{c}^{\ast}\zeta$
of the chain complexes; we call these particular basis elements \emph{standard}
(see Definition~\ref{def:standard-basis} for a general definition).

The main goal of this section is to motivate our definition of standard
basis elements, and hence our choice of the cycles $\psi_{0r}^{k},\psi_{r\infty}^{k},\psi_{\infty0}^{k}$.
The key point is that if we only consider standard basis elements,
then the polygon counting maps $\mu_{2}$ and $\mu_{3}$ can be nicely
described in terms of the corresponding polygon counting maps of attaching
curves without local systems on the cover $\widetilde{\mathbb{T}^{2}}:=\mathbb{R}^{2}/(0,p)\mathbb{Z}$
of $\mathbb{T}^{2}$ (Propositions~\ref{prop:lift-mu2}~and~\ref{prop:main-mu3});
compare \cite[Section 3]{MR1957829}.

In this section, we fix sequences $u=(u_{0},\cdots,u_{p-1})$ and
$c=(c_{0},\cdots,c_{p-1})$, and work in the setting of Subsection~\ref{subsec:Local-systems}.

\subsection{\label{subsec:cover}A cover of the genus $1$ Heegaard diagram}

\begin{figure}[h]
\begin{centering}
\raisebox{-0.5\height}{\includegraphics[scale=0.5]{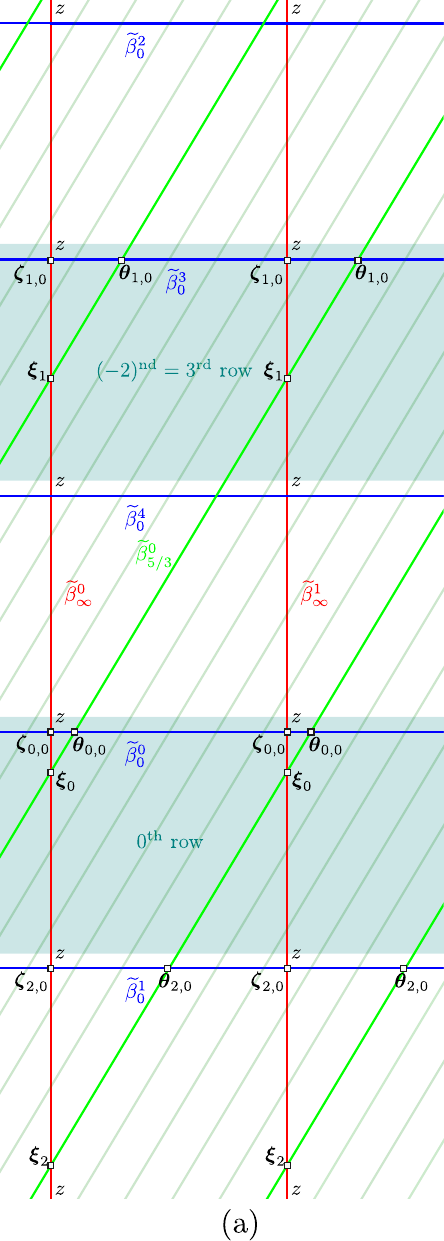}}\quad{}\raisebox{-0.5\height}{\includegraphics[scale=0.5]{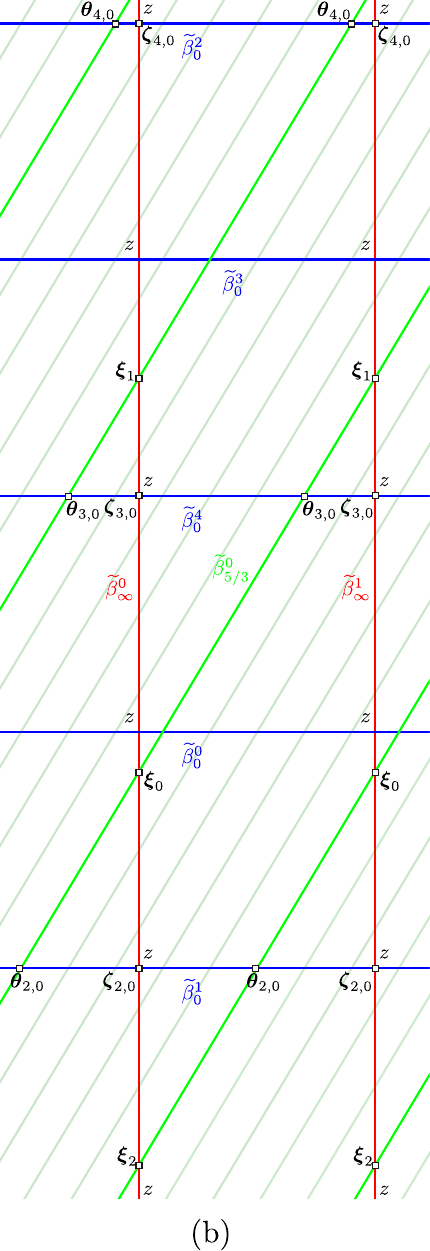}}\quad{}\raisebox{-0.5\height}{\includegraphics[scale=0.5]{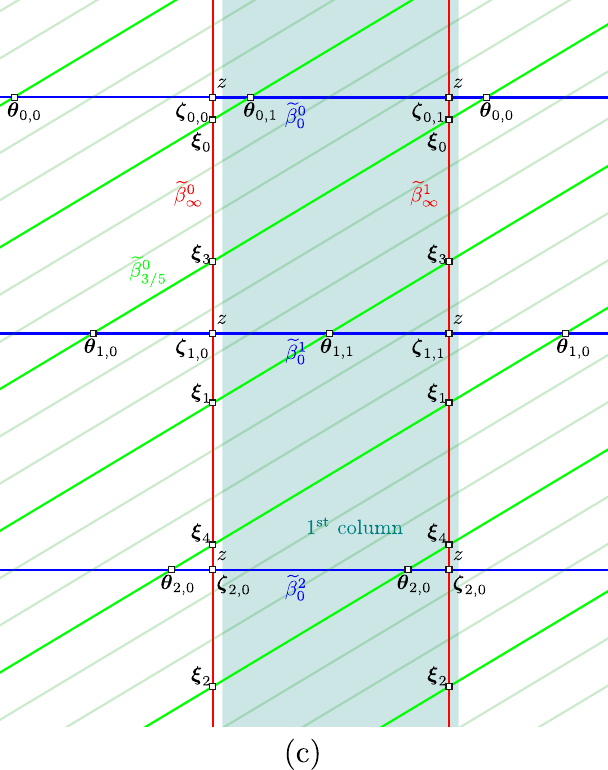}}
\par\end{centering}
\caption{\label{fig:tildet2}Some examples for the cover $\widetilde{\mathbb{T}^{2}}$
of the genus $1$ Heegaard diagram. See Subsection~\ref{subsec:An-explicit-formula}
for the meaning of $\boldsymbol{\theta}_{i,\ell}$, $\boldsymbol{\xi}_{j}$,
and $\boldsymbol{\zeta}_{i,\ell}$. The intersection point labelled
$\boldsymbol{\zeta}_{0,0}$ is the origin. In all three cases, the
lifts of standard basis elements are labelled, and the set $\widetilde{z(u)}$
consists of the points $z$ in the figure. (a): $(p,q,k)=(5,3,0)$,
$u=(1,1,1,1,1)$, $c=(1,1,1,0,0)$. (b): $(p,q,k)=(5,3,2)$, $u=(1,1,U,1,U)$,
$c=(0,0,1,1,1)$. (c): $(p,q,k)=(3,5,0)$, $u=(1,1,1)$, $c=(2,2,1)$.}
\end{figure}

We consider the covering space $\pi:\widetilde{\mathbb{T}^{2}}:=\mathbb{R}^{2}/(0,p)\mathbb{Z}\to\mathbb{T}^{2}$.
We introduce some notions to describe objects in $\widetilde{\mathbb{T}^{2}}$:
see Figure~\ref{fig:tildet2} for some examples.

Recall that the basepoint $z\in\mathbb{T}^{2}$ is at $(\varepsilon,\varepsilon)$
for small $\varepsilon>0$. Let the \emph{$a$th column }be
\[
]0,1[\times\mathbb{R}+(a+\varepsilon-1,0),
\]
let the \emph{$b$th row} ($b\in\mathbb{Z}/p\mathbb{Z}$) be
\[
\mathbb{R}\times]0,1[+(0,-b+\varepsilon-1),
\]
and let the \emph{$(a,b)$th square} be the intersection of the $a$th
column and the $b$th row, i.e.\ 
\[
]0,1[^{2}+(a+\varepsilon-1,-b+\varepsilon-1).
\]

Let $\widetilde{\beta}_{0}^{b}$ be the connected component of $\pi^{-1}(\beta_{0})$
that is in the $b$th row, let $\widetilde{\beta}_{r}^{i}$ ($i=0,\cdots,p-1$)
be the union of the connected components of $\pi^{-1}(\beta_{r})$
that intersect $\pi^{-1}(\theta_{i})$ in the $0$th row, and let
$\widetilde{\beta}_{\infty}^{a}$ be the connected component of $\pi^{-1}(\beta_{\infty})$
that is in the $a$th column.

Define $\widetilde{z(u)}\subset\widetilde{\mathbb{T}^{2}}$ as 
\begin{align*}
\widetilde{z(u)} & :=\{(a+\varepsilon,-b+\varepsilon-1):a\in\mathbb{Z},\ b\in\mathbb{Z}/p\mathbb{Z},\ u_{b}=1\}\\
 & \cup\{(a-\varepsilon,-b+\varepsilon-1):a\in\mathbb{Z},\ b\in\mathbb{Z}/p\mathbb{Z},\ u_{b}=U\}.
\end{align*}
Note that if $u_{i}=1$ for all $i$, then $\widetilde{z(u)}=\pi^{-1}(z)$.

\subsection{\label{subsec:Lifts-and-standard-mu2}Lifts of standard basis elements
and triangle counting maps}

We want to state the main result of this subsection, Proposition~\ref{prop:lift-mu2},
for each cyclic permutation of $\beta_{0}^{E_{0}},\beta_{r}^{\mathbb{F}\llbracket U\rrbracket},\beta_{\infty}^{E_{\infty}^{u}}$.
Hence, for simplicity of notation, let $\gamma_{0}^{F_{0}},\gamma_{1}^{F_{1}},\gamma_{2}^{F_{2}}$
be a cyclic permutation of $\beta_{0}^{E_{0}},\beta_{r}^{\mathbb{F}\llbracket U\rrbracket},\beta_{\infty}^{E_{\infty}^{u}}$.
In this subsection, the indices of $\gamma$ and $F$ are interpreted
mod~$3$.

Recall that we considered the basis $\{x_{i,\ell}\}$ for $E_{0}$
and the basis $\{y_{j}\}$ for $E_{\infty}^{u}$. These bases induce
a basis for ${\rm Hom}_{\mathbb{F}\llbracket U\rrbracket}(F_{i},F_{j})$,
and hence a basis for $\boldsymbol{CF}^{-}(\gamma_{i}^{F_{i}},\gamma_{j}^{F_{j}})$:
more concretely, say $\boldsymbol{b}$ is a \emph{basis element} of
$\boldsymbol{CF}^{-}(\gamma_{i}^{F_{i}},\gamma_{j}^{F_{j}})$ if $\boldsymbol{b}=ea$
for some basis element $e$ of ${\rm Hom}_{\mathbb{F}\llbracket U\rrbracket}(F_{i},F_{j})$
and $a\in\gamma_{i}\cap\gamma_{j}$.

We define lifts of basis elements such that Proposition~\ref{prop:lift-be}
holds.
\begin{defn}[Lifts of basis elements]
\label{def:lift-of-basis}Let $ea$ be a basis element of $\boldsymbol{CF}^{-}(\gamma_{i}^{F_{i}},\gamma_{j}^{F_{j}})$
($i\neq j$), i.e.\ $e$ is a basis element of ${\rm Hom}_{\mathbb{F}\llbracket U\rrbracket}(F_{i},F_{j})$
and $a\in\gamma_{i}\cap\gamma_{j}$. Let us define what it means for
a point $\widetilde{a}\in\pi^{-1}(a)$ to be a \emph{lift }of $ea$.
\begin{enumerate}
\item Let $\theta\in\beta_{0}\cap\beta_{r}$. A point $\widetilde{\theta}\in\pi^{-1}(\theta)$
is a \emph{lift} of $x_{i,\ell}^{\ast}\theta$ or $x_{i,\ell}\theta$
if it lies in the $(a,b)$th square for 
\[
a\equiv\ell\mod c_{i},\ b\equiv-iq^{-1}\mod p.
\]
\item Let $\xi\in\beta_{r}\cap\beta_{\infty}$. A point $\widetilde{\xi}\in\pi^{-1}(\xi)$
is a \emph{lift} of $y_{i'}\xi$ or $y_{i'}^{\ast}\xi$ if it lies
in the $b$th row for 
\[
b\equiv i'\mod p.
\]
\item A point $\widetilde{\zeta}\in\pi^{-1}(\zeta)$ is a \emph{lift} of
$x_{i,\ell}y_{i'}^{\ast}\zeta$ or $y_{i'}x_{i,\ell}^{\ast}\zeta$
if it lies in the $(a,b)$th square for
\[
a\equiv\ell\mod c_{i},\ b\equiv i'\equiv-iq^{-1}\mod p.
\]
In particular, there is no lift of $x_{i,\ell}y_{i'}^{\ast}\zeta$
or $y_{i'}x_{i,\ell}^{\ast}\zeta$ if $i'\not\equiv-iq^{-1}$ modulo
$p$. 
\end{enumerate}
\end{defn}

Recall that $\mu_{2}$ counts holomorphic triangles with Maslov index
$0$. We study $\mu_{2}$ by studying the \emph{contribution ${\rm Cont}(T)$
of each triangle $T$} (Definition~\ref{def:contribution-triangle})
separately. Indeed, $\mu_{2}$ is the sum of ${\rm Cont}(T)$ along
all Maslov index $0$ triangles $T$. From now on, triangles refer
to Maslov index $0$ triangles. 
\begin{defn}
\label{def:contribution-triangle}If $T$ is a triangle with vertices
$c_{t}\in\gamma_{t}\cap\gamma_{t+1}$ ($t=0,1,2$), define the \emph{contribution
of $T$ to $\mu_{2}$} as
\begin{multline*}
{\rm Cont}(T):\boldsymbol{CF}^{-}(\gamma_{0}^{F_{0}},\gamma_{1}^{F_{1}})\otimes\boldsymbol{CF}^{-}(\gamma_{1}^{F_{1}},\gamma_{2}^{F_{2}})\to\boldsymbol{CF}^{-}(\gamma_{0}^{F_{0}},\gamma_{2}^{F_{2}}):\\
e_{1}a_{1}\otimes e_{2}a_{2}\mapsto\begin{cases}
U^{n_{z}(T)}\rho(T)(e_{1}\otimes e_{2})c_{3} & {\rm if}\ (a_{1},a_{2})=(c_{1},c_{2})\\
0 & {\rm otherwise}
\end{cases}
\end{multline*}
where $e_{1}\in{\rm Hom}(F_{0},F_{1})$, $e_{2}\in{\rm Hom}(F_{1},F_{2})$,
$a_{1}\in\gamma_{0}\cap\gamma_{1}$, and $a_{2}\in\gamma_{1}\cap\gamma_{2}$.
\end{defn}

\begin{defn}
\label{def:triangle-lift}Let $T$ be a triangle in $\mathbb{T}^{2}$.
A triangle $\widetilde{T}$ in $\widetilde{\mathbb{T}^{2}}$ is a
\emph{lift} of $T$ if it is topologically a lift of the triangle
$T$ to the covering space $\pi:\widetilde{\mathbb{T}^{2}}\to\mathbb{T}^{2}$.
Let $\boldsymbol{b}_{t}$ be a basis element of $\boldsymbol{CF}^{-}(\gamma_{t}^{F_{t}},\gamma_{t+1}^{F_{t+1}})$
for $t=0,1$, and let $\boldsymbol{b}_{2}$ be a basis element of
$\boldsymbol{CF}^{-}(\gamma_{0}^{F_{0}},\gamma_{2}^{F_{2}})$. A lift
$\widetilde{T}$ of $T$ is a\emph{ $\{\boldsymbol{b}_{0},\boldsymbol{b}_{1},\boldsymbol{b}_{2}\}$-lift
of $T$} if its three vertices are lifts of $\boldsymbol{b}_{0}$,
$\boldsymbol{b}_{1}$, and $\boldsymbol{b}_{2}$.
\end{defn}

\begin{prop}
\label{prop:lift-be}Let $T$ be a triangle in $\mathbb{T}^{2}$,
 let $\boldsymbol{b}_{t}$ be a basis element of $\boldsymbol{CF}^{-}(\gamma_{t}^{F_{t}},\gamma_{t+1}^{F_{t+1}})$
for $t=0,1$, and let $\boldsymbol{b}_{2}$ be a basis element of
$\boldsymbol{CF}^{-}(\gamma_{0}^{F_{0}},\gamma_{2}^{F_{2}})$. If
$T$ has a $\{\boldsymbol{b}_{0},\boldsymbol{b}_{1},\boldsymbol{b}_{2}\}$-lift,
then 
\[
{\rm Cont}(T)(\boldsymbol{b}_{0}\otimes\boldsymbol{b}_{1})=U^{n_{\widetilde{z(u)}}(\widetilde{T})}\boldsymbol{b}_{2}.
\]
\end{prop}

\begin{proof}
Direct from the definitions. Note that in our Heegaard diagram, every
triangle with Maslov index $0$ has a unique holomorphic representative.
\end{proof}
Now, the key observation is that if we only consider the case where
$\boldsymbol{b}_{0}$ and $\boldsymbol{b}_{1}$ are standard (Definition~\ref{def:standard-basis}),
then we can upgrade Proposition~\ref{prop:lift-be} to Proposition~\ref{prop:lift-mu2}.

\begin{defn}[Standard basis element]
\label{def:standard-basis}Define \emph{standard basis elements }of
$\boldsymbol{CF}^{-}(\gamma_{i}^{F_{i}},\gamma_{j}^{F_{j}})$ ($i\neq j$)
as follows.
\begin{enumerate}
\item For $\theta\in\beta_{0}\cap\beta_{r}$, the basis element $x_{i,\ell}^{\ast}\theta$
or $x_{i,\ell}\theta$ is \emph{standard} if all its lifts are in
$\widetilde{\beta}_{r}^{0}$.
\item For $\xi\in\beta_{r}\cap\beta_{\infty}$, the basis element $y_{i'}\xi_{j}$
or $y_{i'}^{\ast}\xi_{j}$ is \emph{standard} if all its lifts are
in~$\widetilde{\beta}_{r}^{0}$.
\item The basis element $x_{i,\ell}y_{i'}^{\ast}\zeta$ or $y_{i'}x_{i,\ell}^{\ast}\zeta$
is \emph{standard} if it admits a lift.
\end{enumerate}
\end{defn}

Note that if $\boldsymbol{b}$ is a basis element, then there exists
a lift of $\boldsymbol{b}$ that is in $\widetilde{\beta}_{r}^{0}$
if and only if every lift of $\boldsymbol{b}$ is in $\widetilde{\beta}_{r}^{0}$,
since if $\widetilde{b_{0}}$ and $\widetilde{b}_{1}$ are lifts of
$\boldsymbol{b}$, then $\widetilde{b}_{1}=\widetilde{b}_{0}+(n,0)$
for some $n\in\mathbb{Z}$, and $\widetilde{\beta}_{r}^{0}=\widetilde{\beta}_{r}^{0}+(n,0)$
for all $n\in\mathbb{Z}$.

Before we prove Proposition~\ref{prop:lift-mu2}, let us record a
few lemmas.
\begin{defn}
Say a basis element $\boldsymbol{b}$ of $\boldsymbol{CF}^{-}(\gamma_{i}^{F_{i}},\gamma_{j}^{F_{j}})$
\emph{comes from an intersection point }$a\in\gamma_{i}\cap\gamma_{j}$
if $\boldsymbol{b}=ea$ for some basis element $e$ of ${\rm Hom}_{\mathbb{F}\llbracket U\rrbracket}(F_{i},F_{j})$.
\end{defn}

\begin{lem}
\label{lem:theta-unique}For each $\theta_{i}\in\beta_{0}\cap\beta_{r}$,
there exists a standard basis element that comes from $\theta_{i}$
if and only if $c_{i}\neq0$. If $c_{i}\neq0$, the standard basis
elements that come from $\theta_{i}$ are $x_{i,\ell}^{\ast}\theta_{i}$
and $x_{i,\ell}\theta_{i}$.
\end{lem}

\begin{proof}
Since $\beta_{r}^{0}\cap\beta_{0}^{0}=\pi^{-1}(\theta_{0})\cap\beta_{0}^{0}$,
we have $\beta_{r}^{0}\cap\beta_{0}^{b}=\pi^{-1}(\theta_{-bq})\cap\beta_{0}^{b}$
for all $b\in\mathbb{Z}$.
\end{proof}
\begin{lem}
\label{lem:unique}For each $\xi_{j}\in\beta_{r}\cap\beta_{\infty}$,
$i'\equiv-jq^{-1}\bmod p$ is the unique $i'=0,\cdots,p-1$ such that
$y_{i'}\xi_{j}$ (or $y_{i'}^{\ast}\xi_{j}$) is standard.
\end{lem}

\begin{proof}
Let $a\in\mathbb{Z}$, $b\in\mathbb{Z}/p\mathbb{Z}$, and let $\widetilde{\theta}\in\pi^{-1}(\theta_{i})\cap\widetilde{\beta}_{0}^{0}$
be such that it is in the $a$th column, and let $\widetilde{\xi}\in\pi^{-1}(\xi_{j})\cap\widetilde{\beta}_{\infty}^{0}$
be such that it is in the $b$th row. Then, $\widetilde{\theta}$
and $\widetilde{\xi}$ are on the same connected component of $\pi^{-1}(\beta_{r})$
if and only if $ap+i-p\equiv bq+j\bmod pq$. Hence, $\widetilde{\xi}$
is standard if and only if $bq+j\equiv0\bmod pq$, which is equivalent
to $b\equiv-jq^{-1}\bmod p$.
\end{proof}
\begin{lem}
\label{lem:Let--be}Let $\widetilde{T}$ be a triangle in $\widetilde{\mathbb{T}^{2}}$
with vertices $\widetilde{\theta}\in\pi^{-1}(\beta_{0}\cap\beta_{r})$,
$\widetilde{\xi}\in\pi^{-1}(\beta_{r}\cap\beta_{\infty})$, and $\widetilde{\zeta}\in\pi^{-1}(\beta_{\infty}\cap\beta_{0})$.
The vertex $\widetilde{\theta}$ is a lift of a standard basis element
if and only if both $\widetilde{\xi}$ and $\widetilde{\zeta}$ are
lifts of standard basis elements.
\end{lem}

\begin{proof}
This follows from the following three statements.
\begin{enumerate}
\item A point $\widetilde{\theta}\in\pi^{-1}(\beta_{0}\cap\beta_{r})$ is
a lift of a standard basis element if and only if $\widetilde{\theta}\in\widetilde{\beta}_{r}^{0}$
and whenever $i\in\mathbb{Z}$ is such that $\widetilde{\theta}\in\widetilde{\beta}_{0}^{(-iq^{-1}\bmod p)}$,
we have $c_{i}\neq0$.
\item A point $\widetilde{\xi}\in\pi^{-1}(\beta_{r}\cap\beta_{\infty})$
is a lift of a standard basis element if and only if $\widetilde{\xi}\in\widetilde{\beta}_{r}^{0}$.
\item A point $\widetilde{\zeta}\in\pi^{-1}(\beta_{\infty}\cap\beta_{0})$
is a lift of a standard basis element if and only if whenever $i\in\mathbb{Z}$
is such that $\widetilde{\zeta}\in\widetilde{\beta}_{0}^{(-iq^{-1}\bmod p)}$,
we have $c_{i}\neq0$.\qedhere
\end{enumerate}
\end{proof}
\begin{prop}[Main proposition for $\mu_{2}$]
\label{prop:lift-mu2}Let $T$ be a triangle in $\mathbb{T}^{2}$,
and let $\boldsymbol{b}_{t}$ be a standard basis element of $\boldsymbol{CF}^{-}(\gamma_{t}^{F_{t}},\gamma_{t+1}^{F_{t+1}})$
for $t=0,1$. If $T$ does not admit a $\{\boldsymbol{b}_{0},\boldsymbol{b}_{1},\boldsymbol{b}_{2}\}$-lift
$\widetilde{T}$ for any standard basis element $\boldsymbol{b}_{2}$
of $\boldsymbol{CF}^{-}(\gamma_{0}^{F_{0}},\gamma_{2}^{F_{2}})$,
then ${\rm Cont}(T)(\boldsymbol{b}_{0}\otimes\boldsymbol{b}_{1})=0$.
Otherwise, there is at most one such $\boldsymbol{b}_{2}$, and 
\[
{\rm Cont}(T)(\boldsymbol{b}_{0}\otimes\boldsymbol{b}_{1})=U^{n_{\widetilde{z(u)}}(\widetilde{T})}\boldsymbol{b}_{2}.
\]
\end{prop}

\begin{proof}
By Proposition~\ref{prop:lift-be}, we only have to show the second
sentence: let us assume ${\rm Cont}(T)(\boldsymbol{b}_{0}\otimes\boldsymbol{b}_{1})\neq0$
and find a lift $\widetilde{T}$ of $T$ such that there exists some
standard basis element $\boldsymbol{b}_{2}$ such that $\widetilde{T}$
is a $\{\boldsymbol{b}_{0},\boldsymbol{b}_{1},\boldsymbol{b}_{2}\}$-lift
of $T$.

Say $T$ has vertices $\theta\in\beta_{0}\cap\beta_{r}$, $\xi\in\beta_{r}\cap\beta_{\infty}$,
and $\zeta\in\beta_{\infty}\cap\beta_{0}$. First, since ${\rm Cont}(T)(\boldsymbol{b}_{0}\otimes\boldsymbol{b}_{1})\neq0$,
$\boldsymbol{b}_{0}$ and $\boldsymbol{b}_{1}$ must come from $\theta$,
$\xi$, or $\zeta$. Let us consider the three cases separately.

\textbf{Case $(\gamma_{0},\gamma_{1},\gamma_{2})=(\beta_{0},\beta_{r},\beta_{\infty})$:}
let $\boldsymbol{b}_{0}=x_{i_{0},\ell_{0}}^{\ast}\theta$ and $\boldsymbol{b}_{1}=y_{i_{1}}\xi$
(note that ${\rm Cont}(T)(\boldsymbol{b}_{0}\otimes\boldsymbol{b}_{1})$
is always nonzero). Let $\widetilde{T}$ be any lift of $T$ to $\widetilde{\mathbb{T}^{2}}$
such that one of its vertices is a lift of $\boldsymbol{b}_{0}$.
Say its vertices are $\widetilde{v}_{0}$, $\widetilde{v}_{1}$, and
$\widetilde{v}_{2}$, in clockwise order, where $\widetilde{v}_{0}$
is a lift of $\boldsymbol{b}_{0}$. By Lemma~\ref{lem:Let--be},
$\widetilde{v}_{1}$ and $\widetilde{v}_{2}$ are lifts of standard
basis elements. By Lemma~\ref{lem:unique}, there is a unique standard
basis element that comes from $\xi$, and so $\widetilde{v}_{1}$
is a lift of $\boldsymbol{b}_{1}$.

\textbf{Case $(\gamma_{0},\gamma_{1},\gamma_{2})=(\beta_{r},\beta_{\infty},\beta_{0})$:}
let $\boldsymbol{b}_{0}=y_{i_{0}}\xi$ and $\boldsymbol{b}_{1}=x_{i_{1},\ell_{1}}y_{i_{1}'}^{\ast}\zeta$.
Let $\widetilde{T}$ be any lift of $T$ to $\widetilde{\mathbb{T}^{2}}$
such that one of its vertices is a lift of $\boldsymbol{b}_{1}$.
Say its vertices are $\widetilde{v}_{0}$, $\widetilde{v}_{1}$, and
$\widetilde{v}_{2}$, in clockwise order, where $\widetilde{v}_{1}$
is a lift of $\boldsymbol{b}_{1}$. By Lemma~\ref{lem:Let--be},
we only need to show that $\widetilde{v}_{0}$ is a lift of $\boldsymbol{b}_{0}$.
This follows since $i_{1}'\equiv i_{0}+\#(A_{\infty}\cap\partial_{\beta_{\infty}}T)\bmod p$,
which is equivalent to ${\rm Cont}(T)(\boldsymbol{b}_{0}\otimes\boldsymbol{b}_{1})\neq0$.

\textbf{Case $(\gamma_{0},\gamma_{1},\gamma_{2})=(\beta_{\infty},\beta_{0},\beta_{r})$:}
let $\boldsymbol{b}_{0}=x_{i_{0},\ell_{0}}y_{i_{0}'}^{\ast}\zeta$
and $\boldsymbol{b}_{1}=x_{i_{1},\ell_{1}}^{\ast}\theta$. Let $\widetilde{T}$
be any lift of $T$ to $\widetilde{\mathbb{T}^{2}}$ such that one
of its vertices is a lift of $\boldsymbol{b}_{0}$. Say its vertices
are $\widetilde{v}_{0}$, $\widetilde{v}_{1}$, and $\widetilde{v}_{2}$,
in clockwise order, where $\widetilde{v}_{0}$ is a lift of $\boldsymbol{b}_{0}$.
By Lemma~\ref{lem:Let--be}, we only need to show that $\widetilde{v}_{1}$
is a lift of $\boldsymbol{b}_{1}$. This follows since $(i_{0},\ell_{0}+\#(A_{0}\cap\partial_{\beta_{0}}T))=(i_{1},\ell_{1})$
as elements of the index set $I$ from Definition~\ref{def:index-set},
which is equivalent to ${\rm Cont}(T)(\boldsymbol{b}_{0}\otimes\boldsymbol{b}_{1})\neq0$.
\end{proof}

\subsection{\label{subsec:An-explicit-formula}Explicit formulas for triangle
counting maps}
\begin{defn}
\label{def:theta-tilde}Let $n\in\mathbb{Z}$, and let $i,\ell\in\mathbb{Z}$
be such that $i=n\bmod p$ and $n=\ell p+i$. Let $\widetilde{\theta}_{n}$
be the intersection point in $\pi^{-1}(\beta_{0})\cap\widetilde{\beta}_{r}^{0}\cap\pi^{-1}(\theta_{i})$
that is in the $\ell$th column. Let $\widetilde{T}_{n}$ be the triangle
in $\widetilde{\mathbb{T}^{2}}$ such that two of its vertices are
$\widetilde{\theta}_{n}$ and the origin $(0,0)$. Let $T_{n}$ be
the image of $\widetilde{T}_{n}$ under the covering map $\pi:\widetilde{\mathbb{T}^{2}}\to\mathbb{T}^{2}$.
Let $z_{n}:=n_{z}(T_{n})=n_{\pi^{-1}(z)}(\widetilde{T}_{n})$.
\end{defn}

\begin{convention}
\label{conv:identification}We sometimes identify the triangles $\widetilde{T}_{n}$
and $\widetilde{T}_{n}+(c_{n}m,0)$ for all $m\in\mathbb{Z}$, since
they are both the $\{\boldsymbol{b}_{0},\boldsymbol{b}_{1},\boldsymbol{b}_{2}\}$-lift
of $T_{n}$ for the same standard basis elements $\boldsymbol{b}_{0},\boldsymbol{b}_{1},\boldsymbol{b}_{2}$.
\end{convention}

See Figures~\ref{fig:zigzags53}~and~\ref{fig:tildet2-1-1} for
some examples of $\widetilde{T}_{n}$. Let us record the following
lemma.
\begin{lem}
\label{lem:thetanrow}The intersection point $\widetilde{\theta}_{n}$
is in the $(-nq^{-1}\bmod p)$th row.
\end{lem}

\begin{proof}
Same as the proof of Lemma~\ref{lem:theta-unique}.
\end{proof}
Definition~\ref{def:notation-sb} lists all the standard basis elements
and our shorthands for them. See Figure~\ref{fig:tildet2}: for each
standard basis element $\boldsymbol{b}$, we labelled its lifts with
our shorthand for $\boldsymbol{b}$.
\begin{defn}[Shorthands for standard basis elements]
\label{def:notation-sb}Write
\[
\boldsymbol{\theta}_{i,\ell}:=x_{i,\ell}^{\ast}\theta_{i},\ \boldsymbol{\xi}_{j}:=y_{-jq^{-1}\bmod p}\xi_{j},\ \boldsymbol{\zeta}_{i,\ell}:=x_{i,\ell}y_{-iq^{-1}\bmod p}^{\ast}\zeta
\]
for $i=0,\cdots,p-1$, $j=0,\cdots,q-1$, $\ell=0,\cdots,c_{i}-1$.
By abuse of notation, also let
\[
\boldsymbol{\theta}_{i,\ell}:=x_{i,\ell}\theta_{i},\ \boldsymbol{\xi}_{j}:=y_{-jq^{-1}\bmod p}^{\ast}\xi_{j},\ \boldsymbol{\zeta}_{i,\ell}:=y_{-iq^{-1}\bmod p}x_{i,\ell}^{\ast}\zeta.
\]
Recall that the indices of $y$ are interpreted modulo $p$, and the
indices of $x$ lie in the index set $I$ (Definition~\ref{def:index-set}).
The indices of $\boldsymbol{\xi}$ are interpreted modulo $q$, and
the indices of $\boldsymbol{\theta}$ and $\boldsymbol{\zeta}$ lie
in the index set $I$.
\end{defn}

Note that if $n=\ell p+i$, $i\in\{0,\cdots,p-1\}$, and $c_{i}\neq0$,
then $\widetilde{\theta}_{n}$ is a lift of $\boldsymbol{\theta}_{i,\ell}$.
\begin{prop}[Formulas for $\mu_{2}$ (Theorem~\ref{thm:gen-local-comp}~(\ref{enu:triangle})),
case $k=0$]
\label{prop:triangle}Let us specialize to $k=0$, i.e.\ as in Definition~\ref{def:actual-local-system},
equip $\beta_{0}$ with $E_{0}^{0}$ and $\beta_{\infty}$ with $E_{\infty}^{0}$.
Given $n$, let $\ell,i,j\in\mathbb{Z}$ denote the integers such
that $i=n\bmod p\in\{0,\cdots,p-1\}$, $j=n\bmod q\in\{0,\cdots,q-1\}$,
and $n=\ell p+i$. Let $\psi_{0r}^{0},\psi_{r\infty}^{0},\psi_{\infty0}^{0}$
be as in the statement of Theorem~\ref{thm:gen-local-comp}. Then,
we have
\begin{gather}
\mu_{2}(\psi_{0r}^{0}\otimes\psi_{r\infty}^{0})=\sum_{n\in\mathbb{Z}}U^{z_{n}}\sum_{d=0}^{s_{i}-1}u_{i,d+\ell}v_{j-p}\boldsymbol{\zeta}_{i,d}\label{eq:0rinf}\\
\mu_{2}(\psi_{r\infty}^{0}\otimes\psi_{\infty0}^{0})=\sum_{n\in\mathbb{Z}}U^{z_{n}}\sum_{d=0}^{s_{i}-1}v_{j-p}w_{i,d}\boldsymbol{\theta}_{i,d+\ell}\label{eq:rinf0}\\
\mu_{2}(\psi_{\infty0}^{0}\otimes\psi_{0r}^{0})=\sum_{n\in\mathbb{Z}}U^{z_{n}}\sum_{d=0}^{s_{i}-1}u_{i,d+\ell}w_{i,d}\boldsymbol{\xi}_{j-p}.\label{eq:inf0r}
\end{gather}
\end{prop}

\begin{proof}
The triangle $T_{n}$ in $\mathbb{T}^{2}$ has vertices $\zeta$,
$\theta_{i}$, and $\xi_{j-p}$. By Lemma~\ref{lem:theta-unique},
the triangles $\widetilde{T}_{n}+(d,0)$ for $d\in\mathbb{Z}$ are
the lifts of $T_{n}$ to $\widetilde{\mathbb{T}^{2}}$ such that all
its three vertices are lifts of standard basis elements. The vertices
of $\widetilde{T}_{n}+(d,0)$ are $\boldsymbol{\theta}_{i,d+\ell},\boldsymbol{\xi}_{j-p},\boldsymbol{\zeta}_{i,d}$.
Hence, $T_{n}$ has a $\{\boldsymbol{b}_{0},\boldsymbol{b}_{1},\boldsymbol{b}_{2}\}$-lift
for standard basis elements $\boldsymbol{b}_{0},\boldsymbol{b}_{1},\boldsymbol{b}_{2}$
if and only if
\[
\{\boldsymbol{b}_{0},\boldsymbol{b}_{1},\boldsymbol{b}_{2}\}=\{\boldsymbol{\theta}_{i,d+\ell},\boldsymbol{\xi}_{j-p},\boldsymbol{\zeta}_{i,d}\}
\]
for some $d=0,\cdots,c_{i}-1$. Hence, the proposition follows.
\end{proof}

\subsection{\label{subsec:Lifts-quad}Lifts of standard basis elements and quadrilateral
counting maps}

This subsection is analogous to Subsection~\ref{subsec:Lifts-and-standard-mu2},
but here we have the standard translate $\gamma_{3}$ of $\gamma_{0}$
from Subsection~\ref{subsec:The-setup} and we consider Maslov index
$-1$ quadrilaterals.

Let us first describe all the Maslov index $-1$ quadrilaterals. To
do so, we find it easy to work in the universal cover $\mathbb{R}^{2}$.
Let $\overline{\pi}:\mathbb{R}^{2}\to\mathbb{T}^{2}$ be the covering
map. For each triangle $\widetilde{T}_{n}$ in $\widetilde{\mathbb{T}^{2}}$,
fix a lift $\overline{T}_{n}$ of $\widetilde{T}_{n}$ to $\mathbb{R}^{2}$.

The Maslov index $-1$ quadrilaterals $Q$ are in one-to-one correspondence
with pairs $(T_{n},\overline{v})$ of a triangle $T_{n}$ and a point
$\overline{v}\in\partial_{\overline{\pi}^{-1}(\gamma_{0})}\overline{T}_{n}\cap\overline{\pi}^{-1}(\Theta^{+})$,
where $\Theta^{+}\in\gamma_{0}\cap\gamma_{3}$ is the intersection
point in the top homological grading. Given $(T_{n},\overline{v})$,
let us describe the corresponding quadrilateral $Q$. Let $\overline{v}_{i}\in\overline{\pi}^{-1}(\gamma_{i})\cap\overline{\pi}^{-1}(\gamma_{(i+1)\bmod3})$
for $i=0,1,2$ be the vertices of $\overline{T}_{n}$. Then, let $\overline{Q}\subset\mathbb{R}^{2}$
be the quadrilateral with sides in $\overline{\pi}^{-1}(\gamma_{i})$
for $i=0,1,2,3$ and vertices $\overline{v}_{0},\overline{v}_{1},\overline{v}_{2}',\overline{v}$,
where $\overline{v}_{2}'\in\overline{\pi}^{-1}(\gamma_{2})\cap\overline{\pi}^{-1}(\gamma_{3})$
is the nearest point to $\overline{v}_{2}$. The quadrilateral $Q$
is given by the image of $\overline{Q}$ under $\overline{\pi}$.
See Figure~\ref{fig:quadrilaterals} for the lifts of some quadrilaterals
for $(p,q,k)=(5,3,0)$. Note that $Q$ and $T_{n}$ only differ in
the small region between $\gamma_{0}$ and $\gamma_{3}$; we say that
\emph{$Q$ is a small perturbation of $T_{n}$}.

\begin{defn}[More lifts of standard basis elements]
\label{def:lift-of-basis-1}Call a basis element $\boldsymbol{b}$
of $\boldsymbol{CF}^{-}(\gamma_{0}^{F_{0}},\gamma_{3}^{F_{0}})$ \emph{standard}
if it is of the form $ee^{\ast}\Theta^{+}$, where $e$ is a basis
element of $F_{0}$ and $\Theta^{+}\in\gamma_{0}\cap\gamma_{3}$ is
the intersection point in the top homological grading.
\begin{enumerate}
\item A point $\widetilde{\Theta}\in\pi^{-1}(\Theta_{0}^{+})$ is a \emph{lift}
of $x_{i,\ell}x_{i,\ell}^{\ast}\Theta_{0}^{+}$ if it lies in the
$(a,b)$th square for 
\[
a\equiv\ell\mod c_{i},\ b\equiv-iq^{-1}\mod p.
\]
\item Every point $\widetilde{\Theta}\in\pi^{-1}(\Theta_{r}^{+})$ is a
\emph{lift} of $\Theta_{r}^{+}$.
\item A point $\widetilde{\Theta}\in\pi^{-1}(\Theta_{\infty}^{+})$ is a
\emph{lift} of $y_{i'}y_{i'}^{\ast}\Theta_{\infty}^{+}$ if it lies
in the $b$th row for 
\[
b\equiv i'\mod p.
\]
\end{enumerate}
\end{defn}

As in Subsection~\ref{subsec:Lifts-and-standard-mu2}, we study $\mu_{3}$
by studying the \emph{contribution }${\rm Cont}(Q)$\emph{ of each
quadrilateral} $Q$ separately. Note that $\mu_{3}$ is the sum of
${\rm Cont}(Q)$ for all Maslov index $-1$ quadrilaterals $Q$.
\begin{defn}
If $Q$ is a Maslov index $-1$ quadrilateral with vertices $c_{t}\in\gamma_{t}\cap\gamma_{(t+1)\bmod4}$
($t=0,1,2,3$), define the \emph{contribution of $Q$ to $\mu_{3}$}
as
\begin{multline*}
{\rm Cont}(Q):\boldsymbol{CF}^{-}(\gamma_{0}^{F_{0}},\gamma_{1}^{F_{1}})\otimes\boldsymbol{CF}^{-}(\gamma_{1}^{F_{1}},\gamma_{2}^{F_{2}})\otimes\boldsymbol{CF}^{-}(\gamma_{2}^{F_{2}},\gamma_{3}^{F_{0}})\to\boldsymbol{CF}^{-}(\gamma_{0}^{F_{0}},\gamma_{3}^{F_{0}}):\\
e_{1}a_{1}\otimes e_{2}a_{2}\otimes e_{3}a_{3}\mapsto\begin{cases}
U^{n_{z}(Q)}\rho(Q)(e_{1}\otimes e_{2}\otimes e_{3})c_{4} & {\rm if}\ (a_{1},a_{2},a_{3})=(c_{1},c_{2},c_{3})\\
0 & {\rm otherwise}
\end{cases}.
\end{multline*}
\end{defn}

We have the following analogue of Proposition~\ref{prop:lift-be}.
\begin{defn}
Let $Q$ be a Maslov index $-1$ quadrilateral in $\mathbb{T}^{2}$.
Let $\boldsymbol{b}_{t}$ be a basis element of $\boldsymbol{CF}^{-}(\gamma_{t}^{F_{t}},\gamma_{t+1}^{F_{t+1}})$
for $t=0,1,2$, and let $\boldsymbol{b}_{3}$ be a standard basis
element of $\boldsymbol{CF}^{-}(\gamma_{0}^{F_{0}},\gamma_{3}^{F_{0}})$.
A lift $\widetilde{Q}$ of $Q$ to $\widetilde{\mathbb{T}^{2}}$ is
a\emph{ $\{\boldsymbol{b}_{0},\boldsymbol{b}_{1},\boldsymbol{b}_{2},\boldsymbol{b}_{3}\}$-lift
of $Q$} if its three vertices are lifts of $\boldsymbol{b}_{0}$,
$\boldsymbol{b}_{1}$, $\boldsymbol{b}_{2}$, and $\boldsymbol{b}_{3}$.
\end{defn}

\begin{prop}
\label{prop:lift-quad-easy}Let $Q$ be a Maslov index $-1$ quadrilateral
in $\mathbb{T}^{2}$, let $\boldsymbol{b}_{t}$ be a basis element
of $\boldsymbol{CF}^{-}(\gamma_{t}^{F_{t}},\gamma_{t+1}^{F_{t+1}})$
for $t=0,1,2$, and let $\boldsymbol{b}_{3}$ be a standard basis
element of $\boldsymbol{CF}^{-}(\gamma_{0}^{F_{0}},\gamma_{3}^{F_{0}})$.
If $Q$ has a $\{\boldsymbol{b}_{0},\boldsymbol{b}_{1},\boldsymbol{b}_{2},\boldsymbol{b}_{3}\}$-lift,
then 
\[
{\rm Cont}(Q)(\boldsymbol{b}_{0}\otimes\boldsymbol{b}_{1}\otimes\boldsymbol{b}_{2})=U^{n_{\widetilde{z(u)}}(\widetilde{Q})}\boldsymbol{b}_{3}.
\]
\end{prop}

\begin{proof}
Direct from the definitions. Note that in our Heegaard diagram, every
Maslov index $-1$ quadrilateral has a unique holomorphic representative.
\end{proof}
Proposition~\ref{prop:main-mu3} is the main proposition of this
subsection. See Figure~\ref{fig:quadrilaterals} for the lifts of
some quadrilaterals for $(p,q,k)=(5,3,0)$, together with the corresponding
$\boldsymbol{b}_{0},\boldsymbol{b}_{1},\boldsymbol{b}_{2}',\boldsymbol{b}_{3}$.
\begin{prop}[Main proposition for $\mu_{3}$]
\label{prop:main-mu3}Let $Q$ be a Maslov index $-1$ quadrilateral
that corresponds to $(T_{n},\overline{v})$, let $\boldsymbol{b}_{t}$
be a standard basis element of $\boldsymbol{CF}^{-}(\gamma_{t}^{F_{t}},\gamma_{(t+1)\bmod3}^{F_{(t+1)\bmod3}})$
for $t=0,1,2$, and let $\boldsymbol{b}_{2}'\in\boldsymbol{CF}^{-}(\gamma_{2}^{F_{2}},\gamma_{3}^{F_{0}})$
be the image of $\boldsymbol{b}_{2}$ under the nearest point map
(Subsection~\ref{subsec:Standard-translates}). If $T_{n}$ does
not admit a $\{\boldsymbol{b}_{0},\boldsymbol{b}_{1},\boldsymbol{b}_{2}\}$-lift,
then ${\rm Cont}(Q)(\boldsymbol{b}_{0}\otimes\boldsymbol{b}_{1}\otimes\boldsymbol{b}_{2}')=0$.
Otherwise, there exists some $d\in\mathbb{Z}$ such that $\widetilde{T}_{n}+(d,0)$
is a $\{\boldsymbol{b}_{0},\boldsymbol{b}_{1},\boldsymbol{b}_{2}\}$-lift
of $T_{n}$. Let $\boldsymbol{b}_{3}$ be the standard basis element
of $\boldsymbol{CF}^{-}(\gamma_{0}^{F_{0}},\gamma_{3}^{F_{0}})$ such
that if $\widehat{\pi}:\mathbb{R}^{2}\to\widetilde{\mathbb{T}^{2}}$
is the covering map, then $\widehat{\pi}(\overline{v}+(d,0))$ is
a lift of $\boldsymbol{b}_{3}$. Then, we have 
\[
{\rm Cont}(Q)(\boldsymbol{b}_{0}\otimes\boldsymbol{b}_{1}\otimes\boldsymbol{b}_{2}')=U^{n_{\widetilde{z(u)}}(\widetilde{T}_{n})}\boldsymbol{b}_{3}.
\]
\end{prop}

\begin{proof}
The second sentence follows from a routine case check similar to the
proof of Proposition~\ref{prop:lift-mu2}. The rest follows from
Proposition~\ref{prop:lift-quad-easy}: let $\widetilde{Q}$ be the
lift of $Q$ from the discussion before Definition~\ref{def:lift-of-basis-1}.
If $\widetilde{T}_{n}+(d,0)$ is a $\{\boldsymbol{b}_{0},\boldsymbol{b}_{1},\boldsymbol{b}_{2}\}$-lift
of $T_{n}$, then $\widetilde{Q}+(d,0)$ is a $\{\boldsymbol{b}_{0},\boldsymbol{b}_{1},\boldsymbol{b}_{2}',\boldsymbol{b}_{3}\}$-lift
of $Q$. Since $\widetilde{z(u)}=\widetilde{z(u)}+(d,0)$ and since
$\widetilde{T}_{n}$ and $\widetilde{Q}$ only differ in the small
regions between $\pi^{-1}(\gamma_{0})$ and $\pi^{-1}(\gamma_{3})$,
we have
\[
n_{\widetilde{z(u)}}(\widetilde{Q}+(d,0))=n_{\widetilde{z(u)}}(\widetilde{Q})=n_{\widetilde{z(u)}}(\widetilde{T}_{n}).\qedhere
\]
\end{proof}

\section{\label{sec:The-hat-version}The hat version}

In this section, we show that Theorem~\ref{thm:gen-local-comp}~(\ref{enu:quadrilateral})
holds whenever Equation~(\ref{eq:uvw-modu}) is satisfied (note that
Equation~(\ref{eq:uvw-modu}) determines the cycles $\psi_{0r}^{k},\psi_{r\infty}^{k},\psi_{\infty0}^{k}$
modulo $U$). Along the way, we motivate the definition of the local
systems $E_{0}^{k}$ and $E_{\infty}^{k}$, and Equation~(\ref{eq:uvw-modu}).

For completeness, we also prove that Theorem~\ref{thm:gen-local-comp}~(\ref{enu:triangle})
holds modulo $U$ whenever Equation~(\ref{eq:uvw-modu}) holds, thereby
proving the hat version of Theorem~\ref{thm:gen-local-comp}. The
proof of Theorem~\ref{thm:gen-local-comp}~(\ref{enu:triangle})
for the minus version in Section~\ref{sec:The-minus-version} does
not depend on this.

Recall the triangles $\widetilde{T}_{n}$ and the intersection points
$\widetilde{\theta}_{n}$ from Definition~\ref{def:theta-tilde}.

\begin{figure}[h]
\begin{centering}
\raisebox{-0.5\height}{\includegraphics[scale=0.5]{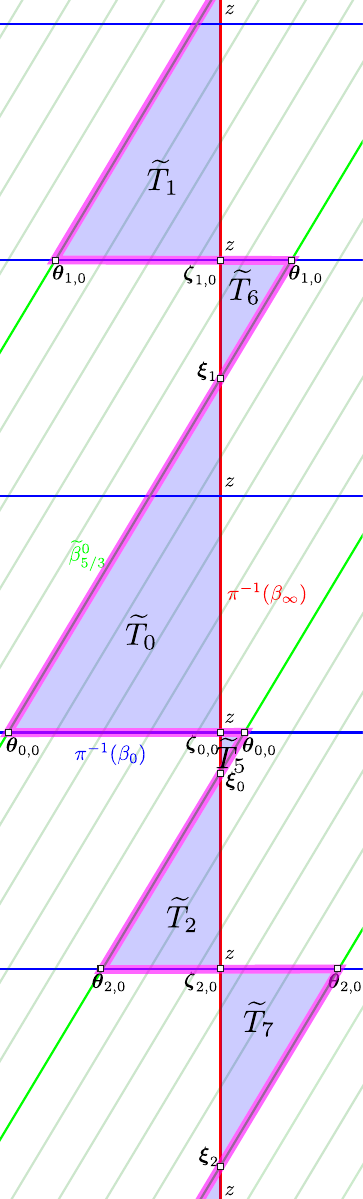}}\enspace{}\raisebox{-0.5\height}{\includegraphics[scale=0.5]{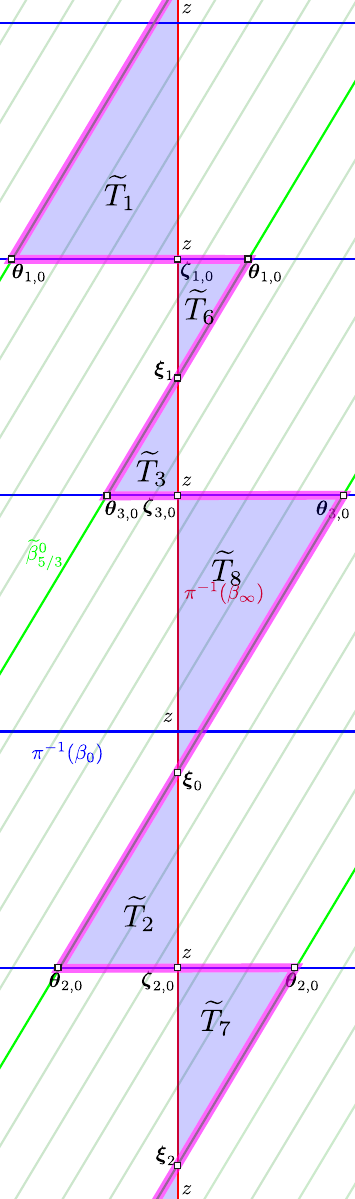}}\enspace{}\raisebox{-0.5\height}{\includegraphics[scale=0.5]{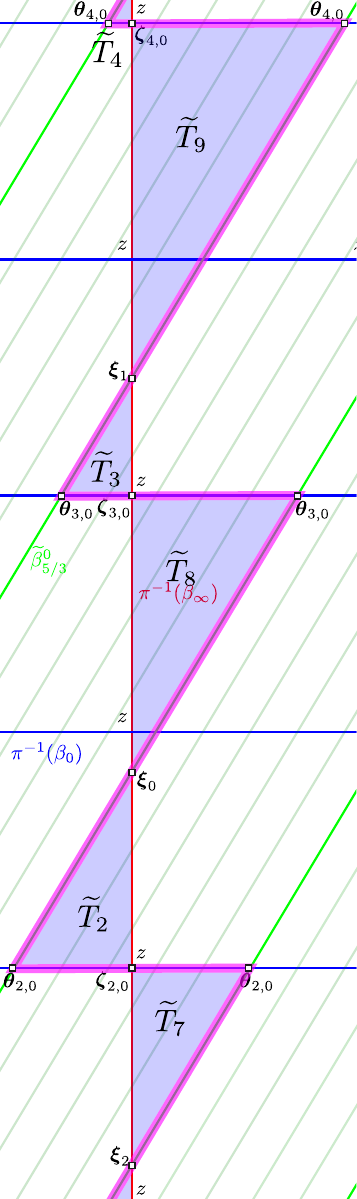}}\enspace{}\raisebox{-0.5\height}{\includegraphics[scale=0.5]{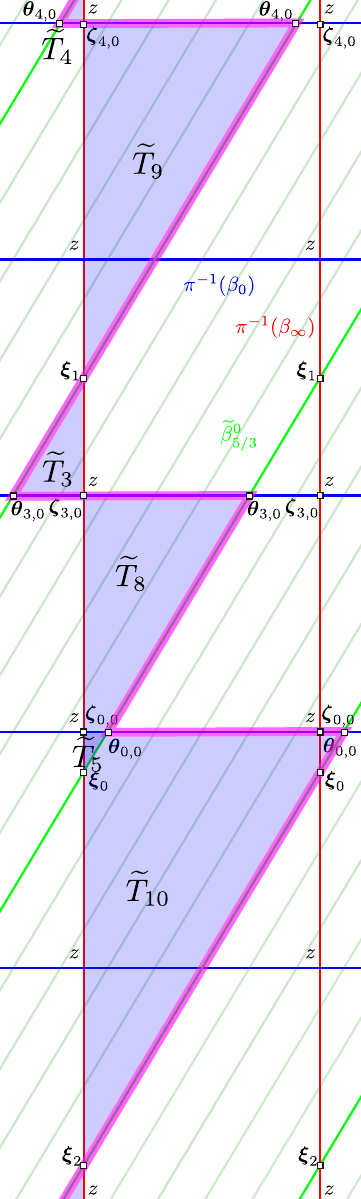}}\enspace{}\raisebox{-0.5\height}{\includegraphics[scale=0.5]{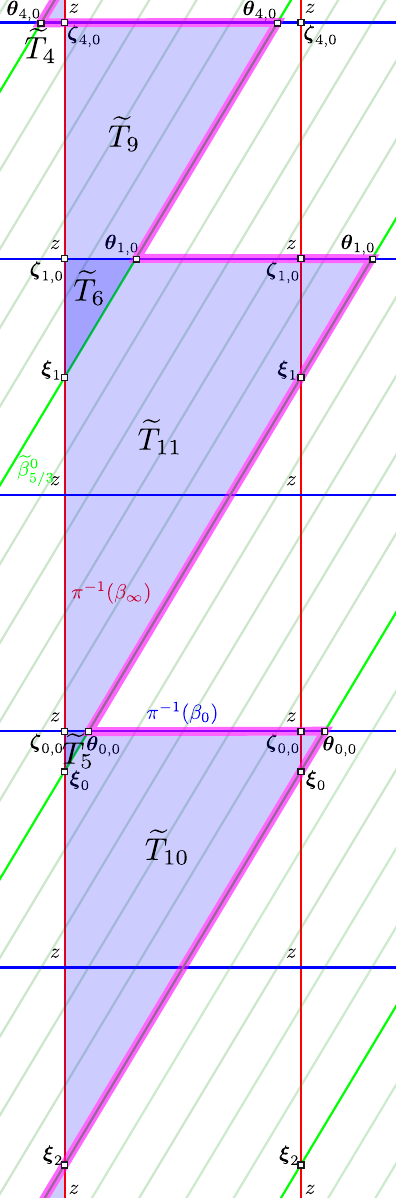}}
\par\end{centering}
\caption{\label{fig:zigzags53}Zig-zags and triangles for $(p,q)=(5,3)$, $k=0,1,2,3,4$,
respectively.}
\end{figure}
\begin{figure}[h]
\begin{centering}
\raisebox{-0.5\height}{\includegraphics[scale=0.45]{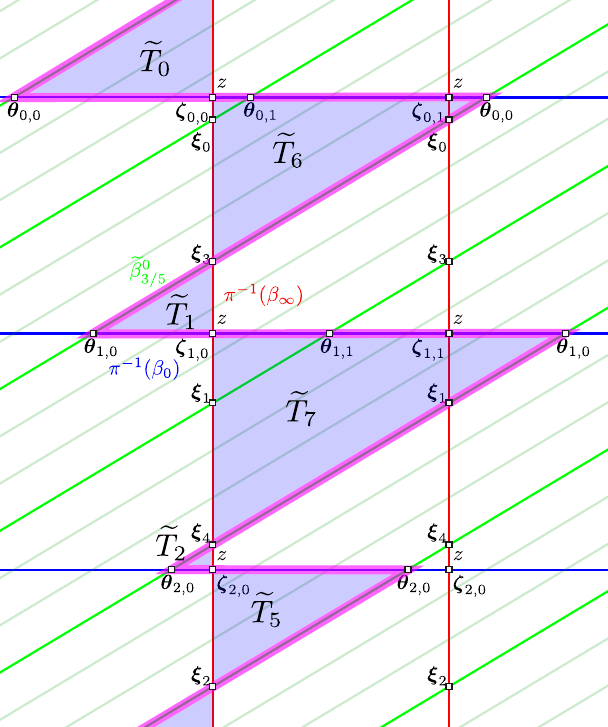}}\quad{}\raisebox{-0.5\height}{\includegraphics[scale=0.45]{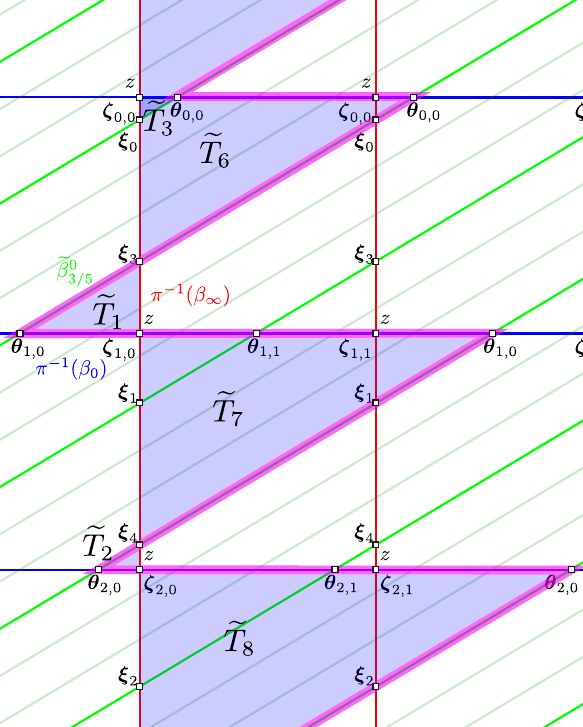}}\quad{}\raisebox{-0.5\height}{\includegraphics[scale=0.45]{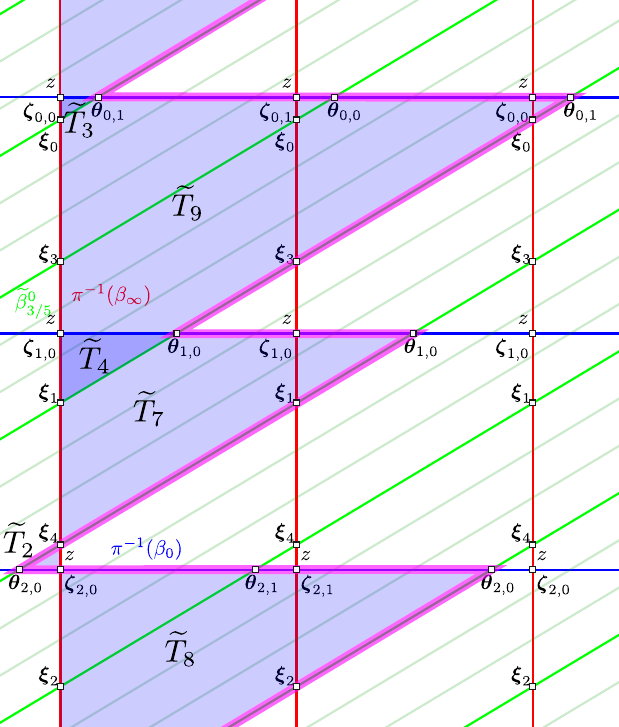}}
\par\end{centering}
\caption{\label{fig:tildet2-1-1}Zig-zags and triangles for $(p,q)=(3,5)$,
$k=0,1,2$, respectively.}
\end{figure}

\subsection{\label{subsec:The-zig-zag}The zig-zag $ZZ_{0}$}

One key observation that motivated the rational surgery exact triangles
(Theorem~\ref{thm:rational-surgery-k=00003D0}~and~\ref{thm:rational-surgery-kgeneral})
is that the boundaries of certain triangles $\widetilde{T}_{n}$ connect
and form a zig-zag. In this subsection, we define the \emph{zig-zag
$ZZ_{0}\subset\widetilde{\mathbb{T}^{2}}$} which corresponds to the
case $k=0$. These are drawn in the leftmost diagrams of Figures~\ref{fig:zigzags53}~and~\ref{fig:tildet2-1-1},
for $(p,q)=(5,3),(3,5)$, respectively.

Let $b\in\mathbb{Z}/p\mathbb{Z}$. Recall from Subsection~\ref{subsec:cover}
that $\widetilde{\beta}_{0}^{b}$ is the connected component of $\pi^{-1}(\beta_{0})$
that is in the $b$th row. Consider all the intersection points $\widetilde{\theta}_{n}\in\widetilde{\beta}_{0}^{b}$
such that $z_{n}:=n_{\pi^{-1}(z)}(\widetilde{T}_{n})=0$. If there
are at least two such intersection points $\widetilde{\theta}_{n}$,
then call $\widetilde{\beta}_{0}^{b}$ \emph{$0$-special}, and call
the leftmost and rightmost such $\widetilde{\theta}_{n}$'s\emph{
$0$-special}. Call the triangle $\widetilde{T}_{n}$ \emph{$0$-special}
if $\widetilde{\theta}_{n}$ is $0$-special (recall that $\widetilde{\theta}_{n}$
is a vertex of $\widetilde{T}_{n}$). We have the following lemma.
\begin{lem}
\label{lem:51}We have $z_{n}=0$ if and only if $n\in\{0,\cdots,p+q-1\}$.
Hence, for $b\in\mathbb{Z}/p\mathbb{Z}$, if $i$ is such that $b\equiv-iq^{-1}\bmod p$,
then there are $s_{i}+1$ (recall Definition~\ref{def:si}) many
$\widetilde{\theta}_{n}\in\widetilde{\beta}_{0}^{b}$ such that $z_{n}=0$.
Thus, $\widetilde{\beta}_{0}^{b}$ is $0$-special if and only if
$b\equiv-iq^{-1}\bmod p$ for some $i=0,1,\cdots,\min(p,q)$, and
the following are equivalent:
\begin{enumerate}
\item the intersection point $\widetilde{\theta}_{n}$ is $0$-special
\item the triangle $\widetilde{T}_{n}$ is $0$-special
\item $n\in\{0,1,\cdots,\min(p,q)-1,\max(p,q),\max(p,q)+1,\cdots,p+q-1\}$.
\end{enumerate}
\end{lem}

\begin{proof}
The first sentence follows from examining the genus $1$ Heegaard
diagram (Subsection~\ref{subsec:The-setup} and Figure~\ref{fig:diagramt2}).
The second sentence follows from Lemma~\ref{lem:thetanrow}, and
the third sentence follows from the definitions and the second sentence.
\end{proof}
We will show in Lemma~\ref{lem:For-each-} that the \emph{$0$-}special
intersection points $\widetilde{\theta}_{n}$ can be connected with
certain line segments $H_{i}\subset\pi^{-1}(\beta_{0})$ and $L_{i}\subset\widetilde{\beta}_{r}^{0}$,
and that these line segments connect at endpoints and form a zig-zag
(see Figures~\ref{fig:zigzags53}~and~\ref{fig:tildet2-1-1}).
Call this the \emph{zig-zag }$ZZ_{0}\subset\pi^{-1}(\beta_{0})\cup\widetilde{\beta}_{r}^{0}\subset\widetilde{\mathbb{T}^{2}}$,
i.e.\ $ZZ_{0}$ is the union of the $H_{i}$'s and $L_{i}$'s. This
zig-zag together with the $y$-axis $\widetilde{\beta}_{\infty}^{0}$
bound the $0$-special triangles.

\begin{defn}[Line segments $H_{i}\subset\pi^{-1}(\beta_{0})$ and $L_{i}\subset\widetilde{\beta}_{r}^{0}$]
\label{def:hi-li}Let $b\in\mathbb{Z}/p\mathbb{Z}$ be such that
$\widetilde{\beta}_{0}^{b}$ is $0$-special, and let  $i=0,\cdots,p-1$
be such that $b\equiv-iq^{-1}\bmod p$. Define $H_{i}\subset\widetilde{\beta}_{0}^{b}$
as the horizontal line segment between the two \emph{$0$-}special
intersection points $\widetilde{\theta}_{n}$ on $\widetilde{\beta}_{0}^{b}$.
To define the line segment $L_{i}$, begin at $\widetilde{\theta}_{i}$
and follow $\widetilde{\beta}_{r}^{0}$ in the $+y$-direction. Continue
until the path crosses the $y$-axis $\widetilde{\beta}_{\infty}^{0}\subset\widetilde{\mathbb{T}^{2}}$,
and then proceed until the path reaches some $\widetilde{\beta}_{0}^{c}$,
at which point we stop. Define the line segment $L_{i}\subset\widetilde{\beta}_{r}^{0}$
as this path.
\end{defn}

\begin{lem}
\label{lem:For-each-}Let $b\in\mathbb{Z}/p\mathbb{Z}$ be such that
$\widetilde{\beta}_{0}^{b}$ is $0$-special, let  $i=0,\cdots,p-1$
be such that $b\equiv-iq^{-1}\bmod p$, and let $t\in\mathbb{Z}$
be such that $\widetilde{\theta}_{i}$ and $\widetilde{\theta}_{t}$
are the two endpoints of $L_{i}\subset\widetilde{\beta}_{r}^{0}$.
Then, (1) $L_{i}\setminus\{\widetilde{\theta}_{i},\widetilde{\theta}_{t}\}$
does not intersect any other \emph{$0$-}special $\widetilde{\beta}_{0}^{a}$,
and (2) $t\ge p$ and $\widetilde{\theta}_{t}$ is \emph{$0$-}special.
\end{lem}

\begin{proof}
Let us first observe that the intersections of $L_{i}$ and $\pi^{-1}(\beta_{0})$
are $\widetilde{\theta}_{i+mq}$ for $m\in\mathbb{Z}_{\ge0}$ such
that $i+mq\le t$. In particular, $t\ge i+q$.

To show (1), assume that $L_{i}$ intersects $\pi^{-1}(\beta_{0})$
at some $\widetilde{\theta}_{i+mq}$ for some $i+mq\neq i,t$. Then,
by the definition of $L_{i}$, $\widetilde{\theta}_{i+mq}\in[-1,0]\times(\mathbb{R}/p\mathbb{Z})$,
and so $i+mq\in\{0,\cdots,p-1\}$. Hence for all integers $N\ge1$,
$i+mq-Np\le-1$ and $i+mq+Np\ge p+q$, and so $z_{i+mq-Np},z_{i+mq+Np}>0$
by Lemma~\ref{lem:51}.

(2) follows from the following: if we let $r:=t\bmod p$, then (a)
$z_{r}=0$, (b) $r<p\le t$, (c) $z_{t}=0$, and (d) $z_{t+Np}>0$
for all integers $N\ge0$. (a) holds by Lemma~\ref{lem:51}. (b)
holds since $\widetilde{\theta}_{t}\in[0,\infty)\times(\mathbb{R}/p\mathbb{Z})$.
(c) holds since $\widetilde{T}_{t}$ is supported in $\mathbb{R}\times(-c-1+\varepsilon,-c]$.
(d) follows from Lemma~\ref{lem:51} since $t+Np\ge p+q$.
\end{proof}

\subsection{\label{subsec:The-triangle-counting}Triangle counting maps}

The goal of this subsection is to motivate the definition of $u=(u_{0},\cdots,u_{p-1})$
and $c=(c_{0},\cdots,c_{p-1})$, to motivate Equation~(\ref{eq:uvw-modu})
and hence define the cycles $\widehat{\psi}_{0r}^{k}\in\widehat{CF}(\beta_{0}^{E_{0}^{k}},\beta_{r})$,
$\widehat{\psi}_{r\infty}^{k}\in\widehat{CF}(\beta_{r},\beta_{\infty}^{E_{\infty}^{k}})$,
and $\widehat{\psi}_{\infty0}^{k}\in\widehat{CF}(\beta_{\infty}^{E_{\infty}^{k}},\beta_{0}^{E_{0}^{k}})$
for the hat version of Theorem~\ref{thm:gen-local-comp} (compare
Lemma~\ref{lem:hat-minus}), and to show that the triangle counting
maps $\mu_{2}(\widehat{\psi}_{0r}^{k}\otimes\widehat{\psi}_{r\infty}^{k})$,
$\mu_{2}(\widehat{\psi}_{r\infty}^{k}\otimes\widehat{\psi}_{\infty0}^{k})$,
and $\mu_{2}(\widehat{\psi}_{\infty0}^{k}\otimes\widehat{\psi}_{0r}^{k})$
are zero.

For $k=0,\cdots,p-1$, let $\Phi_{k}:\widetilde{\mathbb{T}^{2}}\to\widetilde{\mathbb{T}^{2}}$
be the translation of $\widetilde{\mathbb{T}^{2}}$ such that $\Phi_{k}(\widetilde{\theta}_{0})=\widetilde{\theta}_{k}$.
Define the \emph{$k$th zig-zag $ZZ_{k}\subset\pi^{-1}(\beta_{0})\cup\widetilde{\beta}_{r}^{0}$}
in $\widetilde{\mathbb{T}^{2}}$ as $\Phi_{k}(ZZ_{0})$. Then, $\widetilde{\theta}_{i}$
is a vertex of $ZZ_{0}$ if and only if $\widetilde{\theta}_{i+k}$
is a vertex of $ZZ_{k}$. Call $\widetilde{T}_{n}$ \emph{special}
if it is one of 
\begin{equation}
\widetilde{T}_{k},\widetilde{T}_{k+1},\cdots,\widetilde{T}_{k+\min(p,q)-1},\ \widetilde{T}_{k+\max(p,q)},\widetilde{T}_{k+\max(p,q)+1},\cdots,\widetilde{T}_{k+p+q-1}.\label{eq:triangles-k}
\end{equation}
If $u$ and $c$ are given and so the local systems are defined, call
a basis element $\boldsymbol{b}$ \emph{special} if there exists a
special triangle $\widetilde{T}_{n}$ such that there exists a vertex
of $\widetilde{T}_{n}$ that is a lift of $\boldsymbol{b}$.

\subsubsection{\label{subsec:Motivating-the-definitions}Motivating the definitions
of $u$ and $c$}

We will define $u$ and $c$ such that if $\widehat{\psi}_{0r}^{k}$,
$\widehat{\psi}_{r\infty}^{k}$, and $\widehat{\psi}_{\infty0}^{k}$
are the sum of the respective special basis elements, then the special
triangles are precisely the triangles that contribute to the triangle
counting maps ($\mu_{2}$), and they cancel in pairs. This will recover
the definition of $u$ and $c$ from Definition~\ref{def:actual-local-system}.

For each $i=0,\cdots,p-1$, consider all the vertices $\widetilde{\theta}_{n}$
of $ZZ_{k}$ in the $(-iq^{-1}\bmod p)$th row. If there is only at
most one (in fact, there is always at least one), let $c_{i}=0$.
Otherwise, let $\widetilde{\theta}_{n_{0}}$ and $\widetilde{\theta}_{n_{1}}$
be the leftmost and rightmost such $\widetilde{\theta}_{n}$, and
let $c_{i}=(n_{1}-n_{0})/p$. Then, we indeed have $c_{i}=s_{i-k}$
by Lemma~\ref{lem:51}.

We show in Corollary~\ref{cor:u-unique} that there is a unique choice
for $u$ such that the special triangles do not intersect $\widetilde{z(u)}$,
and that this $u$ is given by Equation~(\ref{eq:uc}). Let $\widetilde{z}_{i}:=(\varepsilon,(iq^{-1}\bmod p)+\varepsilon)\in\widetilde{\mathbb{T}^{2}}$
and $\widetilde{w}_{i}:=\widetilde{z}_{i}-(2\varepsilon,0)$.
\begin{lem}
\label{lem:zigzag-basepoint}Let $k=0,1,\cdots,p$, and let $\widetilde{T}(ZZ_{k})\subset\widetilde{\mathbb{T}^{2}}$
be the union of the supports of the special triangles (\ref{eq:triangles-k}),
i.e.\ it is the region bounded by $\widetilde{\beta}_{\infty}^{0}$
and $ZZ_{k}$. For $i=0,\cdots,p-1$, $\widetilde{z}_{i}\in\widetilde{T}(ZZ_{k})$
if and only if $i\in\{0,\cdots,k-1\}$, and $\widetilde{w}_{i}\in\widetilde{T}(ZZ_{k})$
if and only if $i\in\{k,\cdots,p-1\}$. Also, for $n=k,k+1,\cdots,k+p+q-1$,
the triangle $\widetilde{T}_{n}$ is contained in $\widetilde{T}(ZZ_{k})$.
\end{lem}

\begin{proof}
Let 
\[
V(ZZ_{k}):=\{\widetilde{\theta}_{k},\cdots,\widetilde{\theta}_{k+\min(p,q)-1},\widetilde{\theta}_{k+\max(p,q)},\cdots,\widetilde{\theta}_{k+p+q-1}\}
\]
be the set of vertices of $ZZ_{k}$, let $V_{k}=\{\widetilde{\theta}_{k},\widetilde{\theta}_{k+p},\widetilde{\theta}_{k+q},\widetilde{\theta}_{k+p+q}\}$,
and let $P_{k}$ be the parallelogram of base length $1$ and height
$1$ with vertices $V_{k}$.

To show the second sentence, let us induct on $k$. First, the case
$k=0$ holds since $\widetilde{w}_{i}$ is in $\widetilde{T}_{i}$
for $i=0,\cdots,p-1$. For the induction step, let $k\in\{0,\cdots,p-1\}$,
assume the statement for $k$, and let us show it for $k+1$. We have
\[
V(ZZ_{k+1})=\left(V(ZZ_{k})\setminus V_{k}\right)\cup\left(V_{k}\setminus V(ZZ_{k})\right),
\]
and so (ignoring the boundaries) we have (compare Figures~\ref{fig:zigzags53}~and~\ref{fig:tildet2-1-1})
\[
\widetilde{T}(ZZ_{k+1})=\left(\widetilde{T}(ZZ_{k})\setminus P_{k}\right)\cup\left(P_{k}\setminus\widetilde{T}(ZZ_{k})\right).
\]
The statement for $k+1$ follows since $P_{k}\cap\{\widetilde{z}_{0},\cdots,\widetilde{z}_{p-1}\}=\{\widetilde{z}_{k}\}$
and $P_{k}\cap\{\widetilde{w}_{0},\cdots,\widetilde{w}_{p-1}\}=\{\widetilde{w}_{k}\}$.

Finally, the last statement follows since for $t=p$ and $t=q$, the
triangle $\widetilde{T}_{n+t}$ is contained in $\widetilde{T}_{i}$
if $n,n+t\le p-1$, and $\widetilde{T}_{n}$ is contained in $\widetilde{T}_{n+t}$
if $n,n+t\ge p$.
\end{proof}
\begin{cor}
\label{cor:zktkpq}For $k=0,\cdots,p-1$, $\widetilde{z}_{k}$ is
contained in $\widetilde{T}_{k+p+q}$.
\end{cor}

\begin{proof}
By Lemma~\ref{lem:zigzag-basepoint}, $\widetilde{z}_{k}\in\widetilde{T}(ZZ_{k+1})\setminus\widetilde{T}(ZZ_{k})$.
Since $n_{k+\min(p,q)}=0$ by Lemma~\ref{lem:51}, $\widetilde{z}_{k}$
is not contained in $\widetilde{T}_{k+\min(p,q)}$. Hence, $\widetilde{z}_{k}$
must be contained in $\widetilde{T}_{k+p+q}$.
\end{proof}
\begin{cor}[$u$ is unique]
\label{cor:u-unique}There is a unique sequence $u=(u_{0},\cdots,u_{p-1})$,
$u_{i}\in\{1,U\}$, such that the special triangles (\ref{eq:triangles-k})
do not intersect $\widetilde{z(u)}$, and this $u$ is given by Equation~(\ref{eq:uc}).
\end{cor}

\begin{proof}
Lemma~\ref{lem:zigzag-basepoint} shows that Equation~(\ref{eq:uc})
works (alternatively, see Lemma~\ref{lem:nzk}), and that $\widetilde{z}_{0},\cdots,\widetilde{z}_{k-1}$
must be moved to the left of the $y$-axis $\widetilde{\beta}_{\infty}^{0}\subset\widetilde{\mathbb{T}^{2}}$.
Since the triangles $\widetilde{T}_{k},\cdots,\widetilde{T}_{p-1}$
are contained in $\widetilde{T}(ZZ_{k})$, $\widetilde{z(u)}$ should
avoid them. Since $\widetilde{w}_{i}$ is contained in $\widetilde{T}_{i}$
for $i=0,\cdots,p-1$, we cannot move $\widetilde{z}_{k},\cdots,\widetilde{z}_{p-1}$
to the left of $\widetilde{\beta}_{\infty}^{0}$, and so the uniqueness
of $u$ follows.
\end{proof}

\subsubsection{The triangle counting maps vanish}

Now, we have defined $u$ and $c$, and hence the local systems. Define
$\widehat{\psi}_{0r}^{k}\in\widehat{CF}(\beta_{0}^{E_{0}^{k}},\beta_{r})$,
$\widehat{\psi}_{r\infty}^{k}\in\widehat{CF}(\beta_{r},\beta_{\infty}^{E_{\infty}^{k}})$,
and $\widehat{\psi}_{\infty0}^{k}\in\widehat{CF}(\beta_{\infty}^{E_{\infty}^{k}},\beta_{0}^{E_{0}^{k}})$
as the sum of the respective special basis elements. Let us record
the following lemma.
\begin{lem}
\label{lem:hat-minus}The cycles $\psi_{0r}^{k},\psi_{r\infty}^{k},\psi_{\infty0}^{k}$
reduce to $\widehat{\psi}_{0r}^{k},\widehat{\psi}_{r\infty}^{k},\widehat{\psi}_{\infty0}^{k}$
modulo $U$, respectively, if $u_{i,\ell},v_{j},w_{i,\ell}$ satisfy
Equation~(\ref{eq:uvw-modu}). In particular, special basis elements
are standard.
\end{lem}

\begin{proof}
Since (\ref{eq:triangles-k}) are all the special triangles $\widetilde{T}_{n}$,
it is possible to read off all the special basis elements. Let us
divide into two cases and list the special basis elements using the
shorthands of Definition~\ref{def:notation-sb}. If $p\ge q$, then
they are 
\[
\boldsymbol{\theta}_{k,0},\boldsymbol{\theta}_{k+1,0},\cdots,\boldsymbol{\theta}_{k+q-1,0},\ \boldsymbol{\xi}_{0},\cdots,\boldsymbol{\xi}_{q-1},\ \boldsymbol{\zeta}_{k,0},\boldsymbol{\zeta}_{k+1,0},\cdots,\boldsymbol{\zeta}_{k+q-1,0}.
\]
If $p<q$, then they are 
\[
\{\boldsymbol{\theta}_{i,1}:i=0,\cdots,k-1\}\cup\{\boldsymbol{\theta}_{i,0}:i=k,\cdots,p-1\},\ \boldsymbol{\xi}_{k-p},\cdots,\boldsymbol{\xi}_{k-1},\ \boldsymbol{\zeta}_{0,0},\boldsymbol{\zeta}_{1,0},\cdots,\boldsymbol{\zeta}_{q-1,0}.\qedhere
\]
\end{proof}
Proposition~\ref{prop:special-triangles} is our main observation.
To show this, let us first show Lemma~\ref{lem:We-have-}; note that
we show a stronger statement in Lemma~\ref{lem:nzk}.
\begin{lem}
\label{lem:We-have-}We have $n_{\widetilde{z(u)}}(\widetilde{T}_{n})=0$
if and only if $n\in\{k,k+1,\cdots,k+p+q-1\}$.
\end{lem}

\begin{proof}
$(\Leftarrow)$: This follows from Lemma~\ref{lem:zigzag-basepoint}
($\widetilde{T}(ZZ_{k})$ is disjoint from $\widetilde{z(u)}$).

$(\Rightarrow)$: Let us divide into several cases. For $n<0$, the
points $(-\varepsilon,\varepsilon),(\varepsilon,\varepsilon)\in\mathbb{T}^{2}$
are contained in $T_{n}$, and so $n_{\widetilde{z(u)}}(\widetilde{T}_{n})\ge1$.
For $n=0,\cdots,k-1$, the point $\widetilde{z}_{n}-(2\varepsilon,0)$
is contained in $\widetilde{T}_{n}$, and so $n_{\widetilde{z(u)}}(\widetilde{T}_{n})\ge1$.
For $n=k+p+q,\cdots,2p+q-1$, $\widetilde{z}_{n-p-q}$ is contained
in $\widetilde{T}_{n}$ by Corollary~\ref{cor:zktkpq}, and so $n_{\widetilde{z(u)}}(\widetilde{T}_{n})\ge1$.
For $n\ge2p+q$, the points $(-\varepsilon,\varepsilon),(\varepsilon,\varepsilon)\in\mathbb{T}^{2}$
are contained in $T_{n}$, and so $n_{\widetilde{z(u)}}(\widetilde{T}_{n})\ge1$.
\end{proof}
\begin{prop}
\label{prop:special-triangles}The special triangles (\ref{eq:triangles-k}),
up to the identification from Convention~\ref{conv:identification},
are precisely the triangles $\widetilde{T}$ in $\widetilde{\mathbb{T}^{2}}$
such that at least two of the vertices of $\widetilde{T}$ are lifts
of special basis elements and $n_{\widetilde{z(u)}}(\widetilde{T})=0$.
In fact, all three vertices of each of the special triangles are lifts
of special basis elements.
\end{prop}

\begin{proof}
Since special basis elements are standard by Lemma~\ref{lem:hat-minus},
if at least two of the vertices of $\widetilde{T}$ are lifts of special
basis elements, then all three vertices of $\widetilde{T}$ are lifts
of standard basis elements by Lemma~\ref{lem:Let--be}. Hence, by
Lemma~\ref{lem:theta-unique} (compare the proof of Proposition~\ref{prop:triangle})
$\widetilde{T}=\widetilde{T}_{n}+(d,0)$ for some $n,d\in\mathbb{Z}$.
If $\ell,i,j\in\mathbb{Z}$ denote the integers $i=n\bmod p$, $j=n\bmod q$,
and such that $n=\ell p+i$, then the three vertices of $\widetilde{T}$
are lifts of $\boldsymbol{\theta}_{i,d+\ell}$, $\boldsymbol{\xi}_{j-p}$,
and $\boldsymbol{\zeta}_{i,d}$.

First, if $\widetilde{T}_{n}$ is a special triangle, then for at
least one of $\boldsymbol{\theta}_{i,d+\ell}$ and $\boldsymbol{\zeta}_{i,d}$
to be special, we must have $d\equiv0\bmod c_{i}$. Now, by Lemma~\ref{lem:We-have-},
it is sufficient to show that for any $d$ and $n\in\{k+\min(p,q),\cdots,k+\max(p,q)-1\}$,
only at most one of $\boldsymbol{\theta}_{i,d+\ell}$, $\boldsymbol{\xi}_{j-p}$,
and $\boldsymbol{\zeta}_{i,d}$ is special. Let us divide into two
cases: $p\ge q$ and $p<q$. If $p\ge q$, then $\boldsymbol{\theta}_{i,d+\ell}$
and $\boldsymbol{\zeta}_{i,d}$ are not special, since $i\notin\{k,\cdots,k+q-1\}\bmod p$.
If $p<q$, then $\boldsymbol{\xi}_{j-p}$ is not special, since $j-p\notin\{k-p,\cdots,k-1\}\bmod q$.
For both $\boldsymbol{\theta}_{i,d+\ell}$ and $\boldsymbol{\zeta}_{i,d}$
to be special, we need to have $d\equiv0\bmod c_{i}$, and that $\widetilde{T}_{n}$
is a special triangle.
\end{proof}
By Propositions~\ref{prop:lift-mu2}~and~\ref{prop:special-triangles},
for each of $\mu_{2}(\widehat{\psi}_{0r}^{k}\otimes\widehat{\psi}_{r\infty}^{k})$,
$\mu_{2}(\widehat{\psi}_{r\infty}^{k}\otimes\widehat{\psi}_{\infty0}^{k})$,
and $\mu_{2}(\widehat{\psi}_{\infty0}^{k}\otimes\widehat{\psi}_{0r}^{k})$,
the coefficient of each basis element $\boldsymbol{b}$ of $\widehat{CF}(\beta_{0}^{E_{0}^{k}},\beta_{\infty}^{E_{\infty}^{k}})$,
$\widehat{CF}(\beta_{r},\beta_{0}^{E_{0}^{k}})$, and $\widehat{CF}(\beta_{\infty}^{E_{\infty}^{k}},\beta_{r})$,
respectively, is zero if $\boldsymbol{b}$ is not special, and if
$\boldsymbol{b}$ is special, then it is the number of special triangles
$\widetilde{T}_{n}$ such that a vertex of $\widetilde{T}_{n}$ is
a lift of $\boldsymbol{b}$.

We claim that there are exactly two such triangles for each special
basis element $\boldsymbol{b}$. Let $H_{i}\subset\pi^{-1}(\beta_{0})$
and $L_{i}\subset\widetilde{\beta}_{r}^{0}$ be the line segments
from Definition~\ref{def:hi-li}. For $i\in\mathbb{Z}$, let $X_{i}$
be the connected component of $\pi^{-1}(\beta_{0})$ that $H_{i}$
is contained in, and let $C_{i}$ be the connected component of $\widetilde{\beta}_{r}^{0}$
that $L_{i}$ is contained in.

First, if $\boldsymbol{b}$ is a special basis element of $\widehat{CF}(\beta_{0}^{E_{0}^{k}},\beta_{\infty}^{E_{\infty}^{k}})$
or $\widehat{CF}(\beta_{r},\beta_{0}^{E_{0}^{k}})$, then let $i\in\mathbb{Z}$
be such that $\Phi_{k}(X_{i})$ contains a lift of $\boldsymbol{b}$.
Let $t\in\mathbb{Z}$ be such that the two endpoints of $H_{i}$ are
$\widetilde{\theta}_{i}$ and $\widetilde{\theta}_{t}$. Then the
two endpoints of $\Phi_{k}(H_{i})$ are $\widetilde{\theta}_{i+k}$
and $\widetilde{\theta}_{t+k}$, and the two wanted triangles are
$\widetilde{T}_{i+k}$ and $\widetilde{T}_{t+k}$.

Second, if $\boldsymbol{b}$ is a special basis element of $\widehat{CF}(\beta_{\infty}^{E_{\infty}^{k}},\beta_{r})$,
then let $i\in\mathbb{Z}$ be such that $\Phi_{k}(C_{i})$ contains
a lift of $\boldsymbol{b}$. Let $t\in\mathbb{Z}$ be such that the
two endpoints of $L_{i}$ are $\widetilde{\theta}_{i}$ and $\widetilde{\theta}_{t}$.
Then the two endpoints of $\Phi_{k}(L_{i})$ are $\widetilde{\theta}_{i+k}$
and $\widetilde{\theta}_{t+k}$, and the two wanted triangles are
$\widetilde{T}_{i+k}$ and $\widetilde{T}_{t+k}$.

This completes the proof that the triangle counting maps are zero
for the hat version.

\begin{rem}[Minus version]
\label{rem:minus-hard}In general, if $\psi_{0r}^{k},\psi_{r\infty}^{k},\psi_{\infty0}^{k}$
are simply the sum of the special basis elements, then the $\mu_{2}$'s
(Theorem~\ref{thm:gen-local-comp}~(\ref{enu:triangle})) do not
vanish for the minus version. In fact, this already fails for $(p,q,k)=(5,3,0)$
(compare Remark~\ref{rem:pq53}): let (recall the notations from
Definition~\ref{def:notation-sb})
\[
\psi_{0r}:=\boldsymbol{\theta}_{0,0}+\boldsymbol{\theta}_{1,0}+\boldsymbol{\theta}_{2,0},\ \psi_{r\infty}:=\boldsymbol{\xi}_{0}+\boldsymbol{\xi}_{1}+\boldsymbol{\xi}_{2},\ \psi_{\infty0}:=\boldsymbol{\zeta}_{0,0}+\boldsymbol{\zeta}_{1,0}+\boldsymbol{\zeta}_{2,0}
\]
be the sum of the special basis elements. Then, we have 
\[
\mu_{2}(\psi_{\infty0}\otimes\psi_{0r})\equiv U\boldsymbol{\xi}_{1}+U\boldsymbol{\xi}_{2}\mod{U^{2}}.
\]
Indeed, $T_{-3}$ contributes $U\boldsymbol{\xi}_{1}$ and $T_{10}$
contributes $U\boldsymbol{\xi}_{2}$.
\end{rem}

\subsection{\label{subsec:The-quadrilateral-counting}Quadrilateral counting
maps}

\begin{figure}[h]
\begin{centering}
\raisebox{-0.5\height}{\includegraphics[scale=0.7]{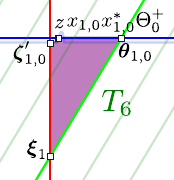}}\quad{}\raisebox{-0.5\height}{\includegraphics[scale=0.7]{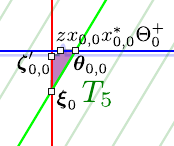}}\quad{}\raisebox{-0.5\height}{\includegraphics[scale=0.7]{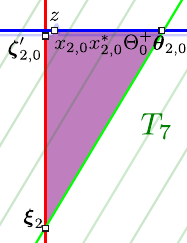}}\quad{}\quad{}\quad{}\quad{}\raisebox{-0.5\height}{\includegraphics[scale=0.7]{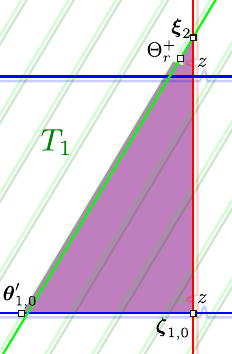}}\quad{}\raisebox{-0.5\height}{\includegraphics[scale=0.7]{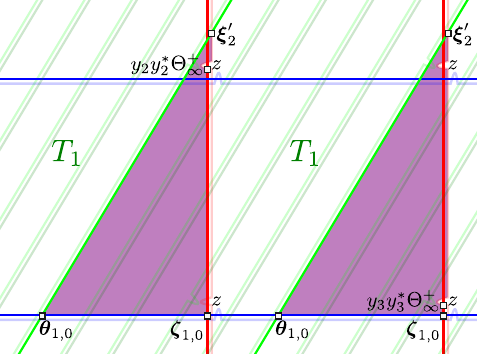}}\quad{}\raisebox{-0.5\height}{\includegraphics[scale=0.7]{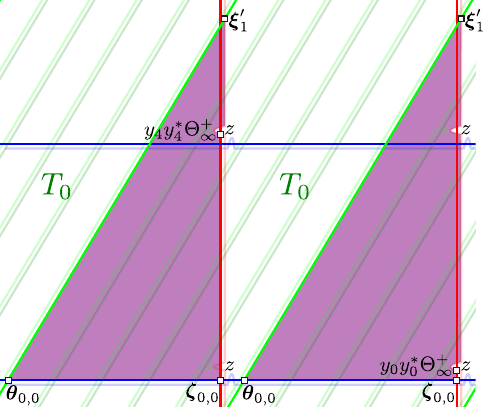}}\quad{}\raisebox{-0.5\height}{\includegraphics[scale=0.7]{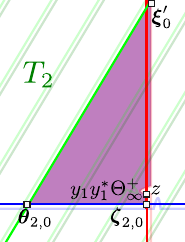}}\quad{}
\par\end{centering}
\caption{\label{fig:quadrilaterals}Lifts of the quadrilaterals that contribute
to $\mu_{3}$ for $(p,q,k)=(5,3,0)$. These quadrilaterals are small
perturbations of triangles; the corresponding triangles are written
in green. The first three quadrilaterals are for $\mu_{3}(\widehat{\psi}_{0r}^{0}\otimes\widehat{\psi}_{r\infty}^{0}\otimes\widehat{\psi}_{\infty0}^{0}{}')$;
the fourth quadrilateral is for $\mu_{3}(\widehat{\psi}_{r\infty}^{0}\otimes\widehat{\psi}_{\infty0}^{0}\otimes\widehat{\psi}_{0r}^{0}{}')$;
the last five quadrilaterals are for $\mu_{3}(\widehat{\psi}_{\infty0}^{0}\otimes\widehat{\psi}_{0r}^{0}\otimes\widehat{\psi}_{r\infty}^{0}{}')$.}
\end{figure}

In this subsection, we prove 
\begin{equation}
\mu_{3}(\widehat{\psi}_{0r}^{k}\otimes\widehat{\psi}_{r\infty}^{k}\otimes\widehat{\psi}_{\infty0}^{k}{}')={\rm Id}_{E_{0}^{k}}\Theta_{0}^{+},\ \mu_{3}(\widehat{\psi}_{r\infty}^{k}\otimes\widehat{\psi}_{\infty0}^{k}\otimes\widehat{\psi}_{0r}^{k}{}')=\Theta_{r}^{+},\ \mu_{3}(\widehat{\psi}_{\infty0}^{k}\otimes\widehat{\psi}_{0r}^{k}\otimes\widehat{\psi}_{r\infty}^{k}{}')={\rm Id}_{E_{\infty}^{k}}\Theta_{\infty}^{+}.\label{eq:mu3-hat}
\end{equation}
By Lemma~\ref{lem:hat-minus}, this shows that Theorem~\ref{thm:gen-local-comp}~(\ref{enu:quadrilateral})
holds whenever $u_{i,\ell},v_{j},w_{i,\ell}$ satisfy Equation~(\ref{eq:uvw-modu}).

Recall from Subsection~\ref{subsec:Lifts-quad} that Maslov index
$-1$ quadrilaterals correspond to pairs $(T_{n},\overline{v})$.
By Proposition~\ref{prop:special-triangles}, the special triangles
(\ref{eq:triangles-k}), up to the identification from Convention~\ref{conv:identification},
are precisely the triangles $\widetilde{T}$ in $\widetilde{\mathbb{T}^{2}}$
such that all its vertices are lifts of special basis elements and
$n_{\widetilde{z(u)}}(\widetilde{T})=0$. Hence, by Proposition~\ref{prop:main-mu3},
for each basis element $\boldsymbol{b}$ of $\boldsymbol{CF}^{-}(\beta_{0}^{E_{0}^{k}},{\beta_{0}'}^{E_{0}^{k}})$,
$\boldsymbol{CF}^{-}(\beta_{r},\beta_{r}')$, resp.\ $\boldsymbol{CF}^{-}(\beta_{\infty}^{E_{\infty}^{k}},{\beta_{\infty}'}^{E_{\infty}^{k}})$,
the coefficient of $\boldsymbol{b}$ in $\mu_{3}(\widehat{\psi}_{0r}^{k}\otimes\widehat{\psi}_{r\infty}^{k}\otimes\widehat{\psi}_{\infty0}^{k}{}')$,
$\mu_{3}(\widehat{\psi}_{r\infty}^{k}\otimes\widehat{\psi}_{\infty0}^{k}\otimes\widehat{\psi}_{0r}^{k}{}')$,
resp.\ $\mu_{3}(\widehat{\psi}_{\infty0}^{k}\otimes\widehat{\psi}_{0r}^{k}\otimes\widehat{\psi}_{r\infty}^{k}{}')$,
is zero if $\boldsymbol{b}$ is not standard, and if $\boldsymbol{b}$
is standard, then it is the number of $(T_{n},\overline{v})$ such
that $\widetilde{T}_{n}$ is a special triangle and $\widehat{\pi}(\overline{v})$
is a lift of $\boldsymbol{b}$. Thus, Equation~(\ref{eq:mu3-hat})
is equivalent to that there are an odd number of such $(T_{n},\overline{v})$.
Figure~\ref{fig:quadrilaterals} shows lifts of all the quadrilaterals
that contribute to $\mu_{3}$ for $(p,q,k)=(5,3,0)$, the corresponding
triangle $T_{n}$, and the standard basis element $\boldsymbol{b}$.

To show that there are an odd number of such $(T_{n},\overline{v})$,
it is convenient to consider the $\pi^{-1}(\beta_{0})$-, $\pi^{-1}(\beta_{r})$-,
and $\pi^{-1}(\beta_{\infty})$-boundaries of the triangles $\widetilde{T}_{n}$
as one-chains with $\mathbb{Z}/2\mathbb{Z}$-coefficients, i.e.\ as
elements of 
\[
H_{1}(\pi^{-1}(\beta_{0}\cup\beta_{r}\cup\beta_{\infty}),\pi^{-1}((\beta_{0}\cap\beta_{r})\cup(\beta_{r}\cap\beta_{\infty})\cup(\beta_{\infty}\cap\beta_{0}));\mathbb{Z}/2\mathbb{Z}).
\]
Let $\partial_{0}$, $\partial_{r}$, resp.\ $\partial_{\infty}$
be the sum of the $\pi^{-1}(\beta_{0})$-, $\pi^{-1}(\beta_{r})$-,
resp.\ $\pi^{-1}(\beta_{\infty})$-boundaries of the special triangles.
Then, for each standard basis element $\boldsymbol{b}=x_{i,\ell}x_{i,\ell}^{\ast}\Theta_{0}^{+}$,
$\Theta_{r}^{+}$, resp.\ $y_{i}y_{i}^{\ast}\Theta_{\infty}^{+}$,
if $\widetilde{\boldsymbol{b}}\subset\widetilde{\mathbb{T}^{2}}$
is the set of lifts of $\boldsymbol{b}$, then we are left to show
that $\widetilde{\boldsymbol{b}}\cap\partial_{0}$, $\widetilde{\boldsymbol{b}}\cap\partial_{r}$,
resp.\ $\widetilde{\boldsymbol{b}}\cap\partial_{\infty}$ is odd.
We consider these three cases separately.

\subsubsection*{Case $\boldsymbol{b}=x_{i,\ell}x_{i,\ell}^{\ast}\Theta_{0}^{+}$}

The set of lifts of $x_{i,\ell}x_{i,\ell}^{\ast}\Theta_{0}^{+}$ intersects
$ZZ_{k}\cap\pi^{-1}(\beta_{0})$ once. Let us show $\partial_{0}=ZZ_{k}\cap\pi^{-1}(\beta_{0})$.
For each $i$ such that $s_{i}\neq0$, if we let $t\in\mathbb{Z}$
be such that the two endpoints of $\Phi_{k}(H_{i})$ are $\widetilde{\theta}_{i+k}$
and $\widetilde{\theta}_{t+k}$, then $\pi^{-1}(\beta_{0})$-boundaries
of the triangles $\widetilde{T}_{i+k}$ and $\widetilde{T}_{t+k}$
add up to $\Phi_{k}(H_{i})$.

\subsubsection*{Case $\boldsymbol{b}=\Theta_{r}^{+}$}

Let us first show that $\partial_{r}=ZZ_{k}\cap\pi^{-1}(\beta_{r})$:
let $i=0,\cdots,p-1$ be such that $s_{i}\neq0$, and let $t\in\mathbb{Z}$
be such that the two endpoints of $\Phi_{k}(L_{i})$ are $\widetilde{\theta}_{i+k}$
and $\widetilde{\theta}_{t+k}$. Then the $\pi^{-1}(\beta_{r})$-boundaries
of the triangles $\widetilde{T}_{i+k}$ and $\widetilde{T}_{t+k}$
add up to $\Phi_{k}(L_{i})$.

Now, let us show that the image of $ZZ_{k}$ in $\mathbb{T}^{2}$
is one copy of $\beta_{r}$. It is sufficient to show this for $k=0$
since $ZZ_{k}$ is obtained by translating $ZZ_{0}$. In this case,
project the $L_{i}$'s to $\mathbb{R}^{2}/((p,0)\mathbb{Z}\oplus(0,1)\mathbb{Z})$:
since the endpoints of each $H_{i}$ are identified, the union of
the projections $L_{i}$'s is exactly the projection of $\widetilde{\beta}_{r}^{0}$,
which projects to one copy of $\beta_{r}$ in $\mathbb{T}^{2}$. Hence,
in particular, the set of lifts of $\Theta_{r}^{+}$ intersects $ZZ_{k}$
once.

\subsubsection*{Case $\boldsymbol{b}=y_{i}y_{i}^{\ast}\Theta_{\infty}^{+}$}

The set of lifts of $y_{i}y_{i}^{\ast}\Theta_{\infty}^{+}$ intersects
the $y$-axis $\widetilde{\beta}_{\infty}^{0}$ once. Let us show
$\partial_{\infty}=\widetilde{\beta}_{\infty}^{0}$. Let $i$ be such
that $s_{i}\neq0$, and let $t\in\mathbb{Z}$, $b,c\in\mathbb{Z}/p\mathbb{Z}$
be such that the two endpoints of $\Phi_{k}(L_{i})$ are $\widetilde{\theta}_{i+k}$
and $\widetilde{\theta}_{t+k}$, and that they are on $\widetilde{\beta}_{0}^{b}$
and $\widetilde{\beta}_{0}^{c}$, respectively. Let $V_{i}\subset\widetilde{\beta}_{\infty}^{0}$
be the line segment given by the path obtained by starting at $\widetilde{\beta}_{0}^{c}\cap\widetilde{\beta}_{\infty}^{0}$
and traversing in the $-y$-direction until $\widetilde{\beta}_{0}^{b}\cap\widetilde{\beta}_{\infty}^{0}$.
Then the $\pi^{-1}(\beta_{\infty})$-boundaries of the triangles $\widetilde{T}_{i+k}$
and $\widetilde{T}_{t+k}$ add up to $V_{i}$.

This completes the proof of Equation~(\ref{eq:mu3-hat}).

\section{\label{sec:The-minus-version}The minus version}

In this section, we prove our main local computation, Theorem~\ref{thm:gen-local-comp}.
More precisely, we show that there exist $u_{i,\ell},v_{j},w_{i,\ell}\in\mathbb{F}\llbracket U\rrbracket$
such that Equation~(\ref{eq:uvw-modu}) and Theorem~\ref{thm:gen-local-comp}~(\ref{enu:triangle})
hold. Recall that Theorem~\ref{thm:gen-local-comp} indeed follows
from this: Theorem~\ref{thm:gen-local-comp}~(\ref{enu:bigon})
holds since $\mu_{1}\equiv0$, and we have shown Theorem~\ref{thm:gen-local-comp}~(\ref{enu:quadrilateral})
in Subsection~\ref{subsec:The-quadrilateral-counting}.

Recall the triangles $\widetilde{T}_{n}$ and the intersection points
$\widetilde{\theta}_{n}$ from Definition~\ref{def:theta-tilde},
and the standard basis elements $\boldsymbol{\theta}_{i,\ell},\boldsymbol{\xi}_{j},\boldsymbol{\zeta}_{i,\ell}$
from Definition~\ref{def:notation-sb}.

\subsection{\label{subsec:triangle-k=00003D0}Triangle counting maps for $k=0$}

In Proposition~\ref{prop:triangle}, we wrote down explicit formulas
for the three triangle counting maps: Equations~(\ref{eq:0rinf}),
(\ref{eq:rinf0}), and (\ref{eq:inf0r}). In Subsubsection~\ref{subsec:triangle0rinf-rinf0},
we show that for any $(u_{i,\ell})$ and $(w_{i,\ell})$, there exist
$(v_{j})$ such that Equation~(\ref{eq:uvw-modu}) holds for $(v_{j})$
and Equations~(\ref{eq:0rinf})~and~(\ref{eq:rinf0}) are zero.
In Subsubsection~\ref{subsec:triangleinf0r}, we show that there
exist $(u_{i,\ell})$ and $(w_{i,\ell})$ such that Equation~(\ref{eq:uvw-modu})
holds for $(u_{i,\ell})$ and $(w_{i,\ell})$, and Equation~(\ref{eq:inf0r})
is zero. As an example, we carry out the argument for $(p,q)=(5,3)$
in Examples~\ref{exa:61example}~and~\ref{exa:53uv} and show Remark~\ref{rem:pq53}.

In this subsection, $\ell,i,j\in\mathbb{Z}$ always denote the integers
such that $i=n\bmod p\in\{0,\cdots,p-1\}$, $j=n\bmod q\in\{0,\cdots,q-1\}$,
and $n=\ell p+i$.

\subsubsection{\label{subsec:triangle0rinf-rinf0}The triangle counting maps (\ref{eq:0rinf})~and~(\ref{eq:rinf0})}

Consider the $\mathbb{F}\llbracket U\rrbracket$-linear map
\[
F:=\sum_{n\in\mathbb{Z}}U^{z_{n}}f_{i,\ell}e_{j-p}^{\ast}:\bigoplus_{j=0}^{q-1}e_{j}\mathbb{F}\llbracket U\rrbracket\to\bigoplus_{i=0}^{p-1}\bigoplus_{\ell=0}^{s_{i}-1}f_{i,\ell}\mathbb{F}\llbracket U\rrbracket.
\]
Here, the indices of $e$ are interpreted modulo $q$, the indices
of $f$ lie in the index set $I$ from Definition~\ref{def:index-set},
and let $f_{i,\ell}:=0$ if $s_{i}=0$. Then, Equation~(\ref{eq:0rinf}),
resp.\ (\ref{eq:rinf0}) can be written as $F(\sum_{j=0}^{q-1}v_{j}e_{j})$,
where we identify $f_{i,\ell}$ with 
\[
\sum_{d=0}^{s_{i}-1}u_{i,d+\ell}\boldsymbol{\zeta}_{i,d},\ \mathrm{resp.}\ \sum_{d=0}^{s_{i}-1}w_{i,d}\boldsymbol{\theta}_{i,d+\ell}.
\]

Let us show that $\ker F\neq0$. Equivalently, we show ${\rm rank}_{\mathbb{F}\llbracket U\rrbracket}\ {\rm ker}F\neq0$.
The trick is to use an involution to define $\mathbb{F}\llbracket U\rrbracket$-linear
subspaces 
\[
A\le{\rm dom}F,\ B\le{\rm codom}F
\]
such that $F(A)\le B$ and ${\rm rank}B<{\rm rank}A$. Although we
check everything algebraically (which we postpone to Appendix~\ref{sec:Checks-for-Subsubsection}),
this involution has a satisfying ``picture interpretation'': see
Subsubsection~\ref{subsec:The-involution}.

Consider the involution $n\mapsto p+q-1-n$ on $\mathbb{Z}$. This
involution has one fixed point if $p+q$ is odd, and has no fixed
points if $p+q$ is even. If $p+q$ is odd, let $\ell_{0},i_{0}\in\mathbb{Z}$
be such that $(p+q-1)/2=\ell_{0}p+i_{0}$ and $i_{0}\in\{0,\cdots,p-1\}$.
Then, we have (by Lemmas~\ref{lem:a0}~and~\ref{lem:a1})
\begin{multline*}
F=\sum_{n\ge\left\lfloor \frac{p+q-1}{2}\right\rfloor +1}U^{z_{n}}\left(f_{i,\ell}e_{j-p}^{\ast}+f_{q-i-1,-\ell}e_{-j-1}^{\ast}\right)\\
+\begin{cases}
U^{z_{(p+q-1)/2}}f_{i_{0},\ell_{0}}e_{(-p+q-1)/2}^{\ast} & {\rm if}\ p+q\ {\rm is}\ {\rm odd}\\
0 & {\rm if}\ p+q\ {\rm is}\ {\rm even}
\end{cases}.
\end{multline*}
Also, if $p+q$ is odd, then $U^{z_{(p+q-1)/2}}f_{i_{0},\ell_{0}}e_{(-p+q-1)/2}^{\ast}\neq0$
(i.e.\ $s_{i_{0}}\neq0$) if and only if $q>p$ by Lemma~\ref{lem:a3}.

Now, we define $A$ and $B$, and compute their ranks.

\subsubsection*{Defining $A$ and computing  ${\rm rank}A$}

Let 
\[
A:={\rm span}\left\langle e_{j}+e_{-p-j-1}:j=0,\cdots,q-1\right\rangle \oplus{\rm span}\left\langle e_{j}:j\equiv-p-j-1\mod q\right\rangle .
\]

\textbf{If $q=2q'+1$ is odd}, then the involution $j\mapsto-(j+p+1)$
on $\mathbb{Z}/q\mathbb{Z}$ has  one fixed point. Hence, ${\rm rank}A=q'+1$.

\textbf{If $q=2q'$ is even}, then $p$ is odd (since $p$ and $q$
are coprime), and so the above involution has  two fixed points. Hence,
${\rm rank}A=q'+1$.

\subsubsection*{Defining $B$ and computing ${\rm rank}B$}

We divide into several cases. First, let 
\[
B':={\rm span}\left\langle f_{i,\ell}+f_{q-i-1,-\ell}:i=0,\cdots,p-1,\ \ell=0,\cdots,s_{i}-1\right\rangle .
\]

\textbf{If $p+q$ is even}, then let $B=B'$. Since $p$ and $q$
are coprime, they are both odd. Let $q=2q'+1$. Then, the involution
$i\mapsto q-i-1$ on $\mathbb{Z}/p\mathbb{Z}$ has one fixed point,
$q'$. By Lemma~\ref{lem:a2}, $s_{q'}$ is odd, and so the involution
$(i,\ell)\mapsto(q-i-1,-\ell)$ on the index set $I$ has one fixed
point. Hence, ${\rm rank}B=(q-1)/2=q'$.

\textbf{If $p+q$ is odd}, let 
\[
B:=B'\oplus{\rm span}\left\langle f_{i_{0},\ell_{0}}\right\rangle .
\]
Recall that $f_{i_{0},\ell_{0}}\neq0$ (i.e.\ $s_{i_{0}}\neq0$)
if and only if $p<q$ (Lemma~\ref{lem:a3}). We further divide into
a few cases.
\begin{itemize}
\item $p$ is even, $q=2q'+1$ is odd: the involution $i\mapsto q-i-1$
on $\mathbb{Z}/p\mathbb{Z}$ has  two fixed points $q'$ and $i_{0}$.
We have the following by Lemmas~\ref{lem:a2}~and~\ref{lem:a3}.
\begin{itemize}
\item If $p>q$, the involution $(i,\ell)\mapsto(q-i-1,-\ell)$ on $I$
has  one fixed point (and its first coordinate is $q'$). Hence, ${\rm rank}B=(q-1)/2=q'$.
\item If $p<q$, the above involution on $I$ has  three fixed points, and
one of them is $(i_{0},\ell_{0})$: $q'$ contributes one and $i_{0}$
contributes two. Hence, ${\rm rank}B=(q-3)/2+1=q'$.
\end{itemize}
\item $p$ is odd, $q=2q'$ is even: the involution $i\mapsto q-i-1$ on
$\mathbb{Z}/p\mathbb{Z}$ has  one fixed point $i_{0}$. We have the
following by Lemma~\ref{lem:a3}.
\begin{itemize}
\item If $p>q$, the involution $(i,\ell)\mapsto(q-i-1,-\ell)$ on $I$
has no fixed points. Hence, ${\rm rank}B=q/2=q'$.
\item If $p<q$, the above involution on $I$ has  two fixed points, and
one of them is $(i_{0},\ell_{0})$. Hence, ${\rm rank}B=(q-2)/2+1=q'$.
\end{itemize}
\end{itemize}
In all of the cases, we have 
\[
{\rm rank}A={\rm rank}B+1>{\rm rank}B\ {\rm and}\ F(A)\le B,
\]
and so ${\rm rank}({\rm ker}F)\neq0$.

Now, let $\sum_{j=0}^{q-1}v_{j}e_{j}\neq0$ be any nonzero element
of $\ker F$. By dividing by some $U^{m}$ if necessary, we can let
$\sum_{j=0}^{q-1}v_{j}e_{j}\not\equiv0\bmod U$. Let us show that
Equation~(\ref{eq:uvw-modu}) holds for $(v_{j})$. Since $z_{n}=0$
if and only if $n\in\{0,\cdots,p+q-1\}$, we have 
\[
F\otimes_{\mathbb{F}\llbracket U\rrbracket}\mathbb{F}=\sum_{i=0}^{\min(p,q)-1}\sum_{\ell=0}^{s_{i}}f_{i,\ell}e_{(\ell-1)p+i}^{\ast}.
\]
Hence, $F(\sum_{j=0}^{q-1}v_{j}e_{j})\equiv0\bmod U$ if and only
if for all $i=0,\cdots,\min(p,q)-1$, we have $v_{-p+i}\equiv v_{(s_{i}-1)p+i}\bmod U$
and for all $t=0,\cdots,s_{i}-2$, we have $v_{tp+i}\equiv0\bmod U$.
By Lemma~\ref{lem:a4}, we have 
\[
\sum_{j=0}^{q-1}v_{j}e_{j}\equiv\sum_{i=0}^{\min(p,q)-1}e_{-p+i}=e_{-1}+e_{-2}+\cdots+e_{-\min(p,q)}\bmod U.
\]
Hence, Equation~(\ref{eq:uvw-modu}) holds for $(v_{j})$.
\begin{rem}
In fact, we have shown that ${\rm ker}(F\otimes_{\mathbb{F}\llbracket U\rrbracket}\mathbb{F})$
is one dimensional and that it is spanned by $e_{-1}+e_{-2}+\cdots+e_{-\min(p,q)}$.
Since 
\[
0<{\rm rank}_{\mathbb{F}\llbracket U\rrbracket}\ {\rm ker}F\le\dim_{\mathbb{F}}\ker\left(F\otimes_{\mathbb{F}\llbracket U\rrbracket}\mathbb{F}\right)=1,
\]
we have ${\rm rank}_{\mathbb{F}\llbracket U\rrbracket}\ {\rm ker}F=1$.
\end{rem}

\begin{example}
\label{exa:61example}Let us carry out the argument of this subsubsection
for $(p,q)=(5,3)$ and check that the $(v_{j})$ from Remark~\ref{rem:pq53}
works. With respect to the basis $(e_{0},e_{1},e_{2})$ and $(f_{0,0},f_{1,0},f_{2,0})$,
the linear map $F$ is given by the matrix
\begin{equation}
\begin{pmatrix}\sum_{m\in\mathbb{Z}}U^{z_{15m+5}} & \sum_{m\in\mathbb{Z}}U^{z_{15m}} & \sum_{m\in\mathbb{Z}}U^{z_{15m+10}}\\
\sum_{m\in\mathbb{Z}}U^{z_{15m+11}} & \sum_{m\in\mathbb{Z}}U^{z_{15m+6}} & \sum_{m\in\mathbb{Z}}U^{z_{15m+1}}\\
\sum_{m\in\mathbb{Z}}U^{z_{15m+2}} & \sum_{m\in\mathbb{Z}}U^{z_{15m+12}} & \sum_{m\in\mathbb{Z}}U^{z_{15m+7}}
\end{pmatrix}.\label{eq:matrix}
\end{equation}
By Lemma~\ref{lem:a0}, $z_{n}=z_{7-n}.$ Hence, $\sum_{m\in\mathbb{Z}}U^{z_{15m+r}}=\sum_{m\in\mathbb{Z}}U^{z_{15m+7-r}}$
and $\sum_{m\in\mathbb{Z}}U^{z_{15m+11}}=0$, and so if we let 
\[
a:=\sum_{m\in\mathbb{Z}}U^{z_{15m+5}},\ b:=\sum_{m\in\mathbb{Z}}(U^{z_{15m}}+U^{z_{15m+10}}),
\]
then we have 
\[
F(e_{0})=a(f_{0,0}+f_{2,0}),\ F(e_{1}+e_{2})=b(f_{0,0}+f_{2,0}),
\]
and so $F(be_{0}+ae_{1}+ae_{2})=0$. Now, using the description of
$z_{n}$ from the proof of Lemma~\ref{lem:a0}, we can check that
for $m\ge0$ and $r\in\{0,\cdots,14\}$, we have 
\begin{equation}
z_{15m+r+7}=\frac{15m^{2}+(2r+7)m}{2}+z_{r+7},\label{eq:z53}
\end{equation}
where $z_{7},\cdots,z_{21}$ are $0,1,1,1,2,2,3,4,4,5,6,7,8,9,10$,
respectively. Hence, the formulas for $a,b$ in (\ref{eq:abc}) follow.
\end{example}

\subsubsection{\label{subsec:triangleinf0r}The triangle counting map (\ref{eq:inf0r})}

Consider the adjoint $F^{\ast}$ of $F$, i.e.\ 
\[
F^{\ast}=\sum_{n\in\mathbb{Z}}U^{z_{n}}e_{j-p}^{\ast}f_{i,\ell}:\bigoplus_{i=0}^{p-1}\bigoplus_{\ell=0}^{s_{i}-1}f_{i,\ell}^{\ast}\mathbb{F}\llbracket U\rrbracket\to\bigoplus_{j=0}^{q-1}e_{j}^{\ast}\mathbb{F}\llbracket U\rrbracket.
\]
Then, Equation~(\ref{eq:inf0r}) can be written as
\[
F^{\ast}\left(\sum_{i=0}^{p-1}\sum_{\ell=0}^{s_{i}-1}\left(\sum_{d=0}^{s_{i}-1}u_{i,d+\ell}w_{i,d}\right)f_{i,\ell}^{\ast}\right),
\]
where we identify $e_{j}^{\ast}$ with $\boldsymbol{\xi}_{j}$.

Since the domain and codomain of $F^{\ast}$ have the same rank, 
\[
{\rm rank}_{\mathbb{F}\llbracket U\rrbracket}\ {\rm ker}F^{\ast}={\rm rank}_{\mathbb{F}\llbracket U\rrbracket}\ {\rm ker}F=1.
\]
Hence, there exists a nonzero $\sum_{i,\ell}t_{i,\ell}f_{i,\ell}^{\ast}\in{\rm ker}F^{\ast}$.
Now, let $(u_{i,\ell})$, $(w_{i,\ell})$ be such that 
\[
\sum_{d=0}^{s_{i}-1}u_{i,d+\ell}w_{i,d}=t_{i,\ell}.
\]
For instance, one can let
\begin{equation}
u_{i,\ell}=\begin{cases}
1 & {\rm if}\ \ell=0\\
0 & {\rm otherwise}
\end{cases},\ w_{i,\ell}=t_{i,-\ell},\ \mathbf{or}\ u_{i,\ell}=t_{i,\ell},\ w_{i,\ell}=\begin{cases}
1 & {\rm if}\ \ell=0\\
0 & {\rm otherwise}
\end{cases}.\label{eq:options6}
\end{equation}
These correspond to the two cases of Equation~(\ref{eq:options}).

Now, let $\sum_{i,\ell}t_{i,\ell}f_{i,\ell}^{\ast}\neq0$ be any nonzero
element of $\ker F^{\ast}$. By dividing by some $U^{m}$ if necessary,
we can let $\sum_{i,\ell}t_{i,\ell}f_{i,\ell}^{\ast}\not\equiv0\bmod U$.
Let us show that if we let $(u_{i,\ell})$ and $(w_{i,\ell})$ be
as in Equation~(\ref{eq:options6}), then Equation~(\ref{eq:uvw-modu})
holds for $(u_{i,\ell})$ and $(w_{i,\ell})$. Since $z_{n}=0$ if
and only if $n\in\{0,\cdots,p+q-1\}$, we have 
\[
F^{\ast}\otimes_{\mathbb{F}\llbracket U\rrbracket}\mathbb{F}=\sum_{i=0}^{\min(p,q)-1}\sum_{\ell=0}^{s_{i}}e_{(\ell-1)p+i}^{\ast}f_{i,\ell}.
\]
Hence, $F^{\ast}(\sum_{i,\ell}t_{i,\ell}f_{i,\ell}^{\ast})\equiv0\bmod U$
if and only if for all $i=0,\cdots,\min(p,q)-1$, we have $t_{i,1}=\cdots=t_{i,s_{i}-1}=0$,
and if $i'\in\{0,\cdots,\min(p,q)-1\}$ is such that $i'\equiv s_{i}p+i\bmod q$
(which exists by Lemma~\ref{lem:a4}), then $t_{i,0}=t_{i',0}$.
By Lemma~\ref{lem:a4}, we have
\[
\sum_{i,\ell}t_{i,\ell}f_{i,\ell}^{\ast}\equiv\sum_{i=0}^{\min(p,q)-1}f_{i,0}^{\ast}\bmod U,
\]
and so Equation~(\ref{eq:uvw-modu}) holds for $(u_{i,\ell})$ and
$(w_{i,\ell})$.
\begin{example}
\label{exa:53uv}Let us check that Equation~(\ref{eq:36}) must hold,
and that the $(u_{i,\ell})$ and $(w_{i,\ell})$ from Remark~\ref{rem:pq53}
work. Let $(p,q)=(5,3)$. Then, with respect to the bases $(f_{0,0}^{\ast},f_{1,0}^{\ast},f_{2,0}^{\ast})$
and $(e_{0}^{\ast},e_{1}^{\ast},e_{2}^{\ast})$, the linear map $F^{\ast}$
is given by the transpose of the matrix (\ref{eq:matrix}). Hence,
\[
F^{\ast}(f_{0,0}^{\ast}+f_{2,0}^{\ast})=b(e_{1}^{\ast}+e_{2}^{\ast}),\ F^{\ast}(f_{1,0}^{\ast})=c(e_{1}^{\ast}+e_{2}^{\ast}),
\]
where $c:=\sum_{m\in\mathbb{Z}}U^{z_{15m+6}}$. Hence, $F^{\ast}(cf_{0,0}^{\ast}+bf_{1,0}^{\ast}+cf_{2,0}^{\ast})=0$
and the formula for $c$ in (\ref{eq:abc}) follows from Equation~(\ref{eq:z53}).
\end{example}

\subsection{\label{subsec:Triangle-maps-for}Triangle counting maps for general
$k$}

In this subsection, we consider the triangle counting maps for general
$k=0,\cdots,p-1$. The point is that we are doing the exact same computation
as when $k=0$ (Proposition~\ref{prop:general-k-same}).

\begin{prop}
\label{prop:general-k-same}To simplify the notation, let $\gamma_{0}^{F_{0}^{t}},\gamma_{1}^{F_{1}^{t}},\gamma_{2}^{F_{2}^{t}}$
be a cyclic permutation of $\beta_{0}^{E_{0}^{t}},\beta_{r}^{\mathbb{F}\llbracket U\rrbracket},\beta_{\infty}^{E_{\infty}^{t}}$
for $t=0,k$. Let 
\[
\boldsymbol{CF}_{{\rm std}}^{-}(\gamma_{u}^{F_{u}^{t}},\gamma_{v}^{F_{v}^{t}})\le\boldsymbol{CF}^{-}(\gamma_{u}^{F_{u}^{t}},\gamma_{v}^{F_{v}^{t}})
\]
be the $\mathbb{F}\llbracket U\rrbracket$-linear subspace that is
spanned by the standard basis elements. Define 
\begin{multline*}
m:\boldsymbol{CF}_{{\rm std}}^{-}(\gamma_{u}^{F_{u}^{0}},\gamma_{v}^{F_{v}^{0}})\to\boldsymbol{CF}_{{\rm std}}^{-}(\gamma_{u}^{F_{u}^{k}},\gamma_{v}^{F_{v}^{k}}):\\
\boldsymbol{\theta}_{i,\ell}\mapsto\begin{cases}
\boldsymbol{\theta}_{i+k,\ell} & {\rm if}\ i\in\{0,\cdots,p-k-1\}\\
\boldsymbol{\theta}_{i+k,\ell+1} & {\rm if}\ i\in\{p-k,\cdots,p-1\}
\end{cases},\ \boldsymbol{\xi}_{j}\mapsto\boldsymbol{\xi}_{j+k},\ \boldsymbol{\zeta}_{i,\ell}\mapsto\boldsymbol{\zeta}_{i+k,\ell}.
\end{multline*}
Then, the following commutes: 
\[\begin{tikzcd}[ampersand replacement=\&]
	{\boldsymbol{CF}^{-} _{\mathrm{std}}(\gamma_{0}^{F_{0}^{0}},\gamma_{1}^{F_{1}^{0}}) \otimes \boldsymbol{CF}^{-} _{\mathrm{std}} (\gamma_{1}^{F_{1}^{0}},\gamma_{2}^{F_{2}^{0}}) } \& {\boldsymbol{CF}^{-}_{\mathrm{std}} (\gamma_{0}^{F_{0}^{0}},\gamma_{2}^{F_{2}^{0}})} \\
	{\boldsymbol{CF}^{-} _{\mathrm{std}}(\gamma_{0}^{F_{0}^{k}},\gamma_{1}^{F_{1}^{k}}) \otimes \boldsymbol{CF}^{-}_{\mathrm{std}} (\gamma_{1}^{F_{1}^{k}},\gamma_{2}^{F_{2}^{k}}) } \& {\boldsymbol{CF}^{-}_{\mathrm{std}} (\gamma_{0}^{F_{0}^{k}},\gamma_{2}^{F_{2}^{k}})}
	\arrow["{\mu_2}", from=1-1, to=1-2]
	\arrow["m", from=1-1, to=2-1]
	\arrow["m", from=1-2, to=2-2]
	\arrow["{\mu_2}", from=2-1, to=2-2]
\end{tikzcd}\]
\end{prop}

\begin{proof}
Let $\widetilde{z(k)}$ be the set $\widetilde{z(u)}\subset\widetilde{\mathbb{T}^{2}}$
from Subsection~\ref{subsec:cover} that corresponds to the monodromy
of $E_{\infty}^{k}$ (Definition~\ref{def:actual-local-system}).
Then, the key statement is 
\begin{equation}
n_{\widetilde{z(k)}}(\widetilde{T}_{n+k})=n_{\pi^{-1}(z)}(\widetilde{T}_{n}).\label{eq:triangle-basepoint}
\end{equation}
Equation~(\ref{eq:triangle-basepoint}) can be checked algebraically;
we write down an algebraic proof in Lemma~\ref{lem:nzk}. However,
there is a ``picture proof'' that we find more satisfying; we give
this proof in Subsubsection~\ref{subsec:zk}.

To show that the proposition follows from Equation~(\ref{eq:triangle-basepoint}),
we use Proposition~\ref{prop:lift-mu2}. Let $\ell,i,j,\ell',i',j'\in\mathbb{Z}$
be such that $i:=n\bmod p$, $j:=n\bmod q$, $n=\ell p+i$, $i':=(n+k)\bmod p$,
$j':=(n+k)\bmod q$, and $n+k=\ell'p+i'$. Then, as in the proof of
Proposition~\ref{prop:triangle}, the triangle $T_{n}$ in $\mathbb{T}^{2}$
has a $\{\boldsymbol{b}_{0},\boldsymbol{b}_{1},\boldsymbol{b}_{2}\}$-lift
for standard basis elements $\boldsymbol{b}_{0},\boldsymbol{b}_{1},\boldsymbol{b}_{2}$
for $t=0$ if and only if
\[
\{\boldsymbol{b}_{0},\boldsymbol{b}_{1},\boldsymbol{b}_{2}\}=\{\boldsymbol{\theta}_{i,d+\ell},\boldsymbol{\xi}_{j-p},\boldsymbol{\zeta}_{i,d}\}
\]
for some $d=0,\cdots,s_{i}-1$ and $T_{n+k}$ has a $\{\boldsymbol{b}_{0},\boldsymbol{b}_{1},\boldsymbol{b}_{2}\}$-lift
for standard basis elements $\boldsymbol{b}_{0},\boldsymbol{b}_{1},\boldsymbol{b}_{2}$
for $t=k$ if and only if
\[
\{\boldsymbol{b}_{0},\boldsymbol{b}_{1},\boldsymbol{b}_{2}\}=\{\boldsymbol{\theta}_{i',d+\ell'},\boldsymbol{\xi}_{j'-p},\boldsymbol{\zeta}_{i',d}\}
\]
for some $d=0,\cdots,s_{i}-1$. Hence, the proposition follows.
\end{proof}
Since $m(\psi_{0r}^{0})=\psi_{0r}^{k}$, $m(\psi_{r\infty}^{0})=\psi_{r\infty}^{k}$,
and $m(\psi_{\infty0}^{0})=\psi_{\infty0}^{k}$, the choice of $u_{i,\ell},v_{j},w_{i,\ell}$
from Subsection~\ref{subsec:triangle-k=00003D0} makes the triangle
counting maps vanish for general $k=0,\cdots,p-1$ as well.

\subsection{\label{subsec:Picture-proofs}Picture proofs}

\subsubsection{\label{subsec:The-involution}The involution}

The involution that we considered in Subsubsection~\ref{subsec:triangle0rinf-rinf0}
has a geometric interpretation which we find satisfying. We will modify
our Heegaard diagram a bit, and then exhibit an involution on $\widetilde{\mathbb{T}^{2}}$.
This in particular shows $z_{n}=z_{p+q-1-n}$ (Lemma~\ref{lem:a0}).

Recall the intersection points $\widetilde{\theta}_{t}$ from Definition~\ref{def:theta-tilde}.
Let $M$ be a midpoint of $\widetilde{\theta}_{0}$ and $\widetilde{\theta}_{p+q-1}$
(this is well defined up to shifting by $(0,p/2)$). Also, recall
from Subsection~\ref{subsec:cover} that we denote the $y$-axis
as $\widetilde{\beta}_{\infty}^{0}\subset\pi^{-1}(\beta_{\infty})\subset\widetilde{\mathbb{T}^{2}}$.

\begin{figure}[h]
\begin{centering}
\raisebox{-0.5\height}{\includegraphics[scale=0.6]{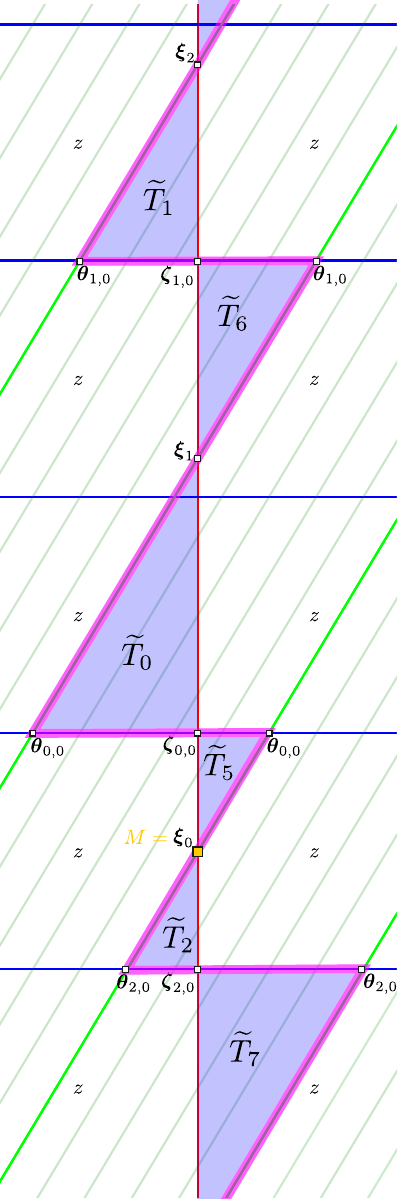}}\quad{}\raisebox{-0.5\height}{\includegraphics[scale=0.6]{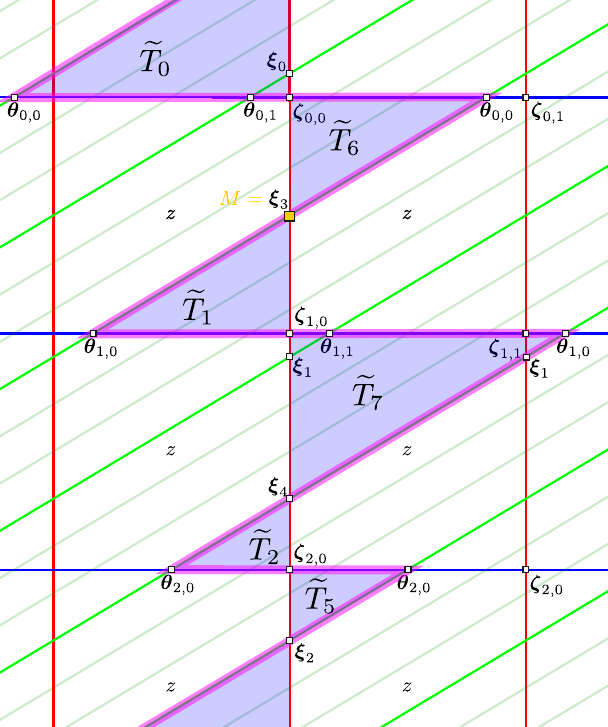}}\quad{}
\par\end{centering}
\caption{\label{fig:triangles}Diagrams for a geometric interpretation of the
involution, for $(p,q)=(5,3)$ and $(3,5)$. Rotation by $\pi$ about
$M$ induces the involution.}
\end{figure}

At first, $\widetilde{\beta}_{\infty}^{0}$ intersects the midpoints
of $\widetilde{\theta}_{p-1}$ and $\widetilde{\theta}_{p}$. Move
$\pi^{-1}(\beta_{\infty})$ horizontally until $\widetilde{\beta}_{\infty}^{0}$
goes through the point $M$ (see Figure~\ref{fig:triangles}). Note
that if $p+q$ is even, then this is equivalent to $\widetilde{\beta}_{\infty}^{0}$
going through the midpoints of $\widetilde{\theta}_{(p+q)/2-1}$ and
$\widetilde{\theta}_{(p+q)/2}$; if $p+q$ is odd, then this is equivalent
to $\widetilde{\beta}_{\infty}^{0}$ going through $\widetilde{\theta}_{(p+q-1)/2}$.
Also note that if $p+q$ is odd, then there is a degenerate triangle:
the three vertices of $\widetilde{T}_{(p+q-1)/2}$ become identical.

While moving $\pi^{-1}(\beta_{\infty})$, move the basepoints $\pi^{-1}(z)$
as in Claim~\ref{claim:While-moving-,}.
\begin{claim}
\label{claim:While-moving-,}While moving $\pi^{-1}(\beta_{\infty})$,
it is possible to move the basepoints $\pi^{-1}(z)$ such that the
basepoints do not cross $\pi^{-1}(\beta_{0}\cup\beta_{r}\cup\beta_{\infty})$
if and only if $\widetilde{\beta}_{\infty}^{0}$ stays strictly between
$\widetilde{\theta}_{-1}$ and $\widetilde{\theta}_{p+q}$.
\end{claim}

\begin{proof}
First, that it is possible is equivalent to the following: for each
connected component $\widetilde{\beta}_{\infty}^{i}$ of $\pi^{-1}(\beta_{\infty})$
and each parallelogram $P$ bounded by $\pi^{-1}(\beta_{0}\cup\beta_{r})$
and a point $\widetilde{z}\in\pi^{-1}(z)\cap P$, (1) if the point
$\widetilde{z}$ is on the right of $\widetilde{\beta}_{\infty}^{i}$,
then $\widetilde{\beta}_{\infty}^{i}$ does not pass through the rightmost
point of the parallelogram $P$, and (2) if the point $\widetilde{z}$
is on the left of $\widetilde{\beta}_{\infty}^{i}$, then $\widetilde{\beta}_{\infty}^{i}$
does not pass through the leftmost point of the parallelogram $P$.
Since there is a symmetry given by translating by $\mathbb{Z}^{2}$,
it is sufficient to check the above for one parallelogram whose vertices
have the same $x$-coordinates as $\widetilde{\theta}_{p-1},\widetilde{\theta}_{p},\widetilde{\theta}_{p+q-1},\widetilde{\theta}_{p+q}$.
Hence, the condition is equivalent to that $\widetilde{\beta}_{\infty}^{0}$
stays on the left side of $\widetilde{\theta}_{p+q}$ and $\widetilde{\beta}_{\infty}^{1}=\widetilde{\beta}_{\infty}^{0}+(1,0)$
stays on the right side of $\widetilde{\theta}_{p-1}$. Thus the claim
follows.
\end{proof}
\begin{figure}[h]
\begin{centering}
\includegraphics[scale=0.7]{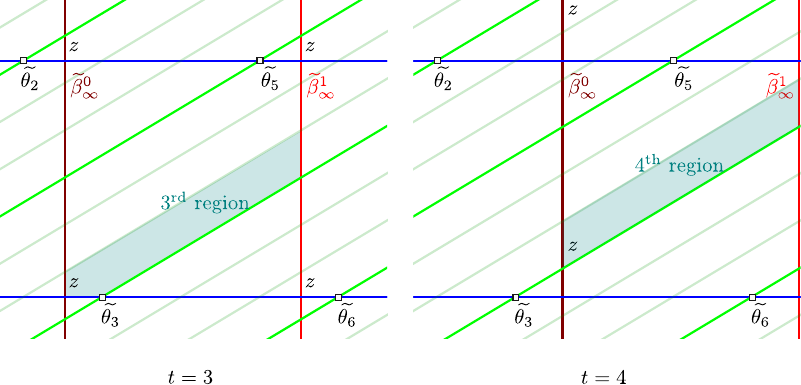}
\par\end{centering}
\caption{\label{fig:parallelograms-1}Left: $(p,q)=(3,5)$, when $\pi^{-1}(\beta_{\infty})$
is between $\widetilde{\theta}_{2}$ and $\widetilde{\theta}_{3}$.
Right: $(p,q)=(3,5)$, when $\pi^{-1}(\beta_{\infty})$ is between
$\widetilde{\theta}_{3}$ and $\widetilde{\theta}_{4}$}
\end{figure}

Now, we claim that after the above operation, we can move the basepoints
to the center of the squares bounded by $\pi^{-1}(\beta_{0})$ and
$\pi^{-1}(\beta_{\infty})$ without crossing $\pi^{-1}(\beta_{0}\cup\beta_{r}\cup\beta_{\infty})$
(see Figure~\ref{fig:triangles}). In the process of moving $\pi^{-1}(\beta_{\infty})$,
say that $\widetilde{\beta}_{\infty}^{0}$ is currently between $\widetilde{\theta}_{t-1}$
and $\widetilde{\theta}_{t}$. Then, $\pi^{-1}(\beta_{r})$ divides
each smallest square bounded by $\pi^{-1}(\beta_{0}\cup\beta_{\infty})$
into $p+q+1$ regions. Order them from bottom right ($0$th region)
to top left ($(p+q)$th region): see Figure~\ref{fig:parallelograms-1}.
Then, the basepoints are in the $t$th region. Hence, if $p+q$ is
even, then the basepoints are in the middle region when $\widetilde{\beta}_{\infty}^{0}$
goes through the midpoints of $\widetilde{\theta}_{(p+q)/2-1}$ and
$\widetilde{\theta}_{(p+q)/2}$.

If $p+q$ is odd, then if (A) $\widetilde{\beta}_{\infty}^{0}$ is
between $\widetilde{\theta}_{(p+q-1)/2-1}$ and $\widetilde{\theta}_{(p+q-1)/2}$,
then the basepoints are in the $(p+q-1)/2$th region. If (B) $\widetilde{\beta}_{\infty}^{0}$
goes through $\widetilde{\theta}_{(p+q-1)/2}$, then $\pi^{-1}(\beta_{r})$
divides each smallest square bounded by $\pi^{-1}(\beta_{0}\cup\beta_{\infty})$
into $p+q$ regions. Order them from bottom right ($0$th region)
to top left ($(p+q-1)$th region) as before. In the process of moving
$\widetilde{\beta}_{\infty}^{0}$ to the right, from (A) to (B), the
$(p+q)$th region disappears. The basepoints are still in the $(p+q-1)/2$th
region, which is the middle region. 

Now, the rotation $R:\widetilde{\mathbb{T}^{2}}\to\widetilde{\mathbb{T}^{2}}$
by $\pi$ about the point $M$ is our wanted involution. First, $R$
fixes $\widetilde{\beta}_{r}^{0}$ since it is rotation by $\pi$
and swaps $\widetilde{\theta}_{0}$ and $\widetilde{\theta}_{p+q-1}$,
which are both on $\widetilde{\beta}_{r}^{0}$. It fixes $\pi^{-1}(\beta_{0})$
and $\pi^{-1}(\beta_{\infty})$ since $M\in\pi^{-1}(\beta_{0}\cap\beta_{\infty})+\frac{1}{2}\mathbb{Z}^{2}$.
Since it swaps $\widetilde{\theta}_{0}$ and $\widetilde{\theta}_{p+q-1}$,
it should swap $\widetilde{\theta}_{n}$ and $\widetilde{\theta}_{p+q-1-n}$,
and hence swaps $\widetilde{T}_{n}$ and $\widetilde{T}_{p+q-1-n}$.
This shows $z_{n}=z_{p+q-n-1}$. That $s_{i}=s_{q-i-1}$ also follows
since we could have defined special basis elements after moving $\pi^{-1}(\beta_{\infty})$,
and the set of intersection points that are lifts of special basis
elements is invariant under $R$ since the definition is invariant
under $R$.

\subsubsection{\label{subsec:zk}Equation~(\ref{eq:triangle-basepoint})}

Equation~(\ref{eq:triangle-basepoint}) also has a satisfying picture
proof. Let us use the same notations as in Subsubsection~\ref{subsec:The-involution}
and proceed similarly: move $\pi^{-1}(\beta_{\infty})$ to the left
so that $\widetilde{\beta}_{\infty}^{0}$ intersects the $x$-axis
between $\widetilde{\theta}_{p-k-1}$ and $\widetilde{\theta}_{p-k}$,
and call it $\pi^{-1}(\beta_{\infty})_{k}$ (see Figure~\ref{fig:diagramt2bspts}~(b)).
Also, move the basepoints $\pi^{-1}(z)$  so that they do not cross
$\pi^{-1}(\beta_{0}\cup\beta_{r}\cup\beta_{\infty})$, which is possible
by Claim~\ref{claim:While-moving-,}. Let $\widetilde{T}_{n,k}$
be the triangle $\widetilde{T}_{n}$ in this new diagram. Then, we
have 
\begin{equation}
n_{\pi^{-1}(z)}(\widetilde{T}_{n,k})=n_{\pi^{-1}(z)}(\widetilde{T}_{n}).\label{eq:tnktn}
\end{equation}

\begin{figure}[h]
\begin{centering}
\includegraphics[scale=0.5]{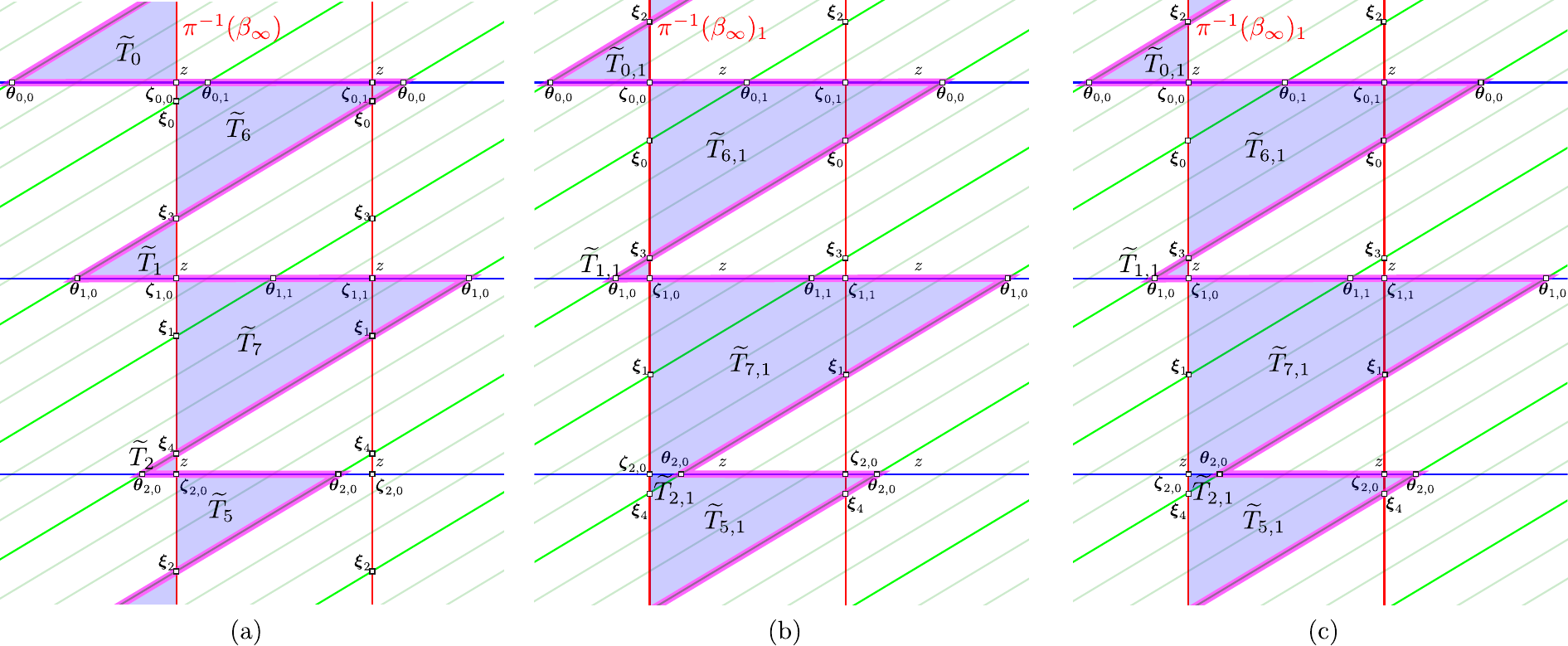}
\par\end{centering}
\caption{\label{fig:diagramt2bspts}Diagrams for the proof of Proposition \ref{prop:general-k-same},
for $(p,q)=(3,5)$. (a): original diagram ($k=0$) (b): $k=1$, after
moving $\pi^{-1}(\beta_{\infty})$ (c): $k=1$, after moving $\pi^{-1}(\beta_{\infty})$
and the basepoints (compare the middle diagram of Figure~\ref{fig:tildet2-1-1})}
\end{figure}

By Subsubsection~\ref{subsec:The-involution}, the basepoints are
in the $(p-k)$th region of each smallest square bounded by $\pi^{-1}(\beta_{0}\cup\beta_{\infty})$.
Since $p-k\le p$, it is possible to move the basepoints to either
the $0$th or $p$th region without crossing $\pi^{-1}(\beta_{0}\cup\beta_{\infty})\cup\widetilde{\beta}_{r}^{0}$,
hence to either directly on the right of $\pi^{-1}(\beta_{\infty})_{k}$
or directly on the left of $\pi^{-1}(\beta_{\infty})_{k}$ (see Figure~\ref{fig:diagramt2bspts}~(c)).
Denote the set of these new basepoints as $\pi^{-1}(z)_{k}$. The
main observation is the following claim:
\begin{claim}
\label{claim:There-exists-a}There exists a translation of $\widetilde{\mathbb{T}^{2}}$
that maps
\begin{equation}
(\widetilde{\mathbb{T}^{2}},\pi^{-1}(\beta_{0}),\widetilde{\beta}_{r}^{0},\pi^{-1}(\beta_{\infty})_{k},\pi^{-1}(z)_{k})\to(\widetilde{\mathbb{T}^{2}},\pi^{-1}(\beta_{0}),\widetilde{\beta}_{r}^{0},\pi^{-1}(\beta_{\infty})_{0},\widetilde{z(k)}).\label{eq:translation-r2}
\end{equation}
Moreover, this maps $\widetilde{T}_{n,k}$ to $\widetilde{T}_{n+k}$.
\end{claim}

\begin{proof}
Consider the translation $\Phi_{k}:\widetilde{\mathbb{T}^{2}}\to\widetilde{\mathbb{T}^{2}}$
that maps $\widetilde{\theta}_{0}$ to $\widetilde{\theta}_{k}$.
Then, $\Phi_{k}(\widetilde{\theta}_{i})=\widetilde{\theta}_{i+k}$,
and it is immediate that $\pi^{-1}(\beta_{0}),\widetilde{\beta}_{r}^{0},\pi^{-1}(\beta_{\infty})_{k}$
map to $\pi^{-1}(\beta_{0}),\widetilde{\beta}_{r}^{0},\pi^{-1}(\beta_{\infty})_{0}$,
respectively. Hence, $\Phi_{k}$ maps $\widetilde{T}_{n,k}$ to $\widetilde{T}_{n+k}$.
We are left to check that $\pi^{-1}(z)_{k}$ maps to $\widetilde{z(k)}$.
Since both $\pi^{-1}(z)_{k}$ and $\widetilde{z(k)}$ are invariant
under translation by $(1,0)\mathbb{Z}$, their positions (i.e.\ whether
for each row, they are on the left or right of $\pi^{-1}(\beta_{\infty})_{k}$
and $\pi^{-1}(\beta_{\infty})_{0}$) are determined by their positions
on the zeroth column, and hence, by Corollary~\ref{cor:u-unique},
they are determined by the condition that $\widetilde{T}_{0,k},\widetilde{T}_{1,k},\cdots,\widetilde{T}_{p+q-1,k}$
do not intersect $\pi^{-1}(z)_{k}$ and that $\widetilde{T}_{k},\widetilde{T}_{k+1},\cdots,\widetilde{T}_{k+p+q-1}$
do not intersect $\widetilde{z(k)}$. Since $\Phi_{k}$ maps $\widetilde{T}_{n,k}$
to $\widetilde{T}_{n+k}$, we have that it maps $\pi^{-1}(z)_{k}$
to $\widetilde{z(k)}$.
\end{proof}
Now, we have
\[
n_{\widetilde{z(k)}}(\widetilde{T}_{n+k})=n_{\pi^{-1}(z)}(\widetilde{T}_{n,k})=n_{\pi^{-1}(z)}(\widetilde{T}_{n}).
\]
where the first equality follows from that $\widetilde{T}_{n,k}$
maps to $\widetilde{T}_{n+k}$ under the translation (\ref{eq:translation-r2}),
and the second equality is Equation~(\ref{eq:tnktn}).

\section{\label{sec:Examples}Examples}

In this section, we demonstrate Theorem~\ref{thm:rational-surgery-kgeneral}
for $(p,q)=(5,3)$ and $k=0,1,2$, for both the unknot $O$ and the
right handed trefoil $T$ with the Seifert framing. These cases can
be computed using purely formal arguments involving the absolute $\mathbb{Q}$-grading,
as in \cite[Section~8]{MR1957829}.

\subsection{\label{subsec:530}Case $k=0$}

First, let us restate Theorem~\ref{thm:rational-surgery-k=00003D0}
for the special case where $p\ge q$, $Y=S^{3}$, and the framing
$\lambda$ is the Seifert framing.
\begin{cor}[Main theorem for $p\ge q$, $k=0$, and a $0$-framed knot in $S^{3}$]
\label{cor:seifert}Let $K$ be a knot in $S^{3}$. Then, for all
coprime positive integers $p\ge q$, there is an $\mathbb{F}\llbracket U\rrbracket$-linear
exact triangle 
\[\begin{tikzcd}[ampersand replacement=\&]
	{\bigoplus_{j=0}^{q-1} \boldsymbol{HF}^{-}(S_{0}^3 (K))} \&\& {\boldsymbol{HF}^-(S_{p/q} ^3 (K))} \\
	\& {\bigoplus_{i=0}^{p-1} \mathbb{F} \llbracket U\rrbracket }
	\arrow[from=1-1, to=1-3]
	\arrow[from=1-3, to=2-2]
	\arrow[from=2-2, to=1-1]
\end{tikzcd}\]
\end{cor}

\begin{figure}[h]
\begin{centering}
\raisebox{-0.5\height}{\includegraphics[scale=0.87]{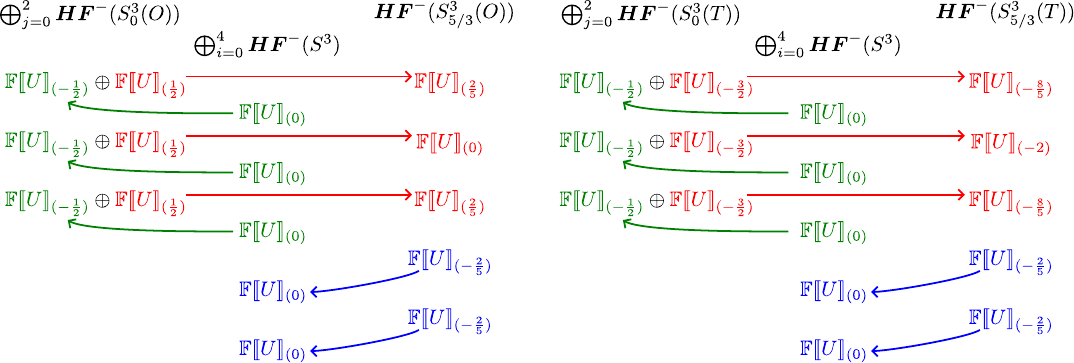}}
\par\end{centering}
\caption{\label{fig:exact530}The rational surgery exact triangle for the unknot
$O\subset S^{3}$ and the right handed trefoil $T\subset S^{3}$,
for $(p,q,k)=(5,3,0)$.}
\end{figure}

The case $(p,q,k)=(5,3,0)$ is shown in Figure~\ref{fig:exact530}.
Here, $\mathbb{F}\llbracket U\rrbracket_{(d)}$ means that the top
grading element $1\in\mathbb{F}\llbracket U\rrbracket$ is supported
in grading $d$. For each color (red, blue, green), the arrows in
that color are isomorphisms between the groups in that color. Note
that these maps are in general not homogeneous with respect to the
homological grading but only filtered: they are multiplication by
a unit in $\mathbb{F}\llbracket U\rrbracket$. Note that for each
arrow on the left side of Figure~\ref{fig:exact530}, there is a
corresponding arrow in the right side of Figure~\ref{fig:exact530};
their filtration degrees are the same.

Observe that the exact triangles in Figure~\ref{fig:exact530} split
into five exact triangles; in fact, similarly to \cite[Section~9]{MR2113020},
we can show purely combinatorially that our rational surgery exact
triangles split into $p$ many exact triangles, given by the ${\rm Spin}^{c}$-splitting.
\begin{cor}[${\rm Spin}^{c}$-splitting of Corollary~\ref{cor:seifert}]
Let $K$ be a knot in $S^{3}$, and let $p\ge q$ be coprime positive
integers. Then, there exist identifications ${\rm Spin}^{c}(S_{0}^{3}(K))\cong\mathbb{Z}$
and ${\rm Spin}^{c}(S_{p/q}^{3}(K))\cong\mathbb{Z}/p\mathbb{Z}$ such
that for each $i=0,\cdots,p-1$, there is an $\mathbb{F}\llbracket U\rrbracket$-linear
exact triangle 
\[\begin{tikzcd}[ampersand replacement=\&]
	{\bigoplus_{-qt \in \{i, i+1 , \cdots ,i+q-1\}\bmod p }\boldsymbol{HF}^{-}(S_{0}^3 (K);t)} \&\& {\boldsymbol{HF}^-(S_{p/q}^3 (K);i)} \\
	\& { \mathbb{F} \llbracket U\rrbracket }
	\arrow[from=1-1, to=1-3]
	\arrow[from=1-3, to=2-2]
	\arrow[from=2-2, to=1-1]
\end{tikzcd}\]
\end{cor}

\subsection{\label{subsec:5312}Cases $(p,q,k)=(5,3,1)$ and $(5,3,2)$}

Conveniently, in these cases, $\boldsymbol{HFK}_{p,q,k}^{-}(S^{3},K)$
is isomorphic to the $2k/p$-modified knot Floer homology $\frac{2k}{p}\boldsymbol{HFK}^{-}(S^{3},K)$
(Definition~\ref{def:t-knotfloer}). Indeed, if we view $\mathbb{F}\llbracket U^{1/5}\rrbracket$
as an $\mathbb{F}\llbracket W,Z\rrbracket$-module by letting $W$
act as $U^{k/p}$ and $Z$ act as $U^{1-(k/p)}$, and view $E_{\infty}^{k}=\bigoplus_{i=0}^{4}y_{i}\mathbb{F}\llbracket U\rrbracket$
as an $\mathbb{F}\llbracket W,Z\rrbracket$-module by letting $W$
act as $\phi_{\infty}^{k}$ and $Z$ act as $U(\phi_{\infty}^{k})^{-1}$,
then the $\mathbb{F}\llbracket U\rrbracket$-module isomorphism 
\[
E_{\infty}^{k}\to\mathbb{F}\llbracket U^{1/5}\rrbracket:y_{i}\mapsto U^{ik/p}
\]
is an $\mathbb{F}\llbracket W,Z\rrbracket$-module isomorphism, and
this isomorphism induces isomorphisms
\[
\boldsymbol{HFK}_{5,3,1}^{-}(S^{3},K)\cong\frac{2}{5}\boldsymbol{HFK}^{-}(S^{3},K),\ \boldsymbol{HFK}_{5,3,2}^{-}(S^{3},K)\cong\frac{4}{5}\boldsymbol{HFK}^{-}(S^{3},K).
\]

Let us describe the absolute $\mathbb{Q}$-grading on these homology
groups. The knot Floer chain complex ${\cal CFK}^{-}(S^{3},K)$ is
equipped with two homological (Maslov) gradings, ${\rm gr}_{w}$ and
${\rm gr}_{z}$. Define the absolute $\mathbb{Q}$-grading on $t\boldsymbol{CFK}^{-}(S^{3},K)$
as $(t/2){\rm gr}_{w}+(1-t/2){\rm gr}_{z}$. Then, $U$ has grading
$-2$ since $U=WZ$, and the differential is homogeneous of grading
$-1$.

For the cases $(p,q,k)=(5,3,1),(5,3,2)$ and $K=O$, we have $\boldsymbol{HFK}_{5,3,k}^{-}(S^{3},O)\cong\mathbb{F}\llbracket U^{1/5}\rrbracket$,
and the element $1\in\mathbb{F}\llbracket U^{1/5}\rrbracket$ has
absolute $\mathbb{Q}$-grading $0$. We view this as $5$ copies of
$\mathbb{F}\llbracket U\rrbracket$, where the top grading elements
are supported in absolute $\mathbb{Q}$-gradings $0,-2/5,-4/5,-6/5,-8/5$,
respectively.

Let us consider the right handed trefoil $T$. The chain complex ${\cal CFK}^{-}(S^{3},T)$
is freely generated by three basis elements, say $a,b,c$ over $\mathbb{F}\llbracket W,Z\rrbracket$,
and the differential is given by $\partial a=Wb+Zc$. Hence, for the
case $(p,q,k)=(5,3,1)$, $\frac{2}{5}\boldsymbol{CFK}^{-}(S^{3},T)$
is freely generated over $\mathbb{F}\llbracket U^{1/5}\rrbracket$
by $a,b,c$, whose absolute $\mathbb{Q}$-gradings are $-1,-8/5,-2/5$,
respectively, and the differential is given by $\partial a=U^{1/5}b+U^{4/5}c$.
Hence, its homology is $\mathbb{F}\llbracket U^{1/5}\rrbracket\oplus\mathbb{F}\llbracket U^{1/5}\rrbracket/U^{1/5}$,
where the $\mathbb{F}\llbracket U^{1/5}\rrbracket$ summand is generated
by $c$, and the $\mathbb{F}\llbracket U^{1/5}\rrbracket/U^{1/5}$
summand is generated by $b+U^{3/5}c$, and hence
\[
\frac{2}{5}\boldsymbol{HFK}^{-}(S^{3},T)\cong(\mathbb{F}\llbracket U^{1/5}\rrbracket)_{(-2/5)}\oplus(\mathbb{F}\llbracket U^{1/5}\rrbracket/U^{1/5})_{(-8/5)}.
\]
(Here, the subscript $(d)$ means that the top grading element is
supported in absolute $\mathbb{Q}$-grading $d$.) We view this as
the direct sum of $\mathbb{F}\llbracket U\rrbracket_{(-2/5)}$, $\mathbb{F}\llbracket U\rrbracket_{(-4/5)}$,
$\mathbb{F}\llbracket U\rrbracket_{(-6/5)}$, $\mathbb{F}\llbracket U\rrbracket_{(-8/5)}\oplus(\mathbb{F}\llbracket U\rrbracket/U)_{(-8/5)}$,
and $\mathbb{F}\llbracket U\rrbracket_{(-2)}$. In fact, the chain
complex $\frac{2}{5}\boldsymbol{CFK}^{-}(S^{3},T)$ itself splits
into five chain complexes over $\mathbb{F}\llbracket U\rrbracket$,
whose homologies are the above, respectively.

For the case $(p,q,k)=(5,3,2)$, $\frac{4}{5}\boldsymbol{CFK}^{-}(S^{3},T)$
is freely generated over $\mathbb{F}\llbracket U^{1/5}\rrbracket$
by $a,b,c$, whose absolute $\mathbb{Q}$-gradings are $-1,-6/5,-4/5$,
respectively, and the differential is given by $\partial a=U^{2/5}b+U^{3/5}c$.
Its homology is 
\[
(\mathbb{F}\llbracket U^{1/5}\rrbracket)_{(-4/5)}\oplus(\mathbb{F}\llbracket U^{1/5}\rrbracket/U^{2/5})_{(-6/5)},
\]
where the two summands are generated by $c$ and $b+U^{1/5}c$, respectively.
We view this as the direct sum of $\mathbb{F}\llbracket U\rrbracket_{(-4/5)}$,
$\mathbb{F}\llbracket U\rrbracket_{(-6/5)}\oplus(\mathbb{F}\llbracket U\rrbracket/U)_{(-6/5)}$,
$\mathbb{F}\llbracket U\rrbracket_{(-8/5)}\oplus(\mathbb{F}\llbracket U\rrbracket/U)_{(-8/5)}$,
$\mathbb{F}\llbracket U\rrbracket_{(-2)}$, and $\mathbb{F}\llbracket U\rrbracket_{(-12/5)}$.
In fact, the chain complex $\frac{4}{5}\boldsymbol{CFK}^{-}(S^{3},T)$
itself splits into five chain complexes over $\mathbb{F}\llbracket U\rrbracket$,
whose homologies are the above, respectively.

The rational surgery exact triangle (Theorem~\ref{thm:rational-surgery-kgeneral})
for the cases $(p,q,k)=(5,3,1)$ and $(5,3,2)$ are shown in Figures~\ref{fig:exact531}~and~\ref{fig:exact532},
respectively. These are color coded similarly to Figure~\ref{fig:exact530}.
Observe that these exact triangles also split into $5$ exact triangles.
Similarly to \cite[Section~9]{MR2113020}, we can show purely combinatorially
that for the case where $Y$ is $S^{3}$ and $\lambda$ is the Seifert
framing, our rational surgery exact triangles split into $p$ many
exact triangles.

\begin{figure}[h]
\begin{centering}
\raisebox{-0.5\height}{\includegraphics[scale=0.87]{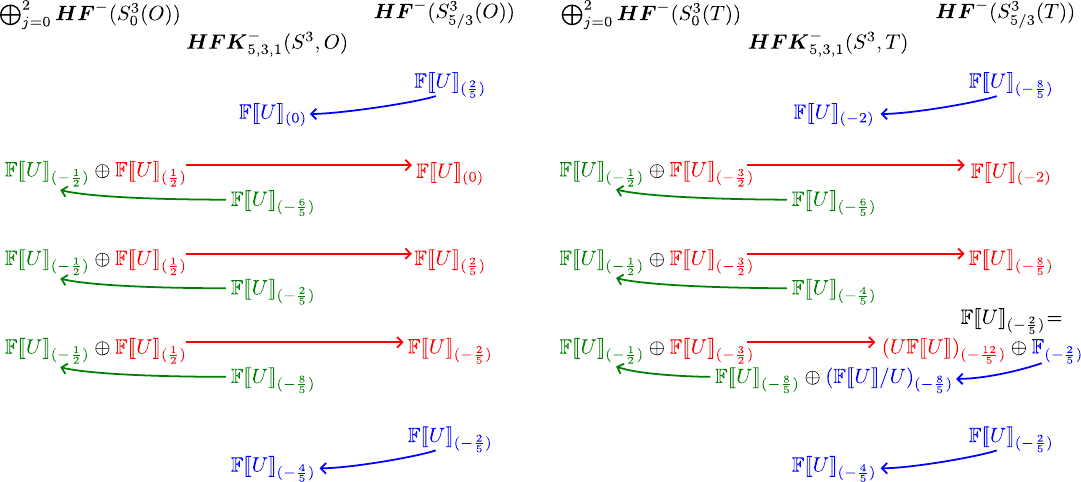}}
\par\end{centering}
\caption{\label{fig:exact531}The rational surgery exact triangle for the unknot
$O\subset S^{3}$ and the right handed trefoil $T\subset S^{3}$,
for $(p,q,k)=(5,3,1)$. }
\end{figure}

\begin{figure}[h]
\begin{centering}
\raisebox{-0.5\height}{\includegraphics[scale=0.87]{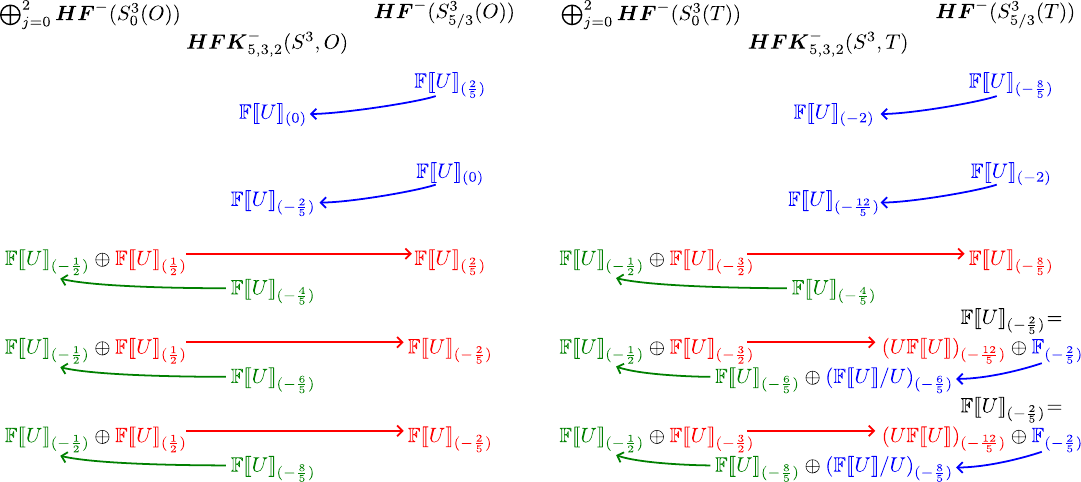}}
\par\end{centering}
\caption{\label{fig:exact532}The rational surgery exact triangle for the unknot
$O\subset S^{3}$ and the right handed trefoil $T\subset S^{3}$,
for $(p,q,k)=(5,3,2)$.}
\end{figure}

\appendix

\section{\label{sec:Checks-for-Subsubsection}Computations for Subsection~\ref{subsec:triangle-k=00003D0}}

Let $p$ and $q$ be coprime positive integers. As in Subsection~\ref{subsec:triangle-k=00003D0},
let $\ell,i,j\in\mathbb{Z}$ always denote the integers such that
 $i:=n\bmod p$, $j:=n\bmod q$, and $n=\ell p+i$. Recall from Definition~\ref{def:theta-tilde}
that $z_{n}:=n_{z}(T_{n})$ is the number of basepoints in the $n$th
triangle $T_{n}$.
\begin{lem}
\label{lem:a0}We have $z_{n}=z_{p+q-1-n}$.
\end{lem}

\begin{proof}
First, the lemma holds for $n=0,\cdots,p+q-1$ since $z_{n}=0$ for
such $n$. For $n\notin\{0,\cdots,p+q-1\}$, let us express $z_{n}$
as a sum. Let $\overline{T}_{n}$ be a lift of $\widetilde{T}_{n}$
to the universal cover $\mathbb{R}^{2}$, and let $\overline{z}\subset\mathbb{R}^{2}$
be the inverse image of $z\in\mathbb{T}^{2}$. Let $b_{n}\in\mathbb{Z}$
be such that the side of $\overline{T}_{n}$ parallel to the $x$-axis
is contained in $\mathbb{R}\times\{b_{n}\}$. We count the number
of basepoints $|\overline{T}_{n}\cap\overline{z}|$ in each row $\mathbb{R}\times\{b_{n}+t+\varepsilon\}$
separately, for each $t\in\mathbb{Z}$.

For $n\ge0$, we have 
\begin{gather}
|\overline{T}_{-n+p-1}\cap\overline{z}\cap(\mathbb{R}\times\{b_{-n+p-1}+t+\varepsilon\})|=\begin{cases}
\max\left(\left\lfloor \frac{n-tq}{p}\right\rfloor ,0\right) & {\rm if}\ t\ge0\\
0 & {\rm if}\ t<0
\end{cases},\label{eq:-np-1}\\
|\overline{T}_{n+p+q-1}\cap\overline{z}\cap(\mathbb{R}\times\{b_{n+p+q-1}-1-t+\varepsilon\})|=\begin{cases}
\max\left(\left\lceil \frac{n-tq}{p}\right\rceil ,0\right) & {\rm if}\ t\ge0\\
0 & {\rm if}\ t<0
\end{cases}.\label{eq:npq-1}
\end{gather}
Since $\left\lfloor x/p\right\rfloor =\left\lceil (x-(p-1))/p\right\rceil $,
we have $z_{-n+p-1}=z_{(n-p+1)+p+q-1}=z_{n+q}$ for $n\ge p-1$.
\end{proof}
\begin{lem}
\label{lem:nzk}Let $k=0,\cdots,p-1$, and let $\widetilde{z(k)}$
be the set $\widetilde{z(u)}\subset\widetilde{\mathbb{T}^{2}}$ from
Subsection~\ref{subsec:cover} that corresponds to the monodromy
of $E_{\infty}^{k}$ (Definition~\ref{def:actual-local-system}).
Then
\begin{equation}
n_{\widetilde{z(k)}}(\widetilde{T}_{n+k})=z_{n}.\label{eq:a4}
\end{equation}
\end{lem}

\begin{proof}
This algebraic proof is routine, albeit quite complicated. Note that
we give a ``picture proof'' that we find more satisfying in Subsubsection~\ref{subsec:zk}.

Let us work in $\mathbb{R}^{2}$ similarly to the proof of Lemma~\ref{lem:a0}.
Let $\overline{z(k)}$ be the inverse image of $\widetilde{z(k)}\subset\widetilde{\mathbb{T}^{2}}$
in $\mathbb{R}^{2}$. Then, we have
\begin{align*}
\overline{z(k)} & =\{(a+\varepsilon,b+\varepsilon):a,b\in\mathbb{Z},\ qb\in\{k,k+1,\cdots,p-1\}\mod p\}\\
 & \cup\{(a-\varepsilon,b+\varepsilon):a,b\in\mathbb{Z},\ qb\in\{0,1,\cdots,k-1\}\mod p\}.
\end{align*}
Also, let $b_{n}$ be such that the $\partial_{\pi^{-1}(\beta_{0})}\overline{T}_{n}$
is contained in $\mathbb{R}\times\{b_{n}\}$. Note that $qb_{n}\equiv n$
mod~$p$.

For $n\ge0$, we have 
\begin{multline*}
|\overline{T}_{-n+p-1}\cap\overline{z(k)}\cap(\mathbb{R}\times\{b_{-n+p-1}+t+\varepsilon\})|\\
=\begin{cases}
\max\left(\left\lfloor \frac{n-tq}{p}\right\rfloor ,0\right) & {\rm if}\ t\ge0\ {\rm and}\ -n-1+tq\in\{k,k+1,\cdots,p-1\}\mod p\\
\max\left(1+\left\lfloor \frac{n-tq}{p}\right\rfloor ,0\right) & {\rm if}\ t\ge0\ {\rm and}\ -n-1+tq\in\{0,1,\cdots,k-1\}\mod p\\
0 & {\rm if}\ t<0
\end{cases}.
\end{multline*}
Hence for $n\ge k$, we have $n_{\widetilde{z(k)}}(\widetilde{T}_{-n+p-1+k})=z_{-n+p-1}$
by Equation~(\ref{eq:-np-1}) since 
\[
\left\lfloor \frac{n-tq}{p}\right\rfloor =\begin{cases}
\left\lfloor \frac{n-k-tq}{p}\right\rfloor  & {\rm if}\ n-tq\in\{k,k+1,\cdots,p-1\}\mod p\\
1+\left\lfloor \frac{n-k-tq}{p}\right\rfloor  & {\rm if}\ n-tq\in\{0,1,\cdots,k-1\}\mod p
\end{cases}
\]
and
\[
n-tq\in\{k,k+1,\cdots,p-1\}\mod p\iff(k-1)-(n-tq)\in\{k,k+1,\cdots,p-1\}\mod p.
\]
This shows Equation~(\ref{eq:a4}) for $n\le p-k-1$.

Similarly, for $n\ge0$, we also have 
\begin{multline*}
|\overline{T}_{n+p+q-1}\cap\overline{z(k)}\cap(\mathbb{R}\times\{b_{n+p+q-1}-1-t+\varepsilon\})|\\
=\begin{cases}
\max\left(\left\lceil \frac{n-tq}{p}\right\rceil ,0\right) & {\rm if}\ t\ge0\ {\rm and}\ n-1-tq\in\{k,k+1,\cdots,p-1\}\mod p\\
\max\left(-1+\left\lceil \frac{n-tq}{p}\right\rceil ,0\right) & {\rm if}\ t\ge0\ {\rm and}\ n-1-tq\in\{0,1,\cdots,k-1\}\mod p\\
0 & {\rm if}\ t<0
\end{cases}.
\end{multline*}
Note that Equation~(\ref{eq:npq-1}) in fact holds for all $n\ge-p-q+1$
since $z_{0}=\cdots=z_{p+q-1}=0$. Hence for $n\ge-k$, we have $n_{\widetilde{z(k)}}(\widetilde{T}_{n+k+p+q-1})=z_{n+p+q-1}$
by Equation~(\ref{eq:npq-1}) since 
\[
\left\lceil \frac{n-tq}{p}\right\rceil =\begin{cases}
\left\lceil \frac{n+k-tq}{p}\right\rceil  & {\rm if}\ n-tq\in\{1,\cdots,p-k\}\mod p\\
-1+\left\lceil \frac{n+k-tq}{p}\right\rceil  & {\rm if}\ n-tq\in\{p-k+1,\cdots,p\}\mod p
\end{cases}
\]
and
\[
n-tq\in\{1,\cdots,p-k\}\mod p\iff n+k-1-tq\in\{k,k+1,\cdots,p-1\}\mod p.
\]
This shows Equation~(\ref{eq:a4}) for $n\ge p+q-k-1$.

We are left to show that $n_{\widetilde{z(k)}}(\widetilde{T}_{n})=0$
for $n=p,p+1,\cdots,p+q-2$. This follows since $T_{n}$ is disjoint
from $(\varepsilon,\varepsilon),(-\varepsilon,\varepsilon)\in\mathbb{T}^{2}$.
\end{proof}
Let $u,v\in\mathbb{Z}$ be such that $q=up+v$ and $v\in\{0,\cdots,p-1\}$;
then we have 
\begin{equation}
s_{0}=\cdots=s_{v-1}=u+1,\ s_{v}=\cdots=s_{p-1}=u.\label{eq:cv}
\end{equation}

\begin{lem}
\label{lem:a1}Let $n=\ell p+i$, $p+q-1-n=\overline{\ell}p+\overline{i}$
for $i,\overline{i},\ell,\overline{\ell}\in\mathbb{Z}$, $i,\overline{i}\in\{0,\cdots,p-1\}$.
Then, 
\[
\overline{i}\equiv q-i-1\mod p,\ s_{i}=s_{\overline{i}},\ {\rm and}\ \overline{\ell}\equiv-\ell\mod{s_{i}}.
\]
\end{lem}

\begin{proof}
That $\overline{i}\equiv q-i-1\bmod p$ is straightforward. For the
rest, notice that
\begin{equation}
p+q-1-n=\begin{cases}
(v-1-i)+p(u+1-\ell) & {\rm if}\ i\in\{0,\cdots,v-1\}\\
(p+v-1-i)+p(u-\ell) & {\rm if}\ i\in\{v,\cdots,p-1\}
\end{cases}.\label{eq:a1}
\end{equation}
Hence, 
\[
\overline{i}=\begin{cases}
v-1-i & {\rm if}\ i\in\{0,\cdots,v-1\}\\
p+v-1-i & {\rm if}\ i\in\{v,\cdots,p-1\}
\end{cases}
\]
and so $i\in\{0,\cdots,v-1\}$ if and only if $\overline{i}\in\{0,\cdots,v-1\}$.
Hence, $s_{i}=s_{\overline{i}}$ by Equation~(\ref{eq:cv}). Also,
we have $\overline{\ell}\equiv-\ell\mod{s_{i}}$ by Equations~(\ref{eq:cv})~and~(\ref{eq:a1}).
\end{proof}
\begin{lem}
\label{lem:a2}Assume that $q$ is odd, let $q=2q'+1$, and let $u',v'\in\mathbb{Z}$
be such that $q'=u'p+v'$ and $v'\in\{0,\cdots,p-1\}$. Then, $s_{q'}=s_{v'}=2u'+1$
is odd, and hence is in particular nonzero.
\end{lem}

\begin{proof}
To use Equation~(\ref{eq:cv}), let us compare $v'$ with $v$. We
divide into two cases: 
\begin{equation}
2q'+1=\begin{cases}
2u'p+(2v'+1) & {\rm if}\ 2v'+1<p\\
(2u'+1)p+(2v'+1-p) & {\rm if}\ 2v'+1\ge p
\end{cases}.\label{eq:a2}
\end{equation}
If $2v'+1<p$, then $v=2v'+1>v'$, and if $2v'+1\ge p$, then $v=2v'+1-p\le v'$.
Now, in both cases, we have $s_{v'}=2u'+1$ by Equations~(\ref{eq:cv})~and~(\ref{eq:a2}).
\end{proof}
\begin{lem}
\label{lem:a3}Assume that $p+q$ is odd, and let $\ell_{0},i_{0}\in\mathbb{Z}$
be such that $(p+q-1)/2=\ell_{0}p+i_{0}$ and $i_{0}\in\{0,\cdots,p-1\}$.
Then, $s_{(p+q-1)/2}=s_{i_{0}}=2\ell_{0}$ is even. In particular,
$p>q$ if and only if $s_{(p+q-1)/2}=0$.
\end{lem}

\begin{proof}
We divide into two cases:
\begin{equation}
q=\begin{cases}
(2\ell_{0}-1)p+(2i_{0}+1) & {\rm if}\ 2i_{0}+1<p\\
2\ell_{0}p+(2i_{0}+1-p) & {\rm if}\ 2i_{0}+1\ge p
\end{cases}.\label{eq:a3}
\end{equation}
If $2i_{0}+1<p$, then $v=2i_{0}+1>i_{0}$, and if $2i_{0}+1\ge p$,
then $v=2i_{0}+1-p\le i_{0}$. Now, in both cases, we have $s_{(p+q-1)/2}=s_{i_{0}}=2\ell_{0}$
by Equations~(\ref{eq:cv})~and~(\ref{eq:a3}).
\end{proof}
\begin{lem}
\label{lem:a4}Let $X=\{0,1,\cdots,\min(p,q)-1\}\subset\mathbb{Z}/q\mathbb{Z}$.
Then, the map 
\[
{\cal Z}:X\to X:j\mapsto s_{j}p+j,
\]
where $j\in\{0,\cdots,\min(p,q)-1\}$, is a well-defined bijection
with one orbit.
\end{lem}

\begin{proof}
Let us divide into two cases. If $p>q$, then $X=\mathbb{Z}/q\mathbb{Z}$
and $\mathcal{Z}(j)=j+p$. The lemma follows since $p$ and $q$ are
coprime.

If $p<q$, then $X=\{0,1,\cdots,p-1\}$ and since $up\equiv-v$ modulo
$q$, we have
\[
{\cal Z}(j)=\begin{cases}
j+(u+1)p=j-v+p & {\rm if}\ j<v\\
j+up=j-v & {\rm if}\ j\ge v
\end{cases},
\]
which is the residue of $j-v$ modulo $p$. Since $q$ and $p$ are
coprime, $v$ and $p$ are coprime as well. Hence, the lemma follows.
\end{proof}
\begin{rem}
In fact, Lemma~\ref{lem:a4} also follows from Lemma~\ref{lem:For-each-}:
first, $X$ is the set of $j=0,\cdots,q-1$ such that $s_{j}\neq0$.
Say $H_{j}$ and $H_{j'}$ (recall Definition~\ref{def:hi-li}) are
consecutive, i.e.\ there are no other $H_{t}$'s as one goes from
$H_{j}$ to $H_{j'}$ in the $-y$-direction. The two endpoints of
$H_{j}$ are $\widetilde{\theta}_{j}$ and $\widetilde{\theta}_{s_{j}p+j}$,
and so $L_{j'}$ connects $\widetilde{\theta}_{s_{j}p+j}$ and $\widetilde{\theta}_{j'}$.
Hence, $j'\equiv s_{j}p+j\bmod q$, and so $\mathcal{Z}(j)=j'$.
\end{rem}

\section{\label{sec:Reduction-of-Theorem}Reduction of Theorem~\ref{thm:rational-surgery-kgeneral}
to Theorem~\ref{thm:gen-local-comp}}

It is standard that Theorem~\ref{thm:rational-surgery-kgeneral}
follows from Theorem~\ref{thm:gen-local-comp}, by a neck-stretching
argument and the triangle detection lemma \cite[Lemma 4.2]{MR2141852}.
In this appendix, we recall this argument.

\subsection{The triangle detection lemma}

We use the following triangle detection lemma to construct our exact
triangle.
\begin{lem}[{Triangle detection lemma \cite[Lemma 4.2]{MR2141852}}]
\label{lem:triangle-det}Let $(C_{i},\partial_{i})$ for $i=-2,-1,0,1,2,3$
be chain complexes, let $f_{i}:C_{i}\to C_{i+1}$ for $i=-2,-1,0,1,2$
be chain maps, and let $H_{i}:C_{i}\to C_{i+2}$ for $i=-2,-1,0,1$
be maps that satisfy the following properties:
\begin{enumerate}
\item \label{enu:nullhomotopic}For $i=-2,-1,0,1$, the map $H_{i}:C_{i}\to C_{i+2}$
is a nullhomotopy of $f_{i+1}\circ f_{i}$.
\item For $i=-2,-1,0$, the chain map
\[
f_{i+2}\circ H_{i}+H_{i+1}\circ f_{i}:C_{i}\to C_{i+3}
\]
is a quasi-isomorphism.
\end{enumerate}
Then, the \emph{iterated mapping cone $(M,\partial_{M})$ }defined
as
\[
M:=C_{0}\oplus C_{1}\oplus C_{2},\ \partial_{M}:=\begin{pmatrix}\partial_{0} & 0 & 0\\
f_{0} & \partial_{1} & 0\\
H_{0} & f_{1} & \partial_{2}
\end{pmatrix}
\]
has trivial homology, i.e.\ $H(M)=0$.
\end{lem}

\begin{rem}
\label{rem:exact-triangle-derived}That $H(M)=0$ implies that there
exists some map $K:H(C_{2})\to H(C_{0})$ such that the following
is exact:
\[\begin{tikzcd}[ampersand replacement=\&]
	{H(C_0 )} \&\& {H(C_1 )} \\
	\& {H(C_2 )}
	\arrow["{H(f_0 )}", from=1-1, to=1-3]
	\arrow["{H(f_1 )}", from=1-3, to=2-2]
	\arrow["K", from=2-2, to=1-1]
\end{tikzcd}\]In particular, $H(C_{0})\xrightarrow{H(f_{0})}H(C_{1})\xrightarrow{H(f_{1})}H(C_{2})$
is exact.
\end{rem}

In fact, in our case, we can use Nakayama's lemma (Lemma~\ref{lem:derived-nakayama})
to weaken the assumptions of Lemma~\ref{lem:triangle-det} to Corollary~\ref{cor:triangle-detection-nakayama}:
in short, we only need Condition~(\ref{enu:nullhomotopic}) in full
for $i=0$, and over $\mathbb{F}=\mathbb{F}\llbracket U\rrbracket/U$
for $i=-2,-1,1$. Note that this does not simplify our solution to
the combinatorial problem (see Remark~\ref{rem:nakayama}).
\begin{lem}[Derived Nakayama]
\label{lem:derived-nakayama}Let $C$ be a finite, free chain complex
over $\mathbb{F}\llbracket U\rrbracket$. If $H(C\otimes_{\mathbb{F}\llbracket U\rrbracket}\mathbb{F})=0$,
then $H(C)=0$.

Let $f:C\to C'$ be an $\mathbb{F}\llbracket U\rrbracket$-linear
chain map between finite, free chain complexes $C,C'$ over $\mathbb{F}\llbracket U\rrbracket$.
If $f\otimes_{\mathbb{F}\llbracket U\rrbracket}\mathbb{F}$ is a quasi-isomorphism,
then $f$ is a quasi-isomorphism.
\end{lem}

\begin{proof}
For instance, see \cite[Lemma~2.37]{nahm2025unorientedskeinexacttriangle}.
\end{proof}
In the following corollary, if $f:C\to C'$ is an $\mathbb{F}\llbracket U\rrbracket$-linear
chain map, then we abuse notation and denote $f\otimes_{\mathbb{F}\llbracket U\rrbracket}\mathbb{F}:C\otimes_{\mathbb{F}\llbracket U\rrbracket}\mathbb{F}\to C'\otimes_{\mathbb{F}\llbracket U\rrbracket}\mathbb{F}$
as $f$ as well. 
\begin{cor}[Triangle detection lemma, stronger version]
\label{cor:triangle-detection-nakayama}Let $(C_{i},\partial_{i})$
for $i=-2,-1,0,1,2,3$ be finite, free chain complexes over $\mathbb{F}\llbracket U\rrbracket$,
let $f_{i}:C_{i}\to C_{i+1}$ for $i=-2,-1,0,1,2$ be chain maps,
and let $H_{0}:C_{0}\to C_{1}$ and $H_{i}:C_{i}\otimes_{\mathbb{F}\llbracket U\rrbracket}\mathbb{F}\to C_{i+2}\otimes_{\mathbb{F}\llbracket U\rrbracket}\mathbb{F}$
for $i=-2,-1,1$ be maps that satisfy the following properties:
\begin{enumerate}
\item \label{enu:The-composition-}The map $H_{0}:C_{0}\to C_{1}$ is a
nullhomotopy of $f_{1}\circ f_{0}$.
\item For $i=-2,-1,1$, the map $H_{i}:C_{i}\otimes_{\mathbb{F}\llbracket U\rrbracket}\mathbb{F}\to C_{i+2}\otimes_{\mathbb{F}\llbracket U\rrbracket}\mathbb{F}$
is a nullhomotopy of $f_{i+1}\circ f_{i}$.
\item For $i=-2,-1,0$, the chain map
\[
f_{i+2}\circ H_{i}+H_{i+1}\circ f_{i}:C_{i}\otimes_{\mathbb{F}\llbracket U\rrbracket}\mathbb{F}\to C_{i+3}\otimes_{\mathbb{F}\llbracket U\rrbracket}\mathbb{F}
\]
is a quasi-isomorphism.
\end{enumerate}
Then, the iterated mapping cone $(M,\partial_{M})$ defined as in
Lemma~\ref{lem:triangle-det} has trivial homology.
\end{cor}

\begin{proof}
Condition~(\ref{enu:The-composition-}) implies that $(M,\partial_{M})$
is a finite, free chain complex over $\mathbb{F}\llbracket U\rrbracket$.
Lemma~\ref{lem:triangle-det} applied to the chain complexes $C_{i}\otimes_{\mathbb{F}\llbracket U\rrbracket}\mathbb{F}$
shows that $M\otimes_{\mathbb{F}\llbracket U\rrbracket}\mathbb{F}$
has trivial homology, and so $M$ has trivial homology by Nakayama's
Lemma (Lemma~\ref{lem:derived-nakayama}).
\end{proof}
Corollary~\ref{cor:triangle-ainf} is the form of the triangle detection
lemma that we will use.
\begin{cor}[Triangle detection lemma, $A_{\infty}$ version]
\label{cor:triangle-ainf}Let $(\Sigma,\boldsymbol{\alpha},\boldsymbol{\beta}_{-2},\boldsymbol{\beta}_{-1},\cdots,\boldsymbol{\beta}_{3},z)$
be a weakly admissible Heegaard diagram, equip $\boldsymbol{\alpha}$
with a local system $(E,\phi,A)$, and equip $\boldsymbol{\beta}_{i}$
with a local system $(E_{i},\phi_{i},A_{i})$ for $i=-2,\cdots,3$.
Let $\psi_{i}\in\boldsymbol{CF}^{-}(\boldsymbol{\beta}_{i}^{E_{i}},\boldsymbol{\beta}_{i+1}^{E_{i+1}})$
for $i=-2,-1,0,1,2$ be cycles that satisfy the following properties:
\begin{enumerate}
\item \label{enu:mu20}For $i=-2,-1,0,1$, we have 
\[
\mu_{2}(\psi_{i}\otimes\psi_{i+1})=0\in\boldsymbol{CF}^{-}(\boldsymbol{\beta}_{0}^{E_{0}},\boldsymbol{\beta}_{2}^{E_{2}}).
\]
\item For $i=-2,-1,0$, the above Condition~(\ref{enu:mu20}) implies that
$\mu_{3}(\psi_{i}\otimes\psi_{i+1}\otimes\psi_{i+2})\in\boldsymbol{CF}^{-}(\boldsymbol{\beta}_{i}^{E_{i}},\boldsymbol{\beta}_{i+3}^{E_{i+3}})$
is a cycle. Let us use the same notation for its image in $\widehat{CF}(\boldsymbol{\beta}_{i}^{E_{i}},\boldsymbol{\beta}_{i+3}^{E_{i+3}})$.
Then, the chain map
\[
\mu_{2}(-\otimes\mu_{3}(\psi_{i}\otimes\psi_{i+1}\otimes\psi_{i+2})):\widehat{CF}(\boldsymbol{\alpha}^{E},\boldsymbol{\beta}_{i}^{E_{i}})\to\widehat{CF}(\boldsymbol{\alpha}^{E},\boldsymbol{\beta}_{i+3}^{E_{i+3}})
\]
is a quasi-isomorphism.
\end{enumerate}
Then, the iterated mapping cone $(M,\partial_{M})$ defined as 
\begin{gather*}
M:=\boldsymbol{CF}^{-}(\boldsymbol{\alpha}^{E},\boldsymbol{\beta}_{0}^{E_{0}})\oplus\boldsymbol{CF}^{-}(\boldsymbol{\alpha}^{E},\boldsymbol{\beta}_{1}^{E_{1}})\oplus\boldsymbol{CF}^{-}(\boldsymbol{\alpha}^{E},\boldsymbol{\beta}_{2}^{E_{2}}),\\
\partial_{M}:=\begin{pmatrix}\mu_{1} & 0 & 0\\
\mu_{2}(-\otimes\psi_{0}) & \mu_{1} & 0\\
\mu_{3}(-\otimes\psi_{0}\otimes\psi_{1}) & \mu_{2}(-\otimes\psi_{1}) & \mu_{1}
\end{pmatrix}
\end{gather*}
has trivial homology. Hence, 
\[
\widehat{M}:=M\otimes_{\mathbb{F}\llbracket U\rrbracket}\mathbb{F}\llbracket U\rrbracket/U,\ M^{+}:=M\otimes_{\mathbb{F}\llbracket U\rrbracket}(U^{-1}\mathbb{F}\llbracket U\rrbracket/\mathbb{F}\llbracket U\rrbracket),\ M^{\infty}:=M\otimes_{\mathbb{F}\llbracket U\rrbracket}U^{-1}\mathbb{F}\llbracket U\rrbracket
\]
also have trivial homology.
\end{cor}

\begin{proof}
We use Corollary~\ref{cor:triangle-detection-nakayama}: let $C_{i}=\boldsymbol{CF}^{-}(\boldsymbol{\alpha}^{E},\boldsymbol{\beta}_{i}^{E_{i}})$,
$f_{i}=\mu_{2}(-\otimes\psi_{i})$, and $H_{i}=\mu_{3}(-\otimes\psi_{i}\otimes\psi_{i+1})$.
Then, the $A_{\infty}$-relations imply that the conditions of Corollary~\ref{cor:triangle-detection-nakayama}
are satisfied. Hence, $M$ has trivial homology. That $\widehat{M},M^{+},M^{\infty}$
have trivial homology are formal consequences of this: for instance,
see \cite[Lemma~2.37]{nahm2025unorientedskeinexacttriangle} for $\widehat{M}$
and $M^{\infty}$. That $H(M^{+})=0$ follows from considering the
following short exact sequence of chain complexes 
\[
0\to M\to M^{\infty}\to M^{+}\to0.\qedhere
\]
\end{proof}

\subsection{\label{subsec:Proof-of-Theorem}Proof of Theorem~\ref{thm:rational-surgery-kgeneral}
assuming Theorem~\ref{thm:gen-local-comp}}

In this subsection, we use a standard neck-stretching argument to
finish the proof of Theorem~\ref{thm:rational-surgery-kgeneral}.
Roughly speaking, for each polygon that we are interested in, if the
neck is sufficiently long, then the holomorphic curve count becomes
combinatorial. One subtlety now is that there are infinitely many
polygons that we have to consider, and so it is not clear whether
we can work in a fixed almost complex structure. There are a few ways
to handle this subtlety: here, we present a proof that uses a pinched
almost complex structure \cite[Section~8]{2011.00113}; the neck-stretching
result we use is \cite[Proposition~6.5]{2201.12906} (or the special
case \cite[Proposition~10.2]{2011.00113}). Another proof uses increasingly
long necks and an approximation argument: see Subsubsection~\ref{subsec:approximation}.

Recall the Heegaard diagram $(\Sigma,\boldsymbol{\alpha},\boldsymbol{\beta}_{0},\boldsymbol{\beta}_{r},\boldsymbol{\beta}_{\infty},z)$
from Subsection~\ref{subsec:inter-local}. To simplify the notation,
denote 
\[
\boldsymbol{\gamma}_{0}^{F_{0}}:=\boldsymbol{\beta}_{0}^{E_{0}^{k}},\ \boldsymbol{\gamma}_{1}^{F_{1}}:=\boldsymbol{\beta}_{r}^{\mathbb{F}\llbracket U\rrbracket},\ \boldsymbol{\gamma}_{2}^{F_{2}}:=\boldsymbol{\beta}_{\infty}^{E_{\infty}^{k}},\ \psi_{0}:=S(\psi_{0r}^{k}),\ \psi_{1}:=S(\psi_{r\infty}^{k}),\ \psi_{2}:=S(\psi_{\infty0}^{k}),
\]
and let us focus on the minus version. Note that we will be using
Corollary~\ref{cor:triangle-ainf}, and so the statements for the
hat, plus, and infinity versions (Remark~\ref{rem:hpi}) follow as
well.

In this notation, the triangle (\ref{eq:exact-triangle34}) that we
want to show is exact is 
\begin{equation}\label{eq:exact-trianglep}
\begin{tikzcd}[ampersand replacement=\&]
	{\boldsymbol{HF}^{-}(\boldsymbol{\alpha},\boldsymbol{\gamma}_{0}^{F_0})} \&\& {\boldsymbol{HF}^{-}(\boldsymbol{\alpha},\boldsymbol{\gamma}_{1}^{F_1})} \\
	\& {\boldsymbol{HF}^{-}(\boldsymbol{\alpha},\boldsymbol{\gamma}_{2}^{F_2})}
	\arrow["{\mu_{2}(-\otimes \psi_{0})}", from=1-1, to=1-3]
	\arrow["{\mu_{2}(-\otimes \psi_{1})}"{description}, from=1-3, to=2-2]
	\arrow["{\mu_{2}(-\otimes \psi_{2})}"{description}, from=2-2, to=1-1]
\end{tikzcd}
\end{equation}Our goal is to show that $\psi_{i}\in\boldsymbol{CF}^{-}(\boldsymbol{\gamma}_{i}^{F_{i}},\boldsymbol{\gamma}_{(i+1)\bmod3}^{F_{(i+1)\bmod3}})$
are cycles for $i=0,1,2$, and that (\ref{eq:exact-trianglep}) is
exact. Below, we show that (\ref{eq:exact-trianglep}) is exact at
$\boldsymbol{HF}^{-}(\boldsymbol{\alpha},\boldsymbol{\gamma}_{1}^{F_{1}})$;
that it is exact at the two other groups follows by the same argument.

Consider the Heegaard diagram $(\Sigma,\boldsymbol{\alpha},\boldsymbol{\gamma}_{-2},\cdots,\boldsymbol{\gamma}_{3},z)$
where $\boldsymbol{\gamma}_{i}$ is a standard translate of $\boldsymbol{\gamma}_{i\bmod3}$
for $i=-2,-1,3$. Let $(F_{i},\phi_{i},A_{i}):=(F_{i\bmod3},\phi_{i\bmod3},A_{i\bmod3})$
for $i=-2,-1,3$, and let $\psi_{i}\in\boldsymbol{CF}^{-}(\boldsymbol{\gamma}_{i}^{F_{i}},\boldsymbol{\gamma}_{i+1}^{F_{i+1}})$
be the image of
\[
\psi_{i\bmod3}\in\boldsymbol{CF}^{-}(\boldsymbol{\gamma}_{i\bmod3}^{F_{i\bmod3}},\boldsymbol{\gamma}_{(i+1)\bmod3}^{F_{(i+1)\bmod3}})
\]
under the nearest point map (Subsection~\ref{subsec:Standard-translates})
for $i=-2,-1,2$. Note that the chains $\psi_{i}$ are invariant under
the chain maps given by changing the almost complex structure, since
there are no nontrivial Maslov index $0$ two-chains. Hence, it is
sufficient to show that the elements $\psi_{i}$ satisfy the conditions
of Corollary~\ref{cor:triangle-ainf} in an almost complex structure
with sufficiently long neck.

First, if we let $\Theta^{+}\in\mathbb{T}_{\boldsymbol{\gamma}_{i}}\cap\mathbb{T}_{\boldsymbol{\gamma}_{i+3}}$
be the intersection point in the top homological grading, then we
can upgrade our genus $1$ computations (Theorem~\ref{thm:gen-local-comp}~(\ref{enu:bigon})~and~(\ref{enu:quadrilateral}))
to 
\[
\mu_{1}(\psi_{i})=0\in\boldsymbol{CF}^{-}(\boldsymbol{\gamma}_{i}^{F_{i}},\boldsymbol{\gamma}_{i+1}^{F_{i+1}}),\ \mu_{3}(\psi_{i}\otimes\psi_{i+1}\otimes\psi_{i+2})={\rm Id}_{F_{i}}\Theta^{+}\in\widehat{CF}(\boldsymbol{\gamma}_{i}^{F_{i}},\boldsymbol{\gamma}_{i+3}^{F_{i+3}})
\]
in any almost complex structure with sufficiently long neck, since
there are only finitely many two-chains that can contribute.

Now, we are left to show
\begin{equation}
\mu_{2}(\psi_{i}\otimes\psi_{i+1})=0\in\boldsymbol{CF}^{-}(\boldsymbol{\gamma}_{i}^{F_{i}},\boldsymbol{\gamma}_{i+2}^{F_{i+2}}).\label{eq:mu1mu2}
\end{equation}
By \cite[Proposition~6.5]{2201.12906} (or the special case \cite[Proposition~10.2]{2011.00113}),
we can upgrade our genus $1$ computation (Theorem~\ref{thm:gen-local-comp}~(\ref{enu:triangle}))
to Equation~(\ref{eq:mu1mu2}) in a pinched almost complex structure;
let us denote the chain complexes for this pinched almost complex
structure as $\boldsymbol{CF}_{p}^{-}$.

Now, we can formally deduce Equation~(\ref{eq:mu1mu2}) (for any
almost complex structure) from it for a pinched complex structure.
Modifying the almost complex structure from the pinched almost complex
structure to an unpinched one induces maps
\begin{gather*}
{\cal F}_{1}:\boldsymbol{CF}_{p}^{-}(\boldsymbol{\gamma}_{i}^{F_{i}},\boldsymbol{\gamma}_{j}^{F_{j}})\to\boldsymbol{CF}^{-}(\boldsymbol{\gamma}_{i}^{F_{i}},\boldsymbol{\gamma}_{j}^{F_{j}}),\\
{\cal F}_{2}:\boldsymbol{CF}_{p}^{-}(\boldsymbol{\gamma}_{i}^{F_{i}},\boldsymbol{\gamma}_{i+1}^{F_{i+1}})\otimes\boldsymbol{CF}_{p}^{-}(\boldsymbol{\gamma}_{i+1}^{F_{i+1}},\boldsymbol{\gamma}_{i+2}^{F_{i+2}})\to\boldsymbol{CF}^{-}(\boldsymbol{\gamma}_{i}^{F_{i}},\boldsymbol{\gamma}_{i+2}^{F_{i+2}}),
\end{gather*}
such that ${\cal F}_{1}$ is a chain map and they satisfy the following
$A_{\infty}$-relation:
\begin{equation}
\mu_{1}({\cal F}_{2}(\psi_{i}\otimes\psi_{i+1}))+\mu_{2}(\mathcal{F}_{1}(\psi_{i})\otimes{\cal F}_{1}(\psi_{i+1}))+{\cal F}_{1}(\mu_{2}(\psi_{i}\otimes\psi_{i+1}))=0.\label{eq:ainfrelf}
\end{equation}

As we already observed, ${\cal F}_{1}(\psi_{i})=\psi_{i}$ since there
are no nontrivial Maslov index $0$ two-chains. Furthermore, the chain
complex $\boldsymbol{CF}^{-}(\boldsymbol{\gamma}_{i}^{F_{i}},\boldsymbol{\gamma}_{i+2}^{F_{i+2}})$
has trivial differential (i.e.\ $\mu_{1}\equiv0$) in any almost
complex structure since we can compute $\boldsymbol{HF}^{-}(\boldsymbol{\gamma}_{i}^{F_{i}},\boldsymbol{\gamma}_{i+2}^{F_{i+2}})$
in an almost complex structure with sufficiently long neck, and show
that it is a free $\mathbb{F}\llbracket U\rrbracket$-module and 
\[
{\rm rank}_{\mathbb{F}\llbracket U\rrbracket}\boldsymbol{HF}^{-}(\boldsymbol{\gamma}_{i}^{F_{i}},\boldsymbol{\gamma}_{i+2}^{F_{i+2}})={\rm rank}_{\mathbb{F}\llbracket U\rrbracket}\boldsymbol{CF}^{-}(\boldsymbol{\gamma}_{i}^{F_{i}},\boldsymbol{\gamma}_{i+2}^{F_{i+2}}).
\]

Now, Equation~(\ref{eq:mu1mu2}) follows from the $A_{\infty}$-relation
(\ref{eq:ainfrelf}) together with that $\mu_{1}\equiv0$ and ${\cal F}_{1}(\psi_{i})=\psi_{i}$,
and this finishes the proof of Theorems~\ref{thm:rational-surgery-k=00003D0}~and~\ref{thm:rational-surgery-kgeneral}.

\subsubsection{\label{subsec:approximation}An approximation argument}

Alternatively, we can use an approximation argument. More precisely,
consider the following definition.
\begin{defn}
If $C$ is a chain complex over $\mathbb{F}\llbracket U\rrbracket$,
then define $C_{N}:=C\otimes_{\mathbb{F}\llbracket U\rrbracket}\mathbb{F}\llbracket U\rrbracket/U^{N}$.
Similarly, let $HF_{N}(\boldsymbol{\alpha}^{E_{\boldsymbol{\alpha}}},\boldsymbol{\beta}^{E_{\boldsymbol{\beta}}})$
be the homology of $\boldsymbol{CF}^{-}(\boldsymbol{\alpha}^{E_{\boldsymbol{\alpha}}},\boldsymbol{\beta}^{E_{\boldsymbol{\beta}}})\otimes_{\mathbb{F}\llbracket U\rrbracket}\mathbb{F}\llbracket U\rrbracket/U^{N}$.

If $C$ and $D$ are chain complexes over $\mathbb{F}\llbracket U\rrbracket$
and $f:C\to D$ is an $\mathbb{F}\llbracket U\rrbracket$-linear chain
map, then define $f_{N}:=f\otimes_{\mathbb{F}\llbracket U\rrbracket}\mathbb{F}\llbracket U\rrbracket/U^{N}:C_{N}\to D_{N}$.
\end{defn}

Then, by Lemma~\ref{lem:approximation}, it is sufficient to prove
that Equation~(\ref{eq:exact-trianglep}) is exact for $HF_{N}$
(instead of $\boldsymbol{HF}^{-}$) for infinitely many $N$. Now,
to prove Equation~(\ref{eq:exact-trianglep}) is exact for $HF_{N}$,
we use the same argument as above, but work in an almost complex structure
with a sufficiently long neck instead of a pinched almost complex
structure.
\begin{lem}[Approximation]
\label{lem:approximation}Let $f:C\to D$, $g:D\to E$ be $\mathbb{F}\llbracket U\rrbracket$-linear
chain maps between finite chain complexes over $\mathbb{F}\llbracket U\rrbracket$.
Then, 
\[
H(C)\xrightarrow{H(f)}H(D)\xrightarrow{H(g)}H(E)
\]
is exact if 
\begin{equation}
H(C_{N})\xrightarrow{H(f_{N})}H(D_{N})\xrightarrow{H(g_{N})}H(E_{N})\label{eq:exactN}
\end{equation}
is exact for infinitely many $N$.
\end{lem}

\begin{proof}
First, let us show $H(g)\circ H(f)=0$. This is equivalent to showing
that for $c\in C$, $g(f(c))\in\partial E$. That (\ref{eq:exactN})
is exact implies that $g(f(c))\in\partial E+U^{N}E$. Since this holds
for infinitely many $N$, $g(f(c))\in\partial E$ by Krull's intersection
theorem (Lemma~\ref{lem:krull}).

Now, let us show ${\rm ker}H(g)\le{\rm Im}H(f)$, i.e.\ if $y\in D$
is such that $H(g)([y])=0$, then $y\in f(C)+\partial D$. That (\ref{eq:exactN})
is exact implies that $y\in f(C)+\partial D+U^{N}D$. Since this holds
for infinitely many $N$, $y\in f(C)+\partial D$ by Krull's intersection
theorem (Lemma~\ref{lem:krull}).
\end{proof}
We recall Krull's intersection theorem.
\begin{lem}[Krull's intersection theorem]
\label{lem:krull}Let $D$ be a finite module over $\mathbb{F}\llbracket U\rrbracket$,
and let $C\le D$ be a submodule. Then, the decreasing intersection
\[
\bigcap_{N\ge1}(C+U^{N}D)=C.
\]
\end{lem}

\begin{proof}
Let $\pi:D\to D/C$ be the quotient map. Then, by Krull's intersection
theorem \cite[\href{https://stacks.math.columbia.edu/tag/00IP}{Tag~00IP}]{stacks-project},
\[
\bigcap_{N\ge1}(C+U^{N}D)=\pi^{-1}\left(\bigcap_{N\ge1}U^{N}(D/C)\right)=C.\qedhere
\]
\end{proof}

\section{\label{sec:Weak-admissibility-and}Weak admissibility and positivity}

In this section, we prove that the attaching curves equipped with
local systems that appear in this paper form an $A_{\infty}$-category,
and that changing the almost complex structure induces an $A_{\infty}$-functor
(compare \cite[Appendices~A~and~B]{2308.15658}, \cite[Subsection~2.3]{nahm2025unorientedskeinexacttriangle}).
More precisely, we show that these hold under some condition on the
Heegaard diagram and the local systems, which we state in Subsection~\ref{subsec:The-condition}.

As we mentioned in Subsection~\ref{subsec:Weak-admissibility-and},
the subtlety is that we considered local systems $(E,\phi,A)$ such
that $\phi$ is not invertible in ${\rm Hom}_{\mathbb{F}\llbracket U\rrbracket}(E,E)$
and only invertible in ${\rm Hom}_{\mathbb{F}\llbracket U\rrbracket}(U^{-1}E,U^{-1}E)$.
We split the proof into two parts, weak admissibility and positivity;
roughly speaking, weak admissibility ensures that the curve counts
are finite, and positivity ensures that we never get negative powers
of $U$.

More precisely: say the attaching curve $\boldsymbol{\alpha}$ is
equipped with $(E_{\boldsymbol{\alpha}},\phi_{\boldsymbol{\alpha}},A_{\boldsymbol{\alpha}})$
and $\boldsymbol{\beta}$ is equipped with $(E_{\boldsymbol{\beta}},\phi_{\boldsymbol{\beta}},A_{\boldsymbol{\beta}})$.
For positivity, we will define an $\mathbb{F}\llbracket U\rrbracket$-submodule
(Definition~\ref{def:filtered-chain})
\[
\boldsymbol{CF}_{fil}^{-}(\boldsymbol{\alpha}^{E_{\boldsymbol{\alpha}}},\boldsymbol{\beta}^{E_{\boldsymbol{\beta}}})\le\boldsymbol{CF}^{-}(\boldsymbol{\alpha}^{E_{\boldsymbol{\alpha}}},\boldsymbol{\beta}^{E_{\boldsymbol{\beta}}}),
\]
which we call the submodule of \emph{filtered chains}, show that $(\boldsymbol{CF}_{fil}^{-}(\boldsymbol{\alpha}^{E_{\boldsymbol{\alpha}}},\boldsymbol{\beta}^{E_{\boldsymbol{\beta}}}),\partial)$
is a chain complex, and more generally, we will show the following
proposition.
\begin{prop}[Positivity]
\label{prop:keypos}For $i=0,\cdots,d$, let $\boldsymbol{\alpha}_{i}$
be an attaching curve equipped with $(E_{i},\phi_{i},A_{i})$. Assume
that these satisfy Condition~\ref{cond:c}. Let $e_{i}\in{\rm Hom}_{\mathbb{F}\llbracket U\rrbracket}(E_{i-1},E_{i})$
and ${\bf x}_{i}\in\mathbb{T}_{\boldsymbol{\alpha}_{i-1}}\cap\mathbb{T}_{\boldsymbol{\alpha}_{i}}$
for $i=1,\cdots,d$, and let ${\bf y}\in\mathbb{T}_{\boldsymbol{\alpha}_{0}}\cap\mathbb{T}_{\boldsymbol{\alpha}_{d}}$.
Let ${\cal D}\in D({\bf x}_{1},\cdots,{\bf x}_{d},{\bf y})$ be a
nonnegative two-chain, and let 
\[
g:=U^{n_{z}({\cal D})}\rho({\cal D})(e_{1}\otimes\cdots\otimes e_{d}).
\]
If $e_{i}{\bf x}_{i}\in\boldsymbol{CF}_{fil}^{-}(\boldsymbol{\alpha}_{i-1}^{E_{i-1}},\boldsymbol{\alpha}_{i}^{E_{i}})$
for $i=1,\cdots,d$, then $g{\bf y}\in\boldsymbol{CF}_{fil}^{-}(\boldsymbol{\alpha}_{0}^{E_{0}},\boldsymbol{\alpha}_{d}^{E_{d}})$.
\end{prop}

Note that in most cases, we will have 
\begin{equation}
\boldsymbol{CF}_{fil}^{-}(\boldsymbol{\alpha}^{E_{\boldsymbol{\alpha}}},\boldsymbol{\beta}^{E_{\boldsymbol{\beta}}})=\boldsymbol{CF}^{-}(\boldsymbol{\alpha}^{E_{\boldsymbol{\alpha}}},\boldsymbol{\beta}^{E_{\boldsymbol{\beta}}}).\label{eq:fileq}
\end{equation}
Indeed, Equation~(\ref{eq:fileq}) holds if $(E_{\boldsymbol{\alpha}},\phi_{\boldsymbol{\alpha}},A_{\boldsymbol{\alpha}})$
or $(E_{\boldsymbol{\beta}},\phi_{\boldsymbol{\beta}},A_{\boldsymbol{\beta}})$
is $U$-standard (Definition~\ref{def:u-standard}).
\begin{rem}[Subtlety in defining the $A_{\infty}$-category]
Recall that in Heegaard Floer homology, we do not define $\boldsymbol{CF}^{-}(\boldsymbol{\alpha}^{E_{\boldsymbol{\alpha}}},\boldsymbol{\alpha}^{E_{\boldsymbol{\alpha}}})$;
hence, to define the $A_{\infty}$-categories that we work in, we
will always have an implicit total ordering on the attaching curves
and we define 
\begin{equation}
{\rm Hom}(\boldsymbol{\alpha}^{E_{\boldsymbol{\alpha}}},\boldsymbol{\beta}^{E_{\boldsymbol{\beta}}}):=\begin{cases}
\boldsymbol{CF}_{fil}^{-}(\boldsymbol{\alpha}^{E_{\boldsymbol{\alpha}}},\boldsymbol{\beta}^{E_{\boldsymbol{\beta}}}) & {\rm if\ }\boldsymbol{\alpha}<\boldsymbol{\beta}\\
0 & {\rm otherwise}
\end{cases}.\label{eq:ainf}
\end{equation}
\end{rem}

\subsection{\label{subsec:The-condition}The condition}

First, let us state precisely the conditions on the Heegaard diagrams
and the local systems that ensure weak admissibility and positivity.
Recall that we assumed that the underlying module $E$ of a local
system $(E,\phi,A)$ is always a finite, free $\mathbb{F}\llbracket U\rrbracket$-module.
\begin{defn}
\label{def:u-standard}A local system $(E,\phi,A)$ is \emph{$U$-standard}
if there exists a finite dimensional $\mathbb{F}$-vector space $E_{\mathbb{F}}$
and an $\mathbb{F}$-linear isomorphism $\phi_{\mathbb{F}}:E_{\mathbb{F}}\to E_{\mathbb{F}}$
such that 
\[
E=E_{\mathbb{F}}\otimes_{\mathbb{F}}\mathbb{F}\llbracket U\rrbracket,\ \phi=\phi_{\mathbb{F}}\otimes_{\mathbb{F}}\mathbb{F}\llbracket U\rrbracket.
\]
In particular, $\phi$ is invertible as an element of ${\rm Hom}_{\mathbb{F}\llbracket U\rrbracket}(E,E)$.
\end{defn}

Note that the trivial local system and the local system $(E_{0},\phi_{0},A_{0})$
from Subsection~\ref{subsec:Local-systems} are $U$-standard.

\begin{figure}[h]
\begin{centering}
\includegraphics[scale=1.5]{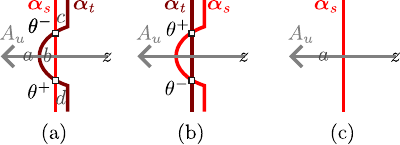}
\par\end{centering}
\caption{\label{fig:local-system-lemma}Local diagrams for the Heegaard diagram
near the oriented arc $A_{u}$. In (a), auxiliary points $a,b,c,d$
are shown; in (c), an auxiliary point $a$ is shown.}
\end{figure}

As in Subsection~\ref{subsec:Local-systems}, for $u=(u_{0},\cdots,u_{p-1})$
such that $u_{i}\in\{1,U\}$ for each $i$, define
\[
E_{u}:=\bigoplus_{i=0}^{p-1}y_{i}\mathbb{F}\llbracket U\rrbracket,\ \phi_{u}:=\sum_{i=0}^{p-1}u_{i}y_{i+1}y_{i}^{\ast}\in{\rm Hom}_{\mathbb{F}\llbracket U\rrbracket}(E_{u},E_{u}).
\]
Note that 
\[
U\phi_{u}^{-1}=\sum_{i=0}^{p-1}Uu_{i}^{-1}y_{i}y_{i+1}^{\ast}\in{\rm Hom}_{\mathbb{F}\llbracket U\rrbracket}(E_{u},E_{u}).
\]

\begin{condition}[Condition for weak admissibility and positivity]
\label{cond:c}We assume that our Heegaard diagram and the local
systems satisfy the following conditions.
\begin{enumerate}
\item All but at most two attaching curves are equipped with $U$-standard
local systems, and the others are equipped with the local system $(E_{u},\phi_{u},A_{u})$,
where (A) there exists some $i$ such that $u_{i}\neq U$, and (B)
$A_{u}$ is an oriented arc such that its initial point is the basepoint
$z$ and $A_{u}$ is disjoint from all the attaching curves that are
equipped with $U$-standard local systems.
\item If there are exactly two attaching curves (call them $\boldsymbol{\alpha}_{s}$,
$\boldsymbol{\alpha}_{t}$ with $\boldsymbol{\alpha}_{s}<\boldsymbol{\alpha}_{t}$)
equipped with $(E_{u},\phi_{u},A_{u})$, then assume that $\boldsymbol{\alpha}_{s}$
and $\boldsymbol{\alpha}_{t}$ are standard translates. Recall from
Subsection~\ref{subsec:Standard-translates} that the Heegaard diagram
looks like Figure~\ref{fig:local-system-lemma}~(a) or (b) near
the oriented arc $A_{u}$.
\item \label{enu:weakadm}Every cornerless two-chain ${\cal D}$ with $n_{z}({\cal D})=0$
has both positive and negative local multiplicities (Definition~\ref{def:weakly-admissible}).
\end{enumerate}
\end{condition}

In this paper, we only need to consider the case Figure~\ref{fig:local-system-lemma}~(a),
and so we will only write down the proof of positivity for this case.
Note that the case Figure~\ref{fig:local-system-lemma}~(b) is similar.

\subsection{\label{subsec:Admissibility}Weak admissibility}

Weak admissibility guarantees that the attaching curves equipped with
local systems form an $A_{\infty}$-category for $\boldsymbol{CF}^{\infty}:=U^{-1}\boldsymbol{CF}^{-}$,
i.e.\ (\ref{eq:ainf}) but where $\boldsymbol{CF}_{fil}^{-}$ is
replaced by $\boldsymbol{CF}^{\infty}$, and that changing the almost
complex structure induces an $A_{\infty}$-functor. These follow from
Proposition~\ref{prop:weak-admissibility}, which is the main proposition
of this subsection.

For simplicity, if $E$ is a finite, free $\mathbb{F}\llbracket U\rrbracket$-module,
denote the localization $U^{-1}E$ as $E^{\infty}$.
\begin{defn}
Let $E$ be a finite, free $\mathbb{F}\llbracket U\rrbracket$-module.
A \emph{grading data} for $E$ is a basis $b_{0},\cdots,b_{n-1}$
of $E$ over $\mathbb{F}\llbracket U\rrbracket$ together with rational
numbers $v_{0},\cdots,v_{n-1}\in\mathbb{Q}$.

Given a grading data $b_{0},\cdots,b_{n-1},v_{0},\cdots,v_{n-1}$
for $E$, a nonzero element $e\in E^{\infty}$ is \emph{homogeneous}
if there exists some $d\in\mathbb{Q}$ such that $e$ is an $\mathbb{F}$-linear
combination of $U^{d-v_{i}}b_{i}$ for $i$ such that $d-v_{i}\in\mathbb{Z}$.
In this case, write ${\rm gr}_{U}e=d$.

Let $E_{0}$ and $E_{1}$ be finite, free $\mathbb{F}\llbracket U\rrbracket$-modules
with grading data $b_{0},\cdots,b_{n-1},v_{0},\cdots,v_{n-1}$ and
$c_{0},\cdots,c_{m-1},w_{0},\cdots,w_{m-1}$, respectively. A nonzero
element $e\in{\rm Hom}_{\mathbb{F}\llbracket U\rrbracket}(E_{0}^{\infty},E_{1}^{\infty})$
is\emph{ homogeneous} if there exists some $d\in\mathbb{Q}$ such
that $e$ is an $\mathbb{F}$-linear combination of $U^{d-w_{j}+v_{i}}c_{j}b_{i}^{\ast}$
for $i,j$ such that $d-w_{j}+v_{i}\in\mathbb{Z}$. In this case,
write ${\rm gr}_{U}e=d$.
\end{defn}

Let us record the following observations.
\begin{lem}
\label{lem:obss}Let $E_{i}$ be a finite, free $\mathbb{F}\llbracket U\rrbracket$-module
with a grading data for $i=0,1,2$.
\begin{enumerate}
\item Let $e\in E_{0}^{\infty}$ or $e\in{\rm Hom}_{\mathbb{F}\llbracket U\rrbracket}(E_{0}^{\infty},E_{1}^{\infty})$.
If $e$ is homogeneous and ${\rm gr}_{U}e=d$, then for any $N\in\mathbb{Z}$,
$U^{N}e$ is also homogeneous and ${\rm gr}_{U}(U^{N}e)=N+d$.
\item Let $e_{i}\in{\rm Hom}_{\mathbb{F}\llbracket U\rrbracket}(E_{i}^{\infty},E_{i+1}^{\infty})$
be homogeneous for $i=0,1$. If $e_{1}e_{0}\in{\rm Hom}_{\mathbb{F}\llbracket U\rrbracket}(E_{0}^{\infty},E_{2}^{\infty})$
is nonzero, then it is homogeneous and $\mathrm{gr}_{U}(e_{1}e_{0})=\mathrm{gr}_{U}(e_{1})+\mathrm{gr}_{U}(e_{0})$.
\end{enumerate}
\end{lem}

\begin{proof}
Immediate.
\end{proof}
Let us choose grading data for our local systems. For $U$-standard
local systems $(E,\phi,A)$, we consider a grading data such that
$\phi$ is homogeneous and ${\rm gr}_{U}\phi={\rm gr}_{U}\phi^{-1}=0$:
indeed, identify $E=E_{\mathbb{F}}\otimes_{\mathbb{F}}\mathbb{F}\llbracket U\rrbracket$
and $\phi=\phi_{\mathbb{F}}\otimes_{\mathbb{F}}\mathbb{F}\llbracket U\rrbracket$.
Let $b_{0},\cdots,b_{n-1}$ be an $\mathbb{F}$-basis of $E_{\mathbb{F}}$;
view this also as an $\mathbb{F}\llbracket U\rrbracket$-basis of
$E$. Then, $v_{0}=\cdots=v_{n-1}=0$ works.

For the local system $(E_{u},\phi_{u},A_{u})$, let $b_{i}:=y_{i}$
and define $v_{i}$ recursively as in the proof of Lemma~\ref{lem:interpret-local-system}:
let $k$ be the number of $i$'s such that $u_{i}=U$. Then, $k\in\{0,\cdots,p-1\}$
since there exists some $i$ such that $u_{i}\neq U$. Let $v_{0}=0$
and 
\[
v_{i+1}:=\begin{cases}
v_{i}+\frac{k}{p} & {\rm if}\ u_{i}=1\\
v_{i}+\frac{k}{p}-1 & {\rm if}\ u_{i}=U
\end{cases}.
\]
Then, $\phi_{u}$ is homogeneous and ${\rm gr}_{U}\phi_{u}=k/p\in[0,1)$.
Also, $\phi_{u}$ is invertible in ${\rm Hom}_{\mathbb{F}\llbracket U\rrbracket}(E_{u}^{\infty},E_{u}^{\infty})$,
$\phi_{u}^{-1}$ is also homogeneous, and ${\rm gr}_{U}\phi_{u}^{-1}=-k/p=-{\rm gr}_{U}\phi_{u}$.

\begin{prop}[Weak admissibility]
\label{prop:weak-admissibility}For $i=0,\cdots,d$, let $\boldsymbol{\alpha}_{i}$
be an attaching curve equipped with $(E_{i},\phi_{i},A_{i})$. Assume
that these satisfy Condition~\ref{cond:c}. Let $e_{i}\in{\rm Hom}_{\mathbb{F}\llbracket U\rrbracket}(E_{i-1}^{\infty},E_{i}^{\infty})$
and ${\bf x}_{i}\in\mathbb{T}_{\boldsymbol{\alpha}_{i-1}}\cap\mathbb{T}_{\boldsymbol{\alpha}_{i}}$
for $i=1,\cdots,d$, and let ${\bf y}\in\mathbb{T}_{\boldsymbol{\alpha}_{0}}\cap\mathbb{T}_{\boldsymbol{\alpha}_{d}}$.
For each integer $N$, there are only finitely many nonnegative two-chains
${\cal D}\in D({\bf x}_{1},\cdots,{\bf x}_{d},{\bf y})$ such that
\[
U^{n_{z}({\cal D})}\rho({\cal D})(e_{1}\otimes\cdots\otimes e_{d})\notin U^{N}{\rm Hom}_{\mathbb{F}\llbracket U\rrbracket}(E_{0},E_{d}),
\]
where we consider $\rho({\cal D})(e_{1}\otimes\cdots\otimes e_{d})$
as an element of ${\rm Hom}_{\mathbb{F}\llbracket U\rrbracket}(E_{0}^{\infty},E_{d}^{\infty})$,
and consider $U^{N}{\rm Hom}_{\mathbb{F}\llbracket U\rrbracket}(E_{0},E_{d})$
as an $\mathbb{F}\llbracket U\rrbracket$-submodule of ${\rm Hom}_{\mathbb{F}\llbracket U\rrbracket}(E_{0}^{\infty},E_{d}^{\infty})$.
\end{prop}

\begin{proof}
First, since each $e_{i}$ is an $\mathbb{F}\llbracket U\rrbracket$-linear
combination of homogeneous elements, we only have to show the proposition
for the case where $e_{i}$ is homogeneous for all $i$. Let us denote
$g:=U^{n_{z}({\cal D})}\rho({\cal D})(e_{1}\otimes\cdots\otimes e_{d})$.

Let us denote the terminal point of the arc $A_{u}$ as $w$. Then,
we have the following claim.
\begin{claim}
\label{claim:degg}If $g\neq0$, then $g$ is homogeneous, and 
\begin{equation}
{\rm gr}_{U}g=({\rm gr}_{U}\phi_{u})n_{w}({\cal D})+(1-{\rm gr}_{U}\phi_{u})n_{z}({\cal D})+\sum_{i=1}^{d}{\rm gr}_{U}e_{i}.\label{eq:cdeg}
\end{equation}
\end{claim}

\begin{proof}[Proof of Claim~\ref{claim:degg}]
For notational simplicity, let $N_{i}:=\#(A_{i}\cap\partial_{\boldsymbol{\alpha}_{i}}({\cal D}))$
be the signed intersection number. Recall that 
\[
g=U^{n_{z}({\cal D})}\phi_{d}^{N_{d}}\circ e_{d}\circ\phi_{d-1}^{N_{d-1}}\circ e_{d-1}\circ\cdots\circ\phi_{1}^{N_{1}}\circ e_{1}\circ\phi_{0}^{N_{0}}.
\]
Since $\phi_{0},\cdots,\phi_{d},e_{1},\cdots,e_{d}$ are homogeneous,
by Lemma~\ref{lem:obss}, $g$ is homogeneous and 
\begin{equation}
{\rm gr}_{U}g=n_{z}({\cal D})+\sum_{i=0}^{d}N_{i}{\rm gr}_{U}\phi_{i}+\sum_{i=1}^{d}{\rm gr}_{U}e_{i}.\label{eq:c4}
\end{equation}

By Condition~\ref{cond:c}, if $(E_{i},\phi_{i},A_{i})\neq(E_{u},\phi_{u},A_{u})$,
then ${\rm gr}_{U}\phi_{i}=0$ and $A_{u}\cap\boldsymbol{\alpha}_{i}=\emptyset$.
Hence, 
\begin{equation}
\sum_{i=0}^{d}N_{i}{\rm gr}_{U}\phi_{i}=\left(\sum_{i=0}^{d}\#(A_{u}\cap\partial_{\boldsymbol{\alpha}_{i}}({\cal D}))\right){\rm gr}_{U}\phi_{u}.\label{eq:c5}
\end{equation}
By considering the ends of $A_{u}\cap{\cal D}$, we have 
\begin{equation}
\sum_{i=0}^{d}\#(A_{u}\cap\partial_{\boldsymbol{\alpha}_{i}}({\cal D}))=n_{w}({\cal D})-n_{z}({\cal D}).\label{eq:c6}
\end{equation}
Hence, we have 
\begin{align*}
{\rm gr}_{U}g & =n_{z}({\cal D})+\sum_{i=0}^{d}N_{i}{\rm gr}_{U}\phi_{i}+\sum_{i=1}^{d}{\rm gr}_{U}e_{i} & \mathrm{(\ref{eq:c4})}\\
 & =n_{z}({\cal D})+\left(\sum_{i=0}^{d}\#(A_{u}\cap\partial_{\boldsymbol{\alpha}_{i}}({\cal D}))\right){\rm gr}_{U}\phi_{u}+\sum_{i=1}^{d}{\rm gr}_{U}e_{i} & \mathrm{(\ref{eq:c5})}\\
 & =n_{z}({\cal D})+(n_{w}({\cal D})-n_{z}({\cal D})){\rm gr}_{U}\phi_{u} & \mathrm{(\ref{eq:c6})}\\
 & =({\rm gr}_{U}\phi_{u})n_{w}({\cal D})+(1-{\rm gr}_{U}\phi_{u})n_{z}({\cal D})+\sum_{i=1}^{d}{\rm gr}_{U}e_{i}. & \qedhere
\end{align*}
\end{proof}
Let $b_{i,0},\cdots,b_{i,n_{i}-1},v_{i,0},\cdots,v_{i,n_{i}-1}$ be
the grading data for $E_{i}$. If $g\neq0$, then $g$ is homogeneous
by Claim~\ref{claim:degg}. Thus, $g$ is an $\mathbb{F}$-linear
combination of $U^{{\rm gr}_{U}g-v_{d,j}+v_{0,k}}b_{d,j}b_{0,k}^{\ast}$
for some $j,k$, and~so 
\begin{equation}
g\in U^{\left\lfloor {\rm gr}_{U}g-\max(v_{d,0},\cdots,v_{d,n_{d}-1})+\min(v_{0,0},\cdots,v_{0,n_{0}-1})\right\rfloor }{\rm Hom}_{\mathbb{F}\llbracket U\rrbracket}(E_{0},E_{d}).\label{eq:deg2}
\end{equation}

By Equations~(\ref{eq:cdeg})~and~(\ref{eq:deg2}), and since $0\le{\rm gr}_{U}\phi_{u}<1$
and $n_{w}({\cal D})\ge0$, for each $N$, there exists some integer
$M$ such that if $n_{z}({\cal D})>M$, then $g\in U^{N}{\rm Hom}_{\mathbb{F}\llbracket U\rrbracket}(E_{0},E_{d})$.
The proof of \cite[Lemma 4.13]{MR2113019} shows that there are only
finitely many nonnegative two-chains ${\cal D}\in D({\bf x}_{1},\cdots,{\bf x}_{d},{\bf y})$
such that $n_{z}({\cal D})\le M$. Hence, the lemma follows.
\end{proof}

\subsection{\label{subsec:Positivity}Positivity}

In Subsection~\ref{subsec:Admissibility}, we showed that the attaching
curves equipped with local systems form an $A_{\infty}$-category
for $\boldsymbol{CF}^{\infty}$ and that changing the almost complex
structure induces an $A_{\infty}$-functor. In this subsection, we
define the submodule of \emph{filtered chains} 
\[
\boldsymbol{CF}_{fil}^{-}(\boldsymbol{\alpha}^{E_{\boldsymbol{\alpha}}},\boldsymbol{\beta}^{E_{\boldsymbol{\beta}}})\le\boldsymbol{CF}^{-}(\boldsymbol{\alpha}^{E_{\boldsymbol{\alpha}}},\boldsymbol{\beta}^{E_{\boldsymbol{\beta}}})\le\boldsymbol{CF}^{\infty}(\boldsymbol{\alpha}^{E_{\boldsymbol{\alpha}}},\boldsymbol{\beta}^{E_{\boldsymbol{\beta}}})
\]
and show Proposition~\ref{prop:keypos}, which in particular says
that the higher multiplication maps and the maps induced by changing
the almost complex structure restrict to 
\[
\boldsymbol{CF}_{fil}^{-}(\boldsymbol{\alpha}_{0}^{E_{\boldsymbol{\alpha}_{0}}},\boldsymbol{\alpha}_{1}^{E_{\boldsymbol{\alpha}_{1}}})\otimes\cdots\otimes\boldsymbol{CF}_{fil}^{-}(\boldsymbol{\alpha}_{d-1}^{E_{\boldsymbol{\alpha}_{d-1}}},\boldsymbol{\alpha}_{d}^{E_{\boldsymbol{\alpha}_{d}}})\to\boldsymbol{CF}_{fil}^{-}(\boldsymbol{\alpha}_{0}^{E_{\boldsymbol{\alpha}_{0}}},\boldsymbol{\alpha}_{d}^{E_{\boldsymbol{\alpha}_{d}}}),
\]
and hence that we can upgrade the statements for $\boldsymbol{CF}^{\infty}$
to $\boldsymbol{CF}_{fil}^{-}$.
\begin{defn}[Filtered maps]
View ${\rm Hom}_{\mathbb{F}\llbracket U\rrbracket}(E_{u},E_{u})$
as a submodule of ${\rm Hom}_{\mathbb{F}\llbracket U\rrbracket}(U^{-1}E_{u},U^{-1}E_{u})$
and $\phi_{u}^{-1}$ as an element of ${\rm Hom}_{\mathbb{F}\llbracket U\rrbracket}(U^{-1}E_{u},U^{-1}E_{u})$.
Define the space of \emph{filtered maps} ${\rm Hom}_{fil}(E_{u},E_{u})\le{\rm Hom}_{\mathbb{F}\llbracket U\rrbracket}(E_{u},E_{u})$
as follows:
\begin{enumerate}
\item If we are in case Figure~\ref{fig:local-system-lemma}~(a), let
\[
{\rm Hom}_{fil}(E_{u},E_{u}):=\left(\phi_{u}{\rm Hom}_{\mathbb{F}\llbracket U\rrbracket}(E_{u},E_{u})\phi_{u}^{-1}\right)\cap{\rm Hom}_{\mathbb{F}\llbracket U\rrbracket}(E_{u},E_{u}).
\]
\item If we are in case Figure~\ref{fig:local-system-lemma}~(b), let
\[
{\rm Hom}_{fil}(E_{u},E_{u}):=\left(\phi_{u}^{-1}{\rm Hom}_{\mathbb{F}\llbracket U\rrbracket}(E_{u},E_{u})\phi_{u}\right)\cap{\rm Hom}_{\mathbb{F}\llbracket U\rrbracket}(E_{u},E_{u}).
\]
\end{enumerate}
\end{defn}

\begin{defn}[Filtered chains]
\label{def:filtered-chain}Say $\boldsymbol{\alpha}$ is equipped
with $(E_{\boldsymbol{\alpha}},\phi_{\boldsymbol{\alpha}},A_{\boldsymbol{\alpha}})$
and $\boldsymbol{\beta}$ is equipped with $(E_{\boldsymbol{\beta}},\phi_{\boldsymbol{\beta}},A_{\boldsymbol{\beta}})$.
Let us define the submodule of \emph{filtered chains} $\boldsymbol{CF}_{fil}^{-}(\boldsymbol{\alpha}^{E_{\boldsymbol{\alpha}}},\boldsymbol{\beta}^{E_{\boldsymbol{\beta}}})\le\boldsymbol{CF}^{-}(\boldsymbol{\alpha}^{E_{\boldsymbol{\alpha}}},\boldsymbol{\beta}^{E_{\boldsymbol{\beta}}})$.
\begin{enumerate}
\item If $(E_{\boldsymbol{\alpha}},\phi_{\boldsymbol{\alpha}},A_{\boldsymbol{\alpha}})$
or $(E_{\boldsymbol{\beta}},\phi_{\boldsymbol{\beta}},A_{\boldsymbol{\beta}})$
is $U$-standard, let $\boldsymbol{CF}_{fil}^{-}(\boldsymbol{\alpha}^{E_{\boldsymbol{\alpha}}},\boldsymbol{\beta}^{E_{\boldsymbol{\beta}}}):=\boldsymbol{CF}^{-}(\boldsymbol{\alpha}^{E_{\boldsymbol{\alpha}}},\boldsymbol{\beta}^{E_{\boldsymbol{\beta}}})$.
\item Otherwise, define $\boldsymbol{CF}_{fil}^{-}(\boldsymbol{\alpha}^{E_{u}},\boldsymbol{\beta}^{E_{u}})\le\boldsymbol{CF}^{-}(\boldsymbol{\alpha}^{E_{u}},\boldsymbol{\beta}^{E_{u}})$
as the $\mathbb{F}\llbracket U\rrbracket$-linear subspace spanned
by
\begin{enumerate}
\item $e{\bf x}$ for $e\in{\rm Hom}_{fil}(E_{u},E_{u})$ and ${\bf x}\in\mathbb{T}_{\boldsymbol{\alpha}}\cap\mathbb{T}_{\boldsymbol{\beta}}$
such that $\theta^{+}\in{\bf x}$, and
\item $e{\bf x}$ for $e\in{\rm Hom}_{\mathbb{F}\llbracket U\rrbracket}(E_{u},E_{u})$
and ${\bf x}\in\mathbb{T}_{\boldsymbol{\alpha}}\cap\mathbb{T}_{\boldsymbol{\beta}}$
such that $\theta^{-}\in{\bf x}$.
\end{enumerate}
\end{enumerate}
\end{defn}

\begin{rem}
Note that $y_{i}y_{i}^{\ast}\in{\rm Hom}_{fil}(E_{u},E_{u})$. Hence,
$y_{i}y_{i}^{\ast}\Theta^{+}\in\boldsymbol{CF}_{fil}^{-}(\boldsymbol{\alpha}^{E_{u}},\boldsymbol{\beta}^{E_{u}})$
if $\boldsymbol{\alpha}$ and $\boldsymbol{\beta}$ are equipped with
$(E_{u},\phi_{u},A_{u})$.
\end{rem}

\begin{lem}
We have 
\[
U\boldsymbol{CF}^{-}(\boldsymbol{\alpha}^{E_{u}},\boldsymbol{\beta}^{E_{u}})\le\boldsymbol{CF}_{fil}^{-}(\boldsymbol{\alpha}^{E_{u}},\boldsymbol{\beta}^{E_{u}}).
\]
\end{lem}

\begin{proof}
We have 
\[
{\rm Hom}_{\mathbb{F}\llbracket U\rrbracket}(E_{u},E_{u})\le{\rm Hom}_{\mathbb{F}\llbracket U\rrbracket}(E_{u},E_{u})\phi_{u}^{-1}\ {\rm and}\ U{\rm Hom}_{\mathbb{F}\llbracket U\rrbracket}(E_{u},E_{u})\le\phi_{u}{\rm Hom}_{\mathbb{F}\llbracket U\rrbracket}(E_{u},E_{u}).\qedhere
\]
\end{proof}

\begin{proof}[Proof of Proposition~\ref{prop:keypos}]
For notational simplicity, let $N_{i}:=\#(A_{i}\cap\partial_{\boldsymbol{\alpha}_{i}}({\cal D}))$
be the signed intersection number. Recall that 
\[
g=U^{n_{z}({\cal D})}\phi_{d}^{N_{d}}\circ e_{d}\circ\phi_{d-1}^{N_{d-1}}\circ e_{d-1}\circ\cdots\circ\phi_{1}^{N_{1}}\circ e_{1}\circ\phi_{0}^{N_{0}}.
\]
Let us start with a warm-up, where there is at most one $i$ such
that $(E_{i},\phi_{i},A_{i})$ is not $U$-standard. First, if all
the local systems $(E_{i},\phi_{i},A_{i})$ are $U$-standard, then
we have $\boldsymbol{CF}_{fil}^{-}(\boldsymbol{\alpha}_{0}^{E_{0}},\boldsymbol{\alpha}_{d}^{E_{d}})=\boldsymbol{CF}^{-}(\boldsymbol{\alpha}_{0}^{E_{0}},\boldsymbol{\alpha}_{d}^{E_{d}})$
and $\phi_{i}^{N_{i}}\in{\rm Hom}_{\mathbb{F}\llbracket U\rrbracket}(E_{i},E_{i})$
for all $i$. Hence, $g\in{\rm Hom}_{\mathbb{F}\llbracket U\rrbracket}(E_{0},E_{d})$
and so $g{\bf y}\in\boldsymbol{CF}^{-}(\boldsymbol{\alpha}_{0}^{E_{0}},\boldsymbol{\alpha}_{d}^{E_{d}})$.

Consider the case where there is \textbf{exactly one} $i$ (say $i=s$)
such that $(E_{i},\phi_{i},A_{i})$ is not $U$-standard. Note that
the Heegaard diagram looks like Figure~\ref{fig:local-system-lemma}~(c)
near $A_{u}$. Since we again have $\boldsymbol{CF}_{fil}^{-}(\boldsymbol{\alpha}_{0}^{E_{0}},\boldsymbol{\alpha}_{d}^{E_{d}})=\boldsymbol{CF}^{-}(\boldsymbol{\alpha}_{0}^{E_{0}},\boldsymbol{\alpha}_{d}^{E_{d}})$
and $\phi_{i}^{N_{i}}\in{\rm Hom}_{\mathbb{F}\llbracket U\rrbracket}(E_{i},E_{i})$
for $i\neq s$, it is sufficient to show 
\begin{equation}
U^{n_{z}({\cal D})}\phi_{s}^{N_{s}}\in{\rm Hom}_{\mathbb{F}\llbracket U\rrbracket}(E_{u},E_{u}).\label{eq:wts1}
\end{equation}
First, if $N_{s}\ge0$, then Equation~(\ref{eq:wts1}) holds. If
$N_{s}<0$, then 
\[
U^{n_{z}({\cal D})}\phi_{s}^{N_{s}}=U^{n_{z}({\cal D})+N_{s}}(U\phi_{s}^{-1})^{-N_{s}},
\]
and if we let $a$ be the point on the Heegaard diagram drawn in Figure~\ref{fig:local-system-lemma}~(c),
then we have
\[
n_{z}({\cal D})+N_{s}=n_{a}({\cal D})\ge0
\]
by considering the ends of $A_{u}\cap{\cal D}$, and so Equation~(\ref{eq:wts1})
holds.

Now, let us consider the main case, where there are \textbf{exactly
two $i$} (say $i=s,t$, $s<t$) such that $(E_{i},\phi_{i},A_{i})$
is not $U$-standard. Let us consider the case where the Heegaard
diagram looks like Figure~\ref{fig:local-system-lemma}~(a) near
$A_{u}$; the other case follows similarly. Let $a,b,c,d$ be the
points on the Heegaard diagram drawn in Figure~\ref{fig:local-system-lemma}~(a).

First, by considering the ends of $A_{u}\cap{\cal D}$, we have 
\begin{equation}
N_{s}=n_{b}({\cal D})-n_{z}({\cal D}),\ N_{t}=n_{a}({\cal D})-n_{b}({\cal D}).\label{eq:nst}
\end{equation}
Furthermore, let 
\[
n_{\pm,{\rm in}}:=\begin{cases}
1 & {\rm if}\ \exists i\ {\rm such\ that}\ \theta^{\pm}\in{\bf x}_{i}\\
0 & {\rm otherwise}
\end{cases},\ n_{\pm,{\rm out}}:=\begin{cases}
1 & {\rm if}\ \theta^{\pm}\in{\bf y}\\
0 & {\rm otherwise}
\end{cases},\ n_{\pm}:=n_{\pm,{\rm in}}-n_{\pm,{\rm out}}.
\]
Then, by examining Figure~\ref{fig:local-system-lemma}~(a), we
have
\begin{gather}
n_{a}({\cal D})+n_{z}({\cal D})=n_{b}({\cal D})+n_{d}({\cal D})-n_{+},\label{eq:n+}\\
n_{a}({\cal D})+n_{z}({\cal D})=n_{b}({\cal D})+n_{c}({\cal D})+n_{-}.\label{eq:n-}
\end{gather}

We will also refer to the following casework, where we used Equation~(\ref{eq:nst}):
\begin{align}
n_{z}({\cal D})+\min(0,N_{s})+\min(0,N_{t}) & =\begin{cases}
n_{z}({\cal D}) & {\rm if}\ N_{s}\ge0,\ N_{t}\ge0\\
n_{b}({\cal D}) & {\rm if}\ N_{s}\le0,\ N_{t}\ge0\\
n_{z}({\cal D})+n_{a}({\cal D})-n_{b}({\cal D}) & {\rm if}\ N_{s}\ge0,\ N_{t}\le0\\
n_{a}({\cal D}) & {\rm if}\ N_{s}\le0,\ N_{t}\le0
\end{cases},\label{eq:00}\\
n_{z}({\cal D})+\min(0,N_{s}+1)+\min(0,N_{t}-1) & =\begin{cases}
n_{z}({\cal D}) & {\rm if}\ N_{s}\ge-1,\ N_{t}\ge1\\
n_{b}({\cal D})+1 & {\rm if}\ N_{s}\le-1,\ N_{t}\ge1\\
n_{z}({\cal D})+n_{a}({\cal D})-n_{b}({\cal D})-1 & {\rm if}\ N_{s}\ge-1,\ N_{t}\le1\\
n_{a}({\cal D}) & {\rm if}\ N_{s}\le-1,\ N_{t}\le1
\end{cases},\label{eq:1-1}\\
n_{z}({\cal D})+\min(0,N_{s}-1)+\min(0,N_{t}+1) & =\begin{cases}
n_{z}({\cal D}) & {\rm if}\ N_{s}\ge1,\ N_{t}\ge-1\\
n_{b}({\cal D})-1 & {\rm if}\ N_{s}\le1,\ N_{t}\ge-1\\
n_{z}({\cal D})+n_{a}({\cal D})-n_{b}({\cal D})+1 & {\rm if}\ N_{s}\ge1,\ N_{t}\le-1\\
n_{a}({\cal D}) & {\rm if}\ N_{s}\le1,\ N_{t}\le-1
\end{cases}.\label{eq:-11}
\end{align}

Let us divide into a few cases.

\textbf{Case $d\ge2$}: in this case, either $n_{+}=0$ or $n_{-}=0$,
and so $n_{a}({\cal D})+n_{z}({\cal D})\ge n_{b}({\cal D})$ by Equations~(\ref{eq:n+})~and~(\ref{eq:n-}).
Hence,
\[
n_{z}({\cal D})-\max(-N_{s},0)-\max(-N_{t},0)=n_{z}({\cal D})+\min(0,N_{s})+\min(0,N_{t})\ge0
\]
 by Equation~(\ref{eq:00}), and so we have
\begin{multline}
U^{n_{z}({\cal D})}\phi_{u}^{N_{t}}{\rm Hom}_{\mathbb{F}\llbracket U\rrbracket}(E_{u},E_{u})\phi_{u}^{N_{s}}\\
=U^{n_{z}({\cal D})-\max(-N_{s},0)-\max(-N_{t},0)}(U^{\max(-N_{t},0)}\phi_{u}^{N_{t}}){\rm Hom}_{\mathbb{F}\llbracket U\rrbracket}(E_{u},E_{u})(U^{\max(-N_{s},0)}\phi_{u}^{N_{s}})\\
\le{\rm Hom}_{\mathbb{F}\llbracket U\rrbracket}(E_{u},E_{u}).\label{eq:argument}
\end{multline}
Hence, we have $g\in{\rm Hom}_{\mathbb{F}\llbracket U\rrbracket}(E_{u},E_{u})$
since $\boldsymbol{CF}_{fil}^{-}(\boldsymbol{\alpha}_{i}^{E_{i}},\boldsymbol{\alpha}_{i+1}^{E_{i+1}})\le\boldsymbol{CF}^{-}(\boldsymbol{\alpha}_{i}^{E_{i}},\boldsymbol{\alpha}_{i+1}^{E_{i+1}})$.
\begin{enumerate}
\item Case $(s,t)\neq(0,d)$: $g{\bf y}$ is filtered since $\boldsymbol{CF}_{fil}^{-}(\boldsymbol{\alpha}_{0}^{E_{0}},\boldsymbol{\alpha}_{d}^{E_{d}})=\boldsymbol{CF}^{-}(\boldsymbol{\alpha}_{0}^{E_{0}},\boldsymbol{\alpha}_{d}^{E_{d}})$.
\item Case $(s,t)=(0,d)$: let us further divide into two cases.
\begin{enumerate}
\item If $\theta^{-}\in{\bf y}$, then $g{\bf y}$ is filtered since $\boldsymbol{CF}_{fil}^{-}(\boldsymbol{\alpha}_{0}^{E_{0}},\boldsymbol{\alpha}_{d}^{E_{d}})=\boldsymbol{CF}^{-}(\boldsymbol{\alpha}_{0}^{E_{0}},\boldsymbol{\alpha}_{d}^{E_{d}})$.
\item If $\theta^{+}\in{\bf y}$, then we need to show $g\in{\rm Hom}_{fil}(E_{u},E_{u})$.
By Equation~(\ref{eq:n+}), we have $n_{a}({\cal D})+n_{z}({\cal D})\ge n_{b}({\cal D})+1$,
and so by arguing similarly to (\ref{eq:argument}) but using Equation~(\ref{eq:1-1})
instead of Equation~(\ref{eq:00}), we have
\[
U^{n_{z}({\cal D})}\phi_{u}^{N_{t}}{\rm Hom}_{\mathbb{F}\llbracket U\rrbracket}(E_{u},E_{u})\phi_{u}^{N_{s}}\le\phi_{u}{\rm Hom}_{\mathbb{F}\llbracket U\rrbracket}(E_{u},E_{u})\phi_{u}^{-1},
\]
and so $g\in\phi_{u}{\rm Hom}_{\mathbb{F}\llbracket U\rrbracket}(E_{u},E_{u})\phi_{u}^{-1}$.
Hence, $g\in{\rm Hom}_{fil}(E_{u},E_{u})$.
\end{enumerate}
\end{enumerate}
\textbf{Case $d=1$, $(s,t)=(0,1)$:} we divide into four cases.
\begin{enumerate}
\item Case $\theta^{+}\in{\bf x}_{1}$, $\theta^{+}\in{\bf y}$: we have
$n_{a}({\cal D})+n_{z}({\cal D})\ge n_{b}({\cal D})$ by Equation~(\ref{eq:n+}).
We want to show 
\[
U^{n_{z}({\cal D})}\phi_{u}^{N_{t}}{\rm Hom}_{fil}(E_{u},E_{u})\phi_{u}^{N_{s}}\le{\rm Hom}_{fil}(E_{u},E_{u}).
\]
This follows from the following (which are equivalent), which are
consequences of Equation~(\ref{eq:00}).
\begin{gather*}
U^{n_{z}({\cal D})}\phi_{u}^{N_{t}}{\rm Hom}_{\mathbb{F}\llbracket U\rrbracket}(E_{u},E_{u})\phi_{u}^{N_{s}}\le{\rm Hom}_{\mathbb{F}\llbracket U\rrbracket}(E_{u},E_{u}),\\
U^{n_{z}({\cal D})}\phi_{u}^{N_{t}+1}{\rm Hom}_{\mathbb{F}\llbracket U\rrbracket}(E_{u},E_{u})\phi_{u}^{N_{s}-1}\le\phi_{u}{\rm Hom}_{\mathbb{F}\llbracket U\rrbracket}(E_{u},E_{u})\phi_{u}^{-1}.
\end{gather*}
\item Case $\theta^{+}\in{\bf x}_{1}$, $\theta^{-}\in{\bf y}$: we have
$n_{a}({\cal D})+n_{z}({\cal D})\ge n_{b}({\cal D})-1$ by Equation~(\ref{eq:n+}).
We want to show
\[
U^{n_{z}({\cal D})}\phi_{u}^{N_{t}}{\rm Hom}_{fil}(E_{u},E_{u})\phi_{u}^{N_{s}}\le{\rm Hom}_{\mathbb{F}\llbracket U\rrbracket}(E_{u},E_{u}).
\]
If $n_{z}({\cal D})+n_{a}({\cal D})\ge n_{b}({\cal D})$, then 
\[
U^{n_{z}({\cal D})}\phi_{u}^{N_{t}}{\rm Hom}_{\mathbb{F}\llbracket U\rrbracket}(E_{u},E_{u})\phi_{u}^{N_{s}}\le{\rm Hom}_{\mathbb{F}\llbracket U\rrbracket}(E_{u},E_{u})
\]
by Equation~(\ref{eq:00}). Hence, assume that $n_{z}({\cal D})+n_{a}({\cal D})-n_{b}({\cal D})<0$;
in particular $n_{b}({\cal D})\ge1$. Then, Equation~(\ref{eq:-11})
implies 
\[
U^{n_{z}({\cal D})}\phi_{u}^{N_{t}+1}{\rm Hom}_{\mathbb{F}\llbracket U\rrbracket}(E_{u},E_{u})\phi_{u}^{N_{s}-1}\le{\rm Hom}_{\mathbb{F}\llbracket U\rrbracket}(E_{u},E_{u}).
\]
\item Case $\theta^{-}\in{\bf x}_{1}$, $\theta^{+}\in{\bf y}$: we have
$n_{a}({\cal D})+n_{z}({\cal D})\ge n_{b}({\cal D})+1$ by Equation~(\ref{eq:n+}).
We want to show
\[
U^{n_{z}({\cal D})}\phi_{u}^{N_{t}}{\rm Hom}_{\mathbb{F}\llbracket U\rrbracket}(E_{u},E_{u})\phi_{u}^{N_{s}}\le{\rm Hom}_{fil}(E_{u},E_{u}).
\]
This holds since Equations~(\ref{eq:00})~and~(\ref{eq:1-1}),
respectively, imply
\begin{gather*}
U^{n_{z}({\cal D})}\phi_{u}^{N_{t}}{\rm Hom}_{\mathbb{F}\llbracket U\rrbracket}(E_{u},E_{u})\phi_{u}^{N_{s}}\le{\rm Hom}_{\mathbb{F}\llbracket U\rrbracket}(E_{u},E_{u}),\\
U^{n_{z}({\cal D})}\phi_{u}^{N_{t}}{\rm Hom}_{\mathbb{F}\llbracket U\rrbracket}(E_{u},E_{u})\phi_{u}^{N_{s}}\le\phi_{u}{\rm Hom}_{\mathbb{F}\llbracket U\rrbracket}(E_{u},E_{u})\phi_{u}^{-1}.
\end{gather*}
\item Case $\theta^{-}\in{\bf x}_{1}$, $\theta^{-}\in{\bf y}$: we have
$n_{a}({\cal D})+n_{z}({\cal D})\ge n_{b}({\cal D})$ by Equation~(\ref{eq:n+}).
We want to show
\[
U^{n_{z}({\cal D})}\phi_{u}^{N_{t}}{\rm Hom}_{\mathbb{F}\llbracket U\rrbracket}(E_{u},E_{u})\phi_{u}^{N_{s}}\le{\rm Hom}_{\mathbb{F}\llbracket U\rrbracket}(E_{u},E_{u});
\]
this follows from Equation~(\ref{eq:00}).\qedhere
\end{enumerate}
\end{proof}

\bibliographystyle{alpha}
\bibliography{/Users/gheehyun/Documents/writings/bib}

\end{document}